\documentclass[12pt]{amsart}
 \usepackage{amsmath, amsthm, amssymb, amsfonts}
\usepackage[utf8]{inputenc}
\usepackage[normalem]{ulem}
\usepackage[colorlinks = true,
            linkcolor = black,
            urlcolor  = blue,
            citecolor = red,
            anchorcolor = red,
            hidelinks, 
            pdftitle={Wall-crossing in K-theory}]{hyperref}
\usepackage{url}
\usepackage{longtable} %tables with pagebreak
\usepackage{enumitem}
\usepackage{relsize}
\usepackage{xcolor}
\usepackage{graphicx}
\usepackage{mathabx}
\usepackage{mathdots}
\usepackage{mathtools}
\usepackage{color}
\usepackage{tikz-cd,tikz}
\usepackage{braket}
\usepackage{stmaryrd}
\usepackage{amssymb}
\usepackage{tipa}
\usepackage[margin=1.1 in]{geometry}
\usepackage{appendix}
\usepackage{physics}
\usepackage{mathrsfs}  
\usepackage{longtable}

\let\underbrace\LaTeXunderbrace

\DeclareMathAlphabet{\matheulerfrak}{U}{euf}{m}{n}
\DeclareMathAlphabet{\mathpgoth}{OT1}{pgoth}{m}{n}

\usepackage{setspace}
\usetikzlibrary{arrows, decorations.markings}
\usetikzlibrary{positioning}
\tikzstyle{vecArrow} = [thick, decoration={markings,mark=at position
   1 with {\arrow[semithick]{open triangle 60}}},
   double distance=1.4pt, shorten >= 5.5pt,
   preaction = {decorate},
   postaction = {draw,line width=1.4pt, white,shorten >= 4.5pt}]
\tikzstyle{innerWhite} = [semithick, white,line width=1.4pt, shorten >= 4.5pt]
\tikzstyle{vecEq} = [thick,
   double distance=1.4pt,
   preaction = {decorate},
   postaction = {draw,line width=1.4pt, white,shorten >= 4.5pt}]
\usetikzlibrary{matrix}
\usetikzlibrary{shapes}
\usetikzlibrary{arrows,decorations.markings, tikzmark}
\usepackage{stmaryrd}

\pgfarrowsdeclare{bad to}{bad to}
{
  \pgfarrowsleftextend{-2\pgflinewidth}
  \pgfarrowsrightextend{\pgflinewidth}
}
{
  \pgfsetlinewidth{0.8\pgflinewidth}
  \pgfsetdash{}{0pt}
  \pgfsetroundcap
  \pgfsetroundjoin
  \pgfpathmoveto{\pgfpoint{-3\pgflinewidth}{4\pgflinewidth}}
  \pgfpathcurveto
  {\pgfpoint{-2.75\pgflinewidth}{2.5\pgflinewidth}}
  {\pgfpoint{0pt}{0.25\pgflinewidth}}
  {\pgfpoint{0.75\pgflinewidth}{0pt}}
  \pgfpathcurveto
  {\pgfpoint{0pt}{-0.25\pgflinewidth}}
  {\pgfpoint{-2.75\pgflinewidth}{-2.5\pgflinewidth}}
  {\pgfpoint{-3\pgflinewidth}{-4\pgflinewidth}}
  \pgfusepathqstroke
}

\theoremstyle{definition}
\newtheorem{definition}{Definition}[section]
\newtheorem{assumption}[definition]{Assumption}
\newtheorem{theorem}[definition]{Theorem}
\newtheorem{example}[definition]{Example}
\newtheorem{proposition}[definition]{Proposition}

\newtheorem{corollary}[definition]{Corollary}

\newtheorem{lemma}[definition]{Lemma}
\newtheorem{conjecture}[definition]{Conjecture}

\newtheorem{remark}[definition]{Remark}

\newtheorem{defprop}[definition]{Definition/Proposition}

\newtheorem{thm}{Theorem}

\newcommand{\sE}{\mathsf{E}}

\newcommand{\Vect}{\mathsf{Vect}}
\newcommand{\Vectb}{\textup{Vect}}

\newcommand{\BA}{{\mathbb{A}}}

\newcommand{\BC}{{\mathbb{C}}}

\newcommand{\BG}{{\mathbb{G}}}

\newcommand{\BK}{{\mathbb{K}}}

\newcommand{\BM}{{\mathbb{M}}}
\newcommand{\BN}{{\mathbb{N}}}

\newcommand{\BP}{{\mathbb{P}}}
\newcommand{\BQ}{{\mathbb{Q}}}
\newcommand{\BR}{{\mathbb{R}}}

\newcommand{\BT}{{\mathbb{T}}}

\newcommand{\BZ}{{\mathbb{Z}}}

\newcommand{\CA}{{\mathcal A}}
\newcommand{\CB}{{\mathcal B}}
\newcommand{\CC}{{\mathcal C}}
\newcommand{\CD}{{\mathcal D}}
\newcommand{\CE}{{\mathcal E}}
\newcommand{\CF}{{\mathcal F}}

\newcommand{\CM}{{\mathcal M}}

\newcommand{\CO}{{\mathcal O}}

\newcommand{\CS}{{\mathcal S}}

\newcommand{\CU}{{\mathcal U}}
\newcommand{\CV}{{\mathcal V}}

\newcommand{\CX}{{\mathcal X}}

\newcommand{\CZ}{{\mathcal Z}}

\newcommand{\1}{\mathbf 1}

\newcommand{\bfT}{{\bf T}}

\newcommand{\myfatslash}{\mathbin{\mkern-6mu\fatslash}}

\newcommand{\FP}{{\mathfrak{P}}}

\newcommand{\Fg}{{\mathfrak{g}}}

\newcommand{\FM}{{\mathfrak{M}}}
\newcommand{\FX}{{\mathfrak{X}}}
\newcommand{\FZ}{{\mathfrak{Z}}}
\newcommand{\FS}{{\mathfrak{S}}}
\newcommand{\FY}{\mathfrak{Z}}

\newcommand{\Ft}{{\mathfrak{t}}}

\newcommand{\fsl}{{\mathfrak{sl}}}

\DeclareMathOperator{\Aut}{Aut}
\DeclareMathOperator{\Sym}{Sym}
\DeclareMathOperator{\id}{id}

\DeclareMathOperator{\im}{im}

\newcommand{\Ext}{\mathcal{E}\text{xt}}

\newcommand{\lppbig}{\big(\hspace{-0.2em}\big(}
\newcommand{\rppbig}{\big)\hspace{-0.2em}\big)}
\newcommand{\lpp}{(\hspace{-0.15em}(}
\newcommand{\rpp}{)\hspace{-0.15em})}

\newcommand{\GL}{{\rm GL}}
\newcommand{\PGL}{{\rm PGL}}

\newcommand{\pt}{{\mathsf{pt}}}

\newcommand{\congpf}{\xymatrix@1@=15pt{\ar[r]^-\sim&}}

\newcommand{\sgn}{\mathrm{sgn}}

\newcommand{\ch}{\mathrm{ch}}

\newcommand{\td}{\mathrm{td}}
\newcommand{\Coh}{\mathrm{Coh}}
\newcommand{\QCoh}{\mathrm{QCoh}}

\newcommand{\Perf}{\mathsf{Perf}}
\newcommand{\Rep}{\mathrm{Rep}}
\newcommand{\Hom}{\mathrm{Hom}}

\newcommand{\DT}{\mathrm{DT}}

\newcommand{\vir}{\mathrm{vir}}

\newcommand{\sym}{\text{sym}}

\newcommand{\RHom}{\mathrm{RHom}}

\newcommand{\rk}{\mathrm{rk}}

\newcommand{\rig}{\mathrm{rig}}
\newcommand{\reg}{\mathrm{reg}}

\newcommand{\gr}{\mathrm{gr}}
\newcommand{\cl}{\mathrm{cl}}
\newcommand{\pe}{\mathrm{pe}}
\renewcommand{\ev}{\mathrm{ev}}
\newcommand{\HN}{\mathrm{HN}}
\newcommand{\isoto}{\xlongrightarrow{\raisebox{-0.3ex}[0pt][0pt]{\(\scriptstyle\sim\)}}}
\newcommand{\Grad}{\mathrm{Grad}}
\newcommand{\Filt}{\mathrm{Filt}}

\newcommand{\vcomment}[1]{{\color{red}V: #1}}

\usepackage{tikz}
\usepackage{tikz-cd}
\usetikzlibrary{decorations.pathmorphing}
\AtBeginDocument{%
	\def\MR#1{}
}

\renewcommand{\Ext}{\textup{Ext}}

\title[$K$-theoretic wall-crossing]{Generalized $K$-theoretic invariants and wall-crossing via non-abelian localization}
\date{\today}

\author[I. Karpov]{Ivan Karpov}
\address{Massachusetts Institute of Technology, Department of Mathematics}
\email{karpov57@mit.edu}

\author[M. Moreira]{Miguel Moreira}
\address{Massachusetts Institute of Technology, Department of Mathematics}
\email{miguel@mit.edu}

\begin{document}

%\subjclass[2020]{14C15, 14D22, 14F06}
% \keywords{Cohomology and Chow rings, moduli spaces of sheaves.}

\begin{abstract}
Given an abelian category and a stability condition satisfying appropriate conditions, we define generalized $K$-theoretic invariants and prove that they satisfy wall-crossing formulas. For this, we introduce a new associative algebra structure on the $K$-homology of the stack of objects of an abelian category, which we call the $K$-Hall algebra. We first define $\delta$-invariants directly coming from the stack of semistable objects and use the $K$-Hall algebra to take a formal logarithm and construct $\varepsilon$-invariants. We prove that these satisfy appropriate wall-crossing formulas using the non-abelian localization theorem. Based on work of Joyce in the cohomological setting, Liu had previously defined similar invariants assuming the existence of a framing functor; we show that when their definition of invariants makes sense it agrees with ours. Our results extend Joyce--Liu wall-crossing to non-standard hearts of $D^b(X)$, for which framing functors are not known to exist.
\end{abstract}

\maketitle

\setcounter{tocdepth}{1} 

\tableofcontents

\section{Introduction}
\label{sec: intro}

A sizable portion of enumerative geometry concerns understanding invariants arising from moduli spaces of objects in certain abelian categories $\CA$, such as the abelian category of sheaves on a smooth projective variety $X$, of representations of a quiver $Q$, or of objects in non-standard hearts of $D^b(X)$. To form good moduli spaces we require a stability condition~$\mu$ and look at the moduli space (or moduli stack) $M_\alpha^\mu$ of $\mu$-semistable objects on $\CA$ of topological type $\alpha\in C(\CA)$ -- topological type can mean, for example, the Chern character of a sheaf, or the dimension vector of a representation of a quiver. The prototypical example of a stability condition is the slope $\mu(E)=\deg(E)/\rk(E)$ of a vector bundle $E$ on a curve. There are many different invariants we may attach to the moduli space $M_\alpha^\mu$, such as:
\begin{enumerate}
    \item Motivic invariants, for example the Euler characteristic $e(M_\alpha^\mu)$ or the virtual Poincaré polynomial. 
    \item Cohomological invariants, i.e. (virtual) intersection numbers of tautological classes $D$:
    \[\int_{[M_\alpha^\mu]^\vir}D\in \BQ\,.\]
    \item $K$-theoretic invariants, i.e. (virtual) Euler characteristic of tautological $K$-theory classes $V$:
    \[\chi(M_\alpha^\mu, \CO^\vir\otimes V)\in \BQ\,.\]
\end{enumerate}

In either case, a fundamental problem is to understand how the invariants depend on the choice of stability condition $\mu$. A closely related question is to define such invariants in cases where there are strictly semistable objects. 

In the motivic setting, a very successful theory of wall-crossing has been around for 20 years, developed by Joyce \cite{JO06I, JO06II, JO06III, JO06IV}, Joyce--Song \cite{JS12} and Kontsevich--Soibelman \cite{KS08}. It plays a key role, for example, in the proof of the DT/PT correspondence \cite{todaDTPT, bridgelandDTPT}, the PT rationality conjecture \cite{todaPTrationality}, and in many other applications.

A more recent development is a proposal of a wall-crossing theory for cohomological invariants \cite{Jo17, joyce, GJT}. Liu adapted the ideas from the cohomological setting to the $K$-theory setting in \cite{Liu}. The cohomological and $K$-theoretic settings share some similarities with the motivic one, most notably the existence of a Lie algebra that is used to write down the wall-crossing formulas; indeed, the wall-crossing formulas in all 3 contexts look exactly the same, but in 3 different Lie algebras. 

However, there are also many differences. In the motivic wall-crossing setting, the invariants in the presence of strictly semistable objects can be defined from the stacks of semistable objects, which have well-defined virtual Poincaré polynomials, for example. On the other hand, there is no reasonable way to take intersection numbers on a stack, so Joyce's approach in the cohomological case is quite different. He instead uses the additional data of a framing functor to construct auxiliary moduli spaces of pairs -- which do not have strictly semistable objects --, and uses those to define ``generalized cohomological invariants''.

There are natural framing functors for sheaves or for representations of quivers. However, if we are interested in counting (weak) Bridgeland stable objects -- which is necessary in some applications, such as the PT rationality \cite{todaPTrationality, bridgelandDTPT} or the rank $r$ from rank $0/1$ theorem \cite{FTr0, FTr1} --, it is unclear how to construct framing functors, and even if one should expect them to exist. For this reason, wall-crossing formulas in such contexts were so far only conjectural.

In this paper, we develop a completely new approach to Joyce--Liu style generalized $K$-theoretic invariants and wall-crossing formulas. Our approach does not require a framing functor, and therefore we are able to prove wall-crossing formulas in greater generality, including for example (weak) Bridgeland stability conditions. What we do is much closer in spirit to the motivic approach, where invariants are extracted directly from Artin stacks of semistable objects. 

We expect that this increase in scope will have interesting applications soon. Our proof is also shorter and, arguably, conceptually cleaner -- although it requires some machinery, most notably the virtual non-abelian localization theorem of Halpern-Leistner \cite{HLremarks} -- so we hope this will make the theory more accessible and easier to extend to new contexts.

\subsection{Summary of results} Throughout, $\CA$ will be an abelian category, $\mu$ a stability condition and $\alpha$ a topological type of objects in $\CA$. We will consider the (derived) stack $\FM=\FM_\CA$ of all objects in $\CA$ and the (open) substack $\FM_\alpha^\mu\subseteq \FM$ of $\mu$-semistable objects in class $\alpha$. We dedicate Section \ref{sec: abeliancats} to explaining the technical assumptions we make on this data. In particular, we will require that $\FM_\alpha^\mu$ is finite type, quasi-smooth, and has a proper good moduli space in the sense of \cite{alper}. Often, we will rigidify these stacks to $\FM^\rig, \FM_\alpha^{\mu, \rig}$, which means removing the copy of $\BG_m$ from the automorphisms of each object. When $\mu$-semistable objects in class $\alpha$ are automatically $\mu$-stable, the rigidification $\FM_\alpha^{\mu, \rig}$ is a proper good moduli space.

\subsubsection{$\delta$-invariants}
A feature of $K$-theory which plays an essential role in this paper is that we can take Euler characteristics of perfect complexes on stacks with appropriate assumptions. In particular,
\[\chi(\FM_\alpha^{\mu, \rig}, V)\coloneqq \sum_{n\in \BZ}(-1)^n \dim H^n(\FM_\alpha^{\mu, \rig}, V)\in \BZ\]
is well-defined for any perfect complex $V$, meaning that the cohomology groups are finite dimensional and only finitely many are non-trivial. Note that the stacks considered have derived structures, so this is really a virtual Euler characteristic. Thus, the stack $\FM_\alpha^\mu$ defines a functional 
\[\delta_\alpha^\mu\colon K^\ast(\FM^\rig)\rightarrow K^\ast(\FM_\alpha^{\mu, \rig})\xrightarrow{\chi} \BZ\,,\]
where $K^\ast(-)$ denotes $K$-theory of perfect complexes and the first map is restriction along the open inclusion $\FM_\alpha^{\mu, \rig}\subseteq \FM^\rig$. For technical reasons, it is important that this functional can be upgraded to an element of $K$-homology $\delta_\alpha^\mu\in K_\ast(\FM^\rig)$, see Definition \ref{def: K-hom} and Theorem \ref{thm: deltainvariants}. The necessary material on $K$-theory and $K$-homology is covered in Section \ref{sec: Khomology}.

\subsubsection{$K$-Hall algebra and $\varepsilon$-invariants}

While the $\delta$-invariants carry information concerning $K$-theoretic invariants of the stack of semistables, they are not the ``correct'' ones. For example, they do not agree with the ones defined in \cite{Liu}. Indeed, it turns out that our $\delta$-invariants do not even live in the version of $K$-homology that Liu uses (cf. Definition \ref{def: regularKhom}, Example \ref{ex: deltanotregular}). A similar phenomena happens in the motivic setting, and we use the same solution: define $\varepsilon$-invariants from $\delta$-invariants by taking a ``logarithm''. We survey the motivic setting, and the no-pole theorem, in Section \ref{subsec: motivic}, which served as inspiration for our approach. 

It turns out that we can define a product on $K_\ast(\FM^\rig)$, which should be regarded as the analog of the multiplication in the motivic Hall algebra. 

\begin{thm}[{=Theorem \ref{thm: associativity}}]
Definition \ref{def: KHall} makes
\[\BK(\CA)\coloneqq K_\ast(\FM^\rig_\CA)_\BQ\]
an associative algebra.
\end{thm}

We use this product structure to take a formal logarithm and define
\begin{equation*}
      \label{eq: epsilonfromdeltaintro}\varepsilon_\alpha^\mu\coloneqq \sum_{k\geq 1}\frac{(-1)^{k-1}}{k}\sum_{\substack{\alpha_1+\ldots+\alpha_k=\alpha\\\mu(\alpha_i)=\mu(\alpha)}}\delta^\mu_{\alpha_1}\ast\ldots \ast\delta^\mu_{\alpha_{k}}\,.
      \end{equation*}
This sum is finite by our assumptions on the stability condition. The relation can be inverted to express $\delta$-invariants in terms of $\varepsilon$-invariants:
\begin{equation*}
  \delta_\alpha^\mu=\sum_{k\geq 1}\frac{1}{k!}\sum_{\substack{\alpha_1+\ldots+\alpha_k=\alpha\\\mu(\alpha_i)=\mu(\alpha)}}\varepsilon^\mu_{\alpha_1}\ast\ldots \ast\varepsilon^\mu_{\alpha_{k}}\,.\end{equation*}

\subsubsection{Wall-crossing formula}
\label{subsubsec: wcintro}

The $\delta$ and $\varepsilon$ invariants satisfy wall-crossing formulas that can be expressed in terms of the associative algebra structure on $\BK(\CA)$. These wall-crossing formulas are obtained as a consequence of a general theorem by Halpern-Leistner, which expresses Euler characteristics of perfect complexes in terms of the centers of a $\Theta$-stratification -- a notion that generalizes Harder--Narasimhan (HN) stratifications.

Every object $E$ of $\CA$ admits a canonical $\mu$-HN filtration 
\[0=E_0\subsetneq E_1 \subsetneq \ldots \subsetneq E_n=E\]
such that the slopes of the successive quotients are increasing. When the stack $\FM$, without any stability, is itself finite type and quasi-smooth, the wall-crossing formula can be expressed as the statement that the element
\[\sum_{\substack{\alpha=\alpha_1+\ldots+\alpha_n\\\mu(\alpha_1)>\ldots>\mu(\alpha_n)}}\delta_{\alpha_1}^\mu\ast\ldots\ast \delta_{\alpha_n}^\mu\]
in $\BK(\CA)$ does not depend on the stability condition $\mu$, and indeed it is equal to the $\delta$-invariant defined with the trivial stability condition. This is the case only in very simple examples, such as representations of quivers.

In general, we use the $\mu$-HN stratification of the stack $\FM_\alpha^{\mu_0}$ to prove a wall-crossing formula between $\mu$ and $\mu_0$, where $\mu_0$ is a stability condition on a wall and $\mu$ is in a chamber adjacent to $\mu$. 

\begin{example}
Let $\alpha$ be a topological type and $\mu_0$ be on a simple wall determined by a partition $\alpha=\alpha_1+\alpha_2$, and $\mu_-, \mu_+$ be stability conditions on the two adjacent chambers, so that
\[\mu_0(\alpha_1)=\mu_0(\alpha_2)\,,\quad  \mu_-(\alpha_1)>\mu_-(\alpha_2)\,,\quad \mu_+(\alpha_1)<\mu_+(\alpha_2)\,.\]
A $\mu_0$-semistable object $E$ has two possible types of $\mu_-$-HN filtrations: either $E$ is $\mu_-$-semistable itself, or it is an extension of the form
\[0\to E_1\to E\to E_2\to 0\]
where $E_1, E_2$ are semistable\footnote{It is part of the ``simple wall'' assumption that $\mu, \mu_-, \mu_+$ stabilities are equivalent for objects in class $\alpha_1, \alpha_2$, and moreover there are no strictly semistables in those classes. Hence \[\varepsilon_{\alpha_i}^{\mu_0}=\varepsilon_{\alpha_i}^{\mu_-}=\varepsilon_{\alpha_i}^{\mu_+}=\delta_{\alpha_i}^{\mu_0}=\delta_{\alpha_i}^{\mu_-}=\delta_{\alpha_i}^{\mu_+}\,.\]} and have topological types $\alpha_1, \alpha_2$, respectively. Thus, the stack $\FM^{\mu_0}_\alpha$ admits a $\Theta$-stratification of the form 
\[\FM^{\mu_0}_\alpha=\FM^{\mu_-}_\alpha\sqcup \FS_{\alpha_1, \alpha_2}\]
where $\FS_{\alpha_1, \alpha_2}$ is the stack of extensions as above. Non-abelian localization then says that
\[\chi(\FM^{\mu_0}_\alpha, F)=\chi(\FM^{\mu_-}_\alpha, F)+\chi(\FM^{\mu_-}_{\alpha_1}\times \FM^{\mu_-}_{\alpha_2}, F\otimes E_{\alpha_1, \alpha_2})\]
for any perfect complex $F$ on $\FM_\alpha$, where for ease of notation we omit all the restrictions of $F$. The complex $E_{\alpha_1, \alpha_2}$ is, roughly speaking, the inverse Euler class of the normal bundle to the direct sum map $\FM_{\alpha_1}\times \FM_{\alpha_2}\to \FM_\alpha$. 

The product on $\BK(\CA)$ is cooked up so that this formula can be expressed as
\[\delta_{\alpha}^{\mu_0}=\delta_{\alpha}^{\mu_-}+\delta_{\alpha_1}^{\mu_-}\ast \delta_{\alpha_2}^{\mu_-}\,.\]
This is an instance of the dominant wall-crossing formula (Theorem \ref{thm: dominantwc}). Similarly, we have 
\[\delta_{\alpha}^{\mu_0}=\delta_{\alpha}^{\mu_+}+\delta_{\alpha_2}^{\mu_+}\ast \delta_{\alpha_1}^{\mu_+}\,.\]
If we convert from $\delta$-invariants to $\varepsilon$-invariants using
\[\delta_\alpha^{\mu_\pm}=\varepsilon_\alpha^{\mu_\pm} \quad \textup{ and }\quad \delta_\alpha^{\mu_0}=\varepsilon_\alpha^{\mu_0}+\frac{1}{2}\varepsilon_{\alpha_1}^{\mu_0}\ast \varepsilon_{\alpha_2}^{\mu_0}+ \frac{1}{2}\varepsilon_{\alpha_2}^{\mu_0}\ast \varepsilon_{\alpha_1}^{\mu_0}\]
we obtain the identities
\[\varepsilon_{\alpha}^{\mu_-}+\frac{1}{2}\big[\varepsilon_{\alpha_1}^{\mu_-}, \varepsilon_{\alpha_2}^{\mu_-}\big]=\varepsilon^{\mu_0}_{\alpha}=\varepsilon_{\alpha}^{\mu_+}-\frac{1}{2}\big[\varepsilon_{\alpha_1}^{\mu_+}, \varepsilon_{\alpha_2}^{\mu_+}\big]\,,\]
where $[\phi,\psi]=\phi\ast\psi-\psi\ast \phi$ is the commutator in $\BK(\CA)$. It is a general feature that wall-crossing formulas between $\varepsilon$-invariants can be written using exclusively commutators.
\end{example}

\begin{thm}[=Theorem \ref{thm: generalwc}]\label{thm: B}
   If $\mu, \mu'$ can be connected by a continuous path crossing finitely many walls then we have the wall-crossing formulas
    \begin{align*}\delta_{\alpha}^{\mu'}&=\sum_{\alpha_1+\ldots+\alpha_l=\alpha}S(\alpha_1, \ldots, \alpha_n; \mu, \mu')\cdot \delta_{\alpha_1}^{\mu}\ast \ldots\ast \delta_{\alpha_n}^{\mu}\\
    \varepsilon_{\alpha}^{\mu'}
    &\nonumber=\sum_{\alpha_1+\ldots+\alpha_l=\alpha}\tilde U(\alpha_1, \ldots, \alpha_n; \mu, \mu')\cdot [[\ldots [\varepsilon_{\alpha_1}^{\mu}, \varepsilon_{\alpha_2}^\mu], \ldots, ],\varepsilon_{\alpha_n}^{\mu}]
    \end{align*}
    in $\BK(\CA)$, where $S(-), \tilde U(-)\in \BQ$ are explicit combinatorial coefficients (cf. \cite[Section 4]{JO06IV}).
\end{thm}

A particular setting where our theorem establishes a wall-crossing formula that was not previously proven is that of tilt stability:

\begin{corollary}
    If $X$ is a surface with nef anticanonical or a Fano 3-fold satisfying the BMT inequality then we have generalized $K$-theoretic invariants of the moduli of tilt semistable objects and these satisfy the wall-crossing formulas in Theorem~\ref{thm: B}.
\end{corollary}
We refer to Sections \ref{subsec: tild} and \ref{subsec: changinghearts} for details on this application.

\subsubsection{Comparison with Joyce--Liu invariants}

Sections \ref{sec: framing}, \ref{sec:Liuvert} and \ref{sec: Ktocoh} are dedicated to comparing our construction with the one by Liu \cite{Liu} and Joyce \cite{joyce}. When the moduli stack $\FM_\alpha^\mu$ contains strictly semistable objects, they define invariants using the auxiliary data of a framing functor $\Phi$ (cf. Definition \ref{def: framing}). A framing functor can be used to construct a moduli space of Joyce--Song pairs $P_\alpha^{\textup{JS}}$, which no longer contains strictly semistable objects. The main result of Section \ref{sec: framing} expresses the relation between invariants defined with Joyce--Song pairs and our $\varepsilon$-invariants:

\begin{thm}[=Theorem \ref{thm: comparisonJL}]
Let $\Phi$ be a framing functor and $P_\alpha^{\mu, \textup{JS}}$ the corresponding moduli of Joyce--Song pairs. We have
    \begin{equation*}
        \Pi_\ast\big(\Lambda_{-1}(\BT_\Pi^\vee)\cap [P_\alpha^{ \textup{JS}}]^\vir\big)=\sum_{\substack{\alpha_{1}+\ldots+\alpha_n=\alpha\\
        \mu(\alpha_i)=\mu(\alpha)}}\frac{(-1)^{n-1}}{n!}\lambda(\alpha_1)\big[\big[\ldots\big[\varepsilon_{\alpha_1}^\mu,\varepsilon_{\alpha_2}^\mu\big],\ldots, \varepsilon_{\alpha_n}^\mu\big]\,. 
    \end{equation*}
\end{thm}
This identity is taken as the definition of $K$-theoretic invariants in \cite{Liu}, following Joyce's cohomological version, except for the fact that Liu uses a different Lie algebra, denoted in this paper by $\widecheck K_\ast^\reg(\FM)_\BQ$ -- his Lie algebra is induced by the structure of a multiplicative vertex algebra on $K_\ast^\reg(\FM)_\BQ$. To establish that the invariants are the same, we compare these two Lie algebras in Section \ref{sec:Liuvert}.

\begin{thm}[=Theorem \ref{thm: liealgebrahom}, Proposition \ref{prop: nopolesknown}(3)]
    The natural map
    \[\widecheck K_\ast^\reg(\FM)_\BQ\to \BK(\CA)\]
    is a homomorphism of Lie algebras. Moreover, when there is a framing functor this natural map sends Liu's classes to $\varepsilon$ classes. 
\end{thm}

An important point here is that the Lie algebra $\widecheck K_\ast^\reg(\FM)_\BQ$ is not obtained as the commutator of some naturally defined product. Indeed, its image in $\BK(\CA)$ is far from being an associative subalgebra, despite being a Lie subalgebra. Moreover, $\widecheck K_\ast^\reg(\FM)_\BQ$ does not typically contain $\delta$-classes. See the discussion in Section \ref{subsec: motivic} for an analogous situation with $\BM_\reg(\CA)$ and $\BM(\CA)$. 

Finally, in Section \ref{sec: Ktocoh} we relate our invariants to Joyce's cohomological invariants. To state our comparison, we introduce the notion of homological lift, cf. Definition \ref{def: homlift}. Roughly speaking, a class $A\in H_\ast(\FX)$ is a homological lift of $\phi\in K_\ast(\FX)$ if they are related by a Riemann--Roch type formula. 

\begin{thm}[=Theorem \ref{thm: homliftsjoyce}]
If there exists a framing functor, Joyce's classes in $H_\ast(\FM_\alpha^\rig)$ are homological lifts of $\varepsilon$ classes. 
\end{thm}

We expect canonical homological lifts of $\varepsilon$ classes to exist in general, without assuming the existence of a framing functor, and we hope to address this question in the future. In Section \ref{subsec: homliftdescendents} we explain that our wall-crossing formulas for $\varepsilon$ invariants implies Joyce's wall-crossing formulas for the homological lifts, at least if we consider only intersection numbers of (algebraic) tautological classes. In Appendix \ref{appendix} we discuss a version of the entire theory where we replace algebraic $K$-theory by other additive invariants of dg categories, such as Hochschild homology or Blanc's topological $K$-theory. Topological $K$-theory, in particular, is closer to cohomology and allows us to deduce wall-crossing formulas without algebraicity assumptions.

\subsection{Analogy with motivic wall-crossing}
\label{subsec: motivic}

Part of the motivation for the present article is to make the cohomological or $K$-theoretic wall-crossing \cite{joyce, Liu} closer to the approach to motivic invariants \cite{JO06III, JO06I, JS12, KS08, bridgelandHallsurvey}. We will now describe some aspects of the motivic story to make this comparison apparent. For simplicity, we restrict ourselves to hereditary categories $\CA$, such as representations of a quiver or moduli of bundles on a curve, and to standard Euler characteristic (i.e. ``naive DT invariants''), rather than Behrend weighted Euler characteristics (``genuine DT invariants''). We also work only in the quantum torus, and not in the motivic Hall algebra.

Let $\mu$ be a stability condition on $\CA$ and $\alpha$ a topological type. When $\mu$-stability and $\mu$-semistablity coincide, we are interested in motivic invariants of the good moduli space $M_\alpha^\mu=\FM_\alpha^{\mu, \rig}$, such as its Euler characteristic $e(M_\alpha^\mu)$ or its virtual Poincaré polynomial $P_q(M_\alpha^\mu)\in \BQ[q^{\frac{1}{2}}]$, which is the unique motivic invariant which agrees with the usual Poincaré polynomial
\[P_q(X)=\sum_{i}(-q^{1/2})^i\dim H^i(X)\]
when $X$ is smooth and proper. Being motivic means that it satisfies
\[P_q(X\times Y)=P_q(X)P_q(Y) \textup{ and }P_q(X)=P_q(U)+P_q(X\setminus U)\]
for any $X,Y$ and $U\subseteq X$ open. The virtual Poincaré polynomial can be constructed for a general variety $X$ using the weight filtration on its cohomology \cite{DKvirtualpoincare}. The specialization $q^{1/2}=1$ recover the Euler characteristic.

To define invariants in the presence of strictly semistable objects, one might try to use directly the stack $\FM_\alpha^\mu$ or its rigidification $\FM_\alpha^{\mu, \rig}$. Indeed, it is explained in \cite[Section 4.2]{JoMotivic} that there is a unique way to extend the virtual Poincaré polynomial to finite type stacks with affine stabilizers by requiring that $P_q([X/G])=P_q(X)/P_q(G)$ for special groups $G$ (in particular, $G=\GL_n$ is special); for such stacks $\CX$, 
\[P_q(\CX)\in \BZ[q^{\pm 1/2}, (1-q^i)^{-1}]\,.\]
Since $P_q(\FM_\alpha^{\mu, \rig})$ has a pole at 1 when stable is not equal to semistable, we cannot set $q^{1/2}=1$ and obtain the Euler characteristic. One way to obtain something like an Euler characteristic is to first take a ``logarithm'' on the motivic Hall algebra, or on the quantum torus. 

The quantum torus of $\CA$ is defined to be the associative $\BQ[q^{\pm 1/2}, (1-q^i)^{-1}]$-algebra
\[\BM(\CA)\coloneqq \bigoplus_{\alpha\in C(\CA)}\BQ[q^{\pm 1/2}, (1-q^i)^{-1}]\cdot e^\alpha\]
with product given by\footnote{Usually there is no factor of $1/(q-1)$ and $\delta$ invariants below are defined using the non rigidified stack, but to make the analogy with the $K$-theory setting more clear we are making this small modification. The usual $\delta, \varepsilon$ invariants are $\frac{\delta_\alpha^\mu}{q-1}, \frac{\varepsilon_\alpha^\mu}{q-1}$ in our convention.} 
\[e^\alpha\ast e^\beta=\frac{q^{-\chi(\beta, \alpha)}}{q-1}e^{\alpha+\beta}\,.\]
The virtual Poincaré polynomial of the (rigidified) stack defines an element
\[\delta_\alpha^\mu\coloneqq P_q(\FM_{\alpha}^{\mu, \rig})\cdot e^\alpha\,.\]
Then, we modify these classes by taking a logarithm in the quantum torus, i.e.
\[\varepsilon_\alpha^\mu=\sum_{\substack{\alpha=\alpha_1+\ldots+\alpha_n\\\mu(\alpha_i)=\mu(\alpha)}}\frac{(-1)^{n-1}}{n}\delta_{\alpha_1}^\mu\ast\ldots\ast \delta_{\alpha_n}^\mu\,.\]

A remarkable theorem of Joyce \cite[Theorem 8.7]{JO06III} is that the $\varepsilon$ invariants no longer have a pole at $q=1$ -- see also \cite{intrinsicDT1} and \cite{BRinertia} for more recent approaches to the proof of this result. More precisely, they are regular elements, i.e. lie in
\[\BM_\reg(\CA)\coloneqq \bigoplus_{\alpha\in C(\CA)}\BQ[q^{\pm 1/2}, (1+q+\ldots+q^{i-1})^{-1}]\cdot e^\alpha\subseteq \BM(\CA)\,.\]
In particular, it now makes sense to specialize $q^{1/2}=1$. This produces the definition of generalized (naive) DT invariants:
\[\varepsilon_\alpha^\mu|_{q=1}=\DT_\alpha^\mu\cdot e^\alpha\,.\]

\begin{example}
    Consider $\CA=\mathsf{Vect}$ the abelian category of finite dimensional vector spaces with the trivial stability condition. Then 
    \[\delta_n=P_q(B\PGL_n)\cdot e^n=q^{-n+1}\prod_{n=1}^{n-1} (q^n-q^i)^{-1}\cdot e^n\,.\]
    Then we have, for example,
    \[\varepsilon_2=\delta_2-\frac{1}{2}\varepsilon_1\ast \varepsilon_1=\left(\frac{1}{q(q^2-1)}-\frac{1}{2q(q-1)}\right)\cdot e^2=-\frac{1}{2q(q+1)}\cdot e^2\,.\]
    In particular, $\DT_2=-\frac{1}{4}$. More generally, $\DT_n=\frac{(-1)^{n-1}}{n^2}$ \cite[Example 6.2]{JS12}.
\end{example}

The regular elements $\BM_\reg(\CA)\subseteq \BM(\CA)$ do not form a subalgebra of $\BM(\CA)$. However, they form a Lie subalgebra of $\BM(\CA)$ with the commutator $[u,v]=u\ast v-v\ast u$. This is clear from the observation that
\[[e^\alpha, e^\beta]=\frac{q^{-\chi(\beta, \alpha)}-q^{-\chi(\alpha, \beta)}}{q-1}\cdot e^{\alpha+\beta}\]
is a polynomial in $q$. 

Wall-crossing formulas can be efficiently written in $\BM(\CA)$. Indeed, they take exactly the form discussed in Section \ref{subsubsec: wcintro}, and served as inspiration for our construction of $\BK(\CA)$. For example when $\FM_\alpha$ itself is finite type
\[\sum_{\substack{\alpha=\alpha_1+\ldots+\alpha_n\\\mu(\alpha_1)>\ldots>\mu(\alpha_n)}}\delta_{\alpha_1}^\mu\ast\ldots\ast \delta_{\alpha_n}^\mu=P_q(\FM_\alpha)\cdot e^\alpha\]
does not depend on $\mu$. Indeed, this is a consequence of the Harder--Narasimhan stratification, as the product $\ast$ is built so that the Harder--Narasimhan stratification gives this formula. This provides a way to compare classes $\delta^{\mu}_\alpha, \delta_\alpha^{\mu'}$ for two stability conditions $\mu, \mu'$. Joyce shows \cite[Theorem 5.4]{JO06IV} that when we replace $\delta$ classes by $\varepsilon$ classes, the resulting formula can be written entirely using the commutator on $\BM(\CA)$; in other words, it can be written entirely inside the Lie subalgebra $\BM_\reg(\CA)$.

\subsection{Conjectures and future directions}In certain aspects, we have opted to keep a slightly simpler (and hence less general) setup than in \cite{joyce}. For example, we do not consider here the possibility for reduced obstruction theories, necessary for wall-crossing on surfaces with $p_g>0$. We have also opted to do everything non-equivariantly, unlike \cite{Liu}. It should be possible to incorporate both of these aspects in our approach without too much trouble. 

Recently, \cite{intrinsicDT1, intrinsicDT2} established a generalization of the motivic DT theory to stacks that do not come from linear moduli problems. This is done by replacing the motivic Hall algebra by a structure they call ``motivic Hall induction'', intrinsically associated to any stack (under reasonable conditions). A similar extension of the content of this paper using their formalism is possible, and will be pursued elsewhere. This should, in particular, answer questions raised, e.g., by Bu in the setting of wall-crossing for self-dual categories: ~\cite{Bu1}.

A more ambitious extension is to the setting of sheaves on Calabi--Yau 4-folds, or more generally $(-2)$-shifted symplectic stacks. A mild modification of our $K$-Hall algebra allows a definition of generalized $K$-theoretic invariants in that context. The missing piece is an appropriate version of the non-abelian localization theorem for $(-2)$-shifted symplectic stacks.

We make Conjectures \ref{conj: nopole} and \ref{conj: homlift}, which are similar in spirit to the no-pole theorem in the motivic setting. The first says that the classes $\varepsilon_\alpha^\mu$ more or less comes from a scheme, and the second says that $\varepsilon_\alpha^\mu$ admits a homological counterpart that lives in algebraic degree equal to the virtual dimension. We show that both of these hold when there are framings, but it would be desirable to have a more general argument. We hope these conjectures might be addressed by relating $\varepsilon$ classes to the (derived) Kirwan desingularization of the moduli stack, cf. \cite{HRS}.

The more straightforward applications are the ones to explicit wall-crossing problems which still remain unsolved at the time of our writing the present paper. First of all, it seems that with little extra care the PT-rationality for Fano threefolds (i.e., the Fano analogue of~\cite{todaPTrationality}) may be obtained. Second, in some cases, we are able to apply the present results to imitate Feyzbakhsh-Thomas programme (cf., e.g., \cite{FTr0}) for Fano $3$-folds, hence, answering (at least, some particular manifestation of) the corresponding Joyce's conjecture~\cite[p.14]{Joycepres}. Both questions will be answered in two upcoming works of ours in the nearest future.

Let us finally mention the following curious question (as usual,  due to Dominic Joyce and his collaborators): since, as we explain, the present paper develops precisely the machinery required for the wall-crossing in the derived categories, the Remark 4.3.2 (c) in~\cite{GJT} hereby starts making sense; in a future, we would like to investigate whether, as \textit{loc.~cit.} predicts, one can, thus, introduce ``Joyce structures with descendants'', say, for projective surfaces.

\subsection{Notation and conventions}\label{subsec: notation}
Every scheme, algebraic space, stack and derived stack in this paper is considered over $\BC$. By (derived) stack we mean a (derived) 1-Artin stack; all the (derived) stacks in this paper will be locally of finite type, have affine diagonal, and a perfect cotangent complex. Higher (derived) stacks are in the sense of higher Artin stacks of \cite{higherstackstoen}. We will typically denote stacks with calligraphic font (e.g. $\CX, \CZ, \CM$) and derived stacks with Fraktur font (e.g. $\FX, \FZ, \FM$).

By (co)homology of a derived stack $\FX$ we mean the (co)homology with $\BQ$-coefficients of the topological realization (cf. \cite[Section 3]{B16}) of the classical truncation $\CX$; this agrees with a 6-functor formalism type definition by \cite[Proposition 2.8]{khanequivcoh}. Homology is considered a direct sum over its degrees and cohomology a direct product, 
\[H_\ast(-)=\bigoplus_{n\geq 0}H_n(-)\textup{ and }H^\ast(-)=\prod_{n\geq 0}H^n(-)\,.\]
Chern classes of perfect complexes are defined by pulling back universal Chern classes along the topological realization of the induced map $\CX\to \FM_{\Perf}$ to the stack of perfect complexes. Given an abelian group $A$ we write $A_\BQ=A\otimes_\BZ \BQ$.

\subsection{Acknowledgments} We would like to thank Younghan Bae, Chenjing Bu, Pavel Etingof, Daniel Halpern-Leistner, Andrés Ibáñes Nuñéz, Dominic Joyce, Adeel Khan, Vasily Krylov,  Woonam Lim, Henry Liu, and Davesh Maulik for conversations related to this project.

The first author wishes to especially thank Davesh Maulik: the present work owes its existence to his suggestion of considering the wall-crossing applications of non-abelian localization.

\section{Preliminaries: abelian categories and stability conditions}
\label{sec: abeliancats}

\subsection{Abelian categories and moduli stacks}\label{subsec: abcat}

Let $\CA$ be a $\BC$-linear abelian category. Under appropriate conditions, it is possible to form a derived moduli stack $\FM=\FM_{\CA}$ over $\BC$ parametrizing objects of $\CA$. This stack will play a central role in this paper and we need it to have certain structures and properties, so we will now explain the necessary assumptions on the abelian category $\CA$.

The standard example to keep in mind of a good abelian category is $\CA=\Coh(X)$ for a smooth projective scheme $X$. Other examples include the categories of finite dimensional vector spaces and finite dimensional representations of an acyclic quiver. Other hearts of $t$-structures on $D^b(X)$ are often good abelian categories (cf. \cite[Example 7.22, 7.26]{AHLH}). We summarize in the next definition/proposition the properties that we require and discuss them in more detail below.

\begin{defprop}\label{defprop: goodabelian}
    We say that $\CA$ is a \emph{good abelian category} if Assumptions \ref{ass: good1} and \ref{ass: good2} hold. When that is the case, there is a derived moduli stack $\FM_\CA$ parametrizing objects of $\CA$. The derived stack $\FM_{\CA}$ is locally finite type, has affine diagonal and its derived tangent complex is given by \eqref{eq: derivedtangentFM}. The derived stack comes with a direct sum map $\Sigma$, a $B\BG_m$ action and a complex $\Ext_{12}$ on $\FM\times \FM$ with compatibilities detailed in \cite[Assumption 4.4 (c), (f)]{joyce}.
\end{defprop}

As discussed in \cite[Section 7.1]{AHLH} following \cite{AZ}, to construct a classical stack $\CM_\CA$ parametrizing objects in $\CA$ we need $\CA$ to be the subcategory of compact objects of some locally Noetherian cocomplete category $\CC$ (for example, $\Coh(X)$ are the compact objects of $\QCoh(X)$ if $X$ is smooth projective). 

\begin{assumption}\label{ass: good1}
    The abelian category $\CA$ is Noetherian and there is a cocomplete abelian category $\CC$ such that $\CA=\CC^{\textup{pe}}$ is the subcategory of perfect objects, and moreover $\CA$ generates $\CC$.
\end{assumption}

In particular, $\CC$ is locally Noetherian and there exists a classical moduli stack $\CM=\CM_\CA$ as in \cite[Section 7.1]{AHLH}. Note that this stack actually depends on the choice of embedding into the cocomplete category $\CC$, as remarked in \cite[Warning 7.11]{AHLH}, but we will suppress it from the notation since in our examples there is always a natural choice of $\CC$.

We would like to have a natural derived enhancement of $\CM$, which we denote by $\FM=\FM_\CA$, whose derived tangent bundle is given by
\begin{equation}\label{eq: derivedtangentFM}
    \BT_{\FM}=\Delta^\ast \Ext_{12}[1]\,.\end{equation}
This means that $\Hom_\CA(F, F)$ controls the stackiness of $\FM$ at $F$, $\Ext^1_\CA(F, F)$ the deformations and $\Ext^{>1}_\CA(F, F)$ the (higher) obstructions. 

Toën-Vaquié \cite{TV07} construct higher derived stacks parametrizing objects in saturated dg categories. See Definition 2.4 in loc. cit. for the definition of saturated; when $X$ is a smooth proper scheme, the dg enhancement of $D^b(X)$ (which is unique by \cite[Corollary 7.2]{CSdgenhancement}) is saturated \cite[Lemma 3.27]{TV07}. Hence, we require the following:
\begin{assumption}\label{ass: good2}
    There is a saturated dg category $\bfT$ such that $\CA$ is a full subcategory of the homotopy category $\textup{Ho}(\bfT)$ and $\CM_\CA$ is an open substack of the classical truncation of $\FM_{\bfT}$.
\end{assumption}
This open inclusion induces a derived enhancement $\FM_\CA$ of $\CM_\CA$, which is open inside $\FM_{\bfT}$; cf. \cite[Proposition 2.1]{STV}. When $\CA=\Coh(X)$, or some other heart of $D^b(\Coh(X))$, for a projective smooth variety $X$, we let $\bfT$ be the dg enhancement of $D^b(\Coh(X))$. The openness of $\CM_{\Coh(X)}$ is shown for example in \cite[Proposition 6.2.7]{HLstructure}; other natural hearts related to stability conditions are also shown to be open in\footnote{In all of these, the results are stated as $\CM_\CA$ being open inside Lieblich's \cite{lieblich} stack of complexes with $\Ext^{<0}(F, F)=0$. This is an open substack of $\CM_{D^b(X)}$, as pointed out in \cite[Corollary 3.21]{TV07}.} \cite[Section 4]{PiyToda}, \cite[Proposition 6.2.7]{HLstructure} and \cite[Proposition 4.6]{BCR}. It is shown in \cite[Theorem 0.2]{TV07} that $\FM_{\bfT}$ is locally of finite type, and hence the same is true for $\FM_\CA$. By \cite[Lemma 7.20]{AHLH} it follows that $\FM_\CA$ has affine diagonal. The derived stack $\FM_\CA$ comes with the following structures:
\begin{enumerate}
    \item A direct sum morphism $\Sigma\colon \FM\times \FM\to \FM$ that acts on points as $(F, G)\mapsto F\oplus G$.
    \item A $B\BG_m$ action $\Psi\colon B\BG_m\times \FM\to \FM$ induced by the scaling action $\BG_m\rightarrow \Aut_{\CA}(F)$.
    \item A perfect complex $\Ext_{12}$ on $\FM\times \FM$ whose value over a point $(F_1, F_2)$ on $\FM\times \FM$ is $\RHom_\CA(F_1, F_2)$. We also denote by $\Ext_{21}$ the complex $\RHom_\CA(F_2, F_1)$, which is obtained by pulling back $\Ext_{12}$ along the map $\FM\times \FM\to \FM\times \FM$ which permutes the two factors.
\end{enumerate}
These can all be defined on $\FM_\bfT$ (see, for example, \cite[Proposition 8.29]{KPS}) and restricted to $\FM_\CA$. They satisfy a series of compatibilities which are described in detail in \cite[Assumption 4.4 (c), (f)]{joyce}. For example, the complex $\Ext_{12}$ should be ``bilinear'' with respect to $\Sigma$, in the sense that
\[\RHom_\CA(F_1, F_2\oplus F_2')=\RHom_\CA(F_1, F_2)\oplus \RHom_\CA(F_1, F_2)\,.\]
The complex $\Ext_{12}$ has weight $-1$ with respect to the $B\BG_m$ action on the first coordinate and weight $1$ with respect to the action on the second coordinate.
When $\CA=\Coh(X)$ (or more generally $\CA$ is the heart of a t-structure in $D^b(X)$) the Ext complex is given by
\[\Ext_{12}=Rp_\ast\big(\mathcal{RH}om(\CF_{13}, \CF_{23})\big)\]
where $\CF$ is the universal bundle (complex) on $\FM\times X$, $\CF_{13}, \CF_{23}$ are the pullbacks of $\CF$ to $\FM\times \FM\times X$ indicated by the subscript, and $p$ is the projection $\FM\times \FM\times X\to \FM\times \FM$. See for example \cite[(5.4)]{GJT} for the explicit form in the case of quivers.

We denote by $C(\CA)$ the set of connected components of $\FM_\CA$, or equivalently $\CM_\CA$, and we let $\FM_\alpha\subseteq \FM_\CA$ be the connected component corresponding to $\alpha\in C(\CA)$, so that
\[\FM_\CA=\bigsqcup_{\alpha\in C(\CA)}\FM_\alpha\,.\]
The direct sum map on $\FM_\CA$ gives $C(\CA)$ the structure of a monoid. Since the rank of $\Ext_{12}$ is locally constant, we define the Euler form
\[\chi(\alpha, \beta)=\rk\big((\Ext_{12})_{|\FM_{\alpha\times \FM_\beta}}\big)\,.\]
Given an object $E$ of $\CA$, we denote by $[E]\in C(\CA)$ the connected component of the corresponding $\BC$-point in the stack $\FM_\alpha$.

\begin{remark}
    If $\CA$ is a good abelian category of homological dimension $\leq 2$, meaning that $\Ext^{>2}_\CA=0$, then $\FM_{\CA}$ is a quasi-smooth derived stack, due to \eqref{eq: derivedtangentFM}.
\end{remark}

Also, using the $B\mathbb G_m$ action, one may define the \textit{rigidified} (or the \textit{projective linear}, to use the terminology from~\cite{joyce}) version $\FM^{\rig}$ of $\FM$.

\begin{definition}\label{def: rigstack} Suppose that $\FX$ is a (derived) stack admitting a free $B\BG_m$ action. Then, there is a \textit{rigidified} (derived) stack $\FX^\rig=\FX\myfatslash \BG_m$ (cf. \cite[Appendix A]{AOV}); it comes with a canonical map $\pi\colon \FX\to \FX^\rig$ which is a $\BG_m$ gerbe -- in particular, its fibers are isomorphic to $B\BG_m$. Roughly speaking, $\FX^\rig$ is the stack with the same points as $\FX$ and with isotropy groups obtained by quotienting out the $\BG_m$ coming from the action. The rigidification sits in the following homotopy cartesian square:
\begin{center}
    \begin{tikzcd}
B\BG_m\times \FX\arrow[r, "\Psi"]\arrow[d, "p_2"]& \FX\arrow[d, "\pi"]\\
\FX\arrow[r, "\pi"]&\FX^\rig
    \end{tikzcd}
\end{center}
Every complex in $\FX$ pulled back from $\FX^\rig$ has weight 0, by commutativity of the square above. 
\end{definition}

\subsection{Stability conditions, $\Theta$-stratifications and good moduli spaces}
\label{subsec: stabconditions}

To make sense of taking integrals (i.e, Euler characteristics) of perfect complexes over some stack, one usually requires the existence of a proper \textit{good moduli space}. This statement is formalized as Theorem \ref{thm: deltainvariants} below.

\begin{definition}[\cite{alper}] Let $\CX$ be an Artin stack. A good moduli space for $\CX$ is an algebraic space $X$ together with a morphism $\phi: \CX \to X$ such that
\begin{enumerate}
\item $\phi$ is quasi-compact;
\item the pushforward functor $\phi_*\colon \QCoh(\CX) \to \QCoh(X)$ is exact on the categories of the quasi-coherent sheaves;
\item the natural morphism $\mathcal O_{X} \to \phi_*\mathcal O_\CX$ is an isomorphism.
\end{enumerate}
A good moduli space of a derived stack $\FX$ is a good moduli space of its classical truncation $\CX=\FX_\cl$ (alternatively, see \cite{ahps}).
\end{definition}

Note that if $\CX$ admits a proper good moduli space then $\CX$ is quasi-compact. This is rarely the case for the stacks parametrizing objects in abelian categories considered above: it is not true for $\CA=\Coh(X)$ if $\dim(X)>0$, but it is true for representations of acyclic quivers. 

\begin{proposition}[{\cite[Theorem 7.23]{AHLH}}]\label{prop: existencegoodmoduli1}
    Let $\CA$ be a good abelian category and $\alpha\in C(\CA)$. If the stack $\FM_{\alpha}$ is of finite type, then $\FM_\alpha$ admits a proper good moduli space.
\end{proposition}

Typically, to get proper good moduli spaces we need to impose some sort of stability. In the rest of the present section, we will  work with a good abelian category $\mathcal A$ as in Section~\ref{subsec: abcat}, and recall the criterion for the existence of proper good moduli spaces for the moduli stacks of semistable objects on $\FM_{\mathcal A}$ (following the exposition of~\cite[Subsection 3.1]{joyce}).

We start with the notion of a weak stability condition on $\mathcal A$. The definition we use is due to Joyce \cite{JO06III};  its roots go back to Rudakov~\cite{rudakov}

%First of all, we set up the relevant notation. 

%$K_0(\mathcal A)$ will stand for the Grothendieck group of $\mathcal A$, and $C(\mathcal A)$  for the subset of $K_0(\mathcal A)$ which consists of the classes $[E]$ of the non-zero objects $E \in \mathcal A$. 

Let $(T, \leq)$ be a totally  ordered set, and let $\mu$ be a map $\mu: C(\mathcal A) \to T$. Given an object $E$ of $\CA$, we will write $\mu(E)\coloneqq \mu([E])$; in other words, $\mu$ is a locally constant function $\FM_\CA\to T$. We call $\mu$ the \emph{slope function}.

\begin{definition}\label{def: stabcondition} We call $(\mu, T, \leq)$ \textit{a weak stability condition} on $\mathcal A$ if for any $\alpha, \beta, \gamma\in C(\CA)$ with $\beta = \alpha + \gamma$ either $\mu(\alpha) \leq \mu(\beta) \leq \mu(\gamma)$, or $\mu(\alpha) \geq \mu(\beta) \geq \mu(\gamma)$.\footnote{We recall that for $\mu$ to be a stability condition (as opposed to a \textit{weak} stabilty condition), the more restrictive condition should hold: either $\mu(\alpha) < \mu(\beta) < \mu(\gamma)$, or $\mu(\alpha) > \mu(\beta) > \mu(\gamma)$, or $\mu(\alpha) = \mu(\beta) = \mu(\gamma)$.} \end{definition}

A weak stability condition defines the notion of $\mu$-stable and $\mu$-semistabile objects as follows.

\begin{definition}\label{def: stableobject} An object $E \in \mathcal A$ is said to be $\mu$-semistable if $\mu(E') \leq \mu(E/E')$ for all subobjects $0\neq E' \subsetneq E$. We say that $E$ is $\mu$-stable if we have a strict inequality $\mu(E') <\mu(E/E')$ in the same setting. If $E$ is $\mu$-semistable but not $\mu$-stable we say it is strictly semistable.
\end{definition}

If the category $\CA$ is $\mu$-Artinian and $\CA$ is Noetherian (which is part of our definition of good abelian category) then any object $E\in \CA$ has a canonical Harder-Narasimhan filtration
\begin{equation}\label{eq: HNfiltration}
0 = E_0 \subsetneq E_1 \subsetneq E_2 \ldots \subsetneq E_n = E\,,\end{equation}
so that all of the successive quotients $F_i=E_i/E_{i-1}$ are $\mu$-semistable, and 
\[\mu(F_1) > \mu(F_2)> \ldots > \mu(F_n)\,.\]
We call \eqref{eq: HNfiltration} the $\mu$-HN filtration of $E$ and the objects $F_1, \ldots, F_n$ the $\mu$-HN factors of $E$. 

Moreover, any $\mu$-semistable object has a Jordan--Hölder filtration whose successive quotients are $\mu$-stable objects with the same slope as the original object; the Jordan--Hölder filtration is not unique in general, but the successive quotients of any two Jordan--Hölder are the same up to a permutation. We refer the reader to \cite[Theorem 4.4, 4.5]{JO06III} for a more precise formulation of these facts and proofs.

%\begin{remark}
%Once the data of the weak stability condition on $\mathcal A$ is given, the objects acquire the canonical Harder-Narasimhan filtrations.

\begin{definition} 
We denote by $\CM^{\mu}_{\alpha}\subseteq \CM_\alpha$ the substack of $\mu$-semistable objects in $\CM_\alpha$.
\end{definition}

We will make a series of assumptions that  guarantee that the stacks of semistables are sufficiently nice to define enumerative invariants. In some natural cases, such assumptions do not hold for arbitrary topological types $\alpha\in C(\CA)$, so we will consider a smaller set of $C(\CA)_\pe\subseteq C(\CA)$ of permissible classes as in \cite[Assumption 5.1(e)]{joyce}. We now state these assumptions and briefly explain them and their implications.

\begin{assumption}\label{ass: stability}
    Let $\CA$ be a good abelian category and $\mu$ a weak stability condition for which Harder--Narasimhan filtrations exist (e.g. $\CA$ is $\mu$-Artinian). We assume that there is a set of permissible classes $C(\CA)_{\pe}\subseteq C(\CA)$ such that for every $\alpha\in C(\CA)_{\pe}$ the following holds:
\begin{enumerate}
        \item The stack $\CM_\alpha^\mu$ is open in $\CM_\alpha$, quasi-compact, and its derived enhancement is quasi-smooth.
        \item There is a (pseudo) $\Theta$-stratification on $\CM_\alpha$ adapted to $\mu$, as explained in Section \ref{subsec: stratificationabelian}, which satisfies the descendending chain condition. In particular, $\CM_\alpha^\mu$ is the semistable loci for this $\Theta$-stratification.
        \item $\mu$ is equivalent to an additive stability condition on $\alpha$, as defined in \cite[Section 7.3]{AHLH}.
        \item If $E$ is $\mu$-semistable with $[E]=\alpha$ and $E'\subseteq E$ then $[E']\in C(\CA)_\pe$.
        \item There are only finitely many partitions
        \[\alpha=\alpha_1+\ldots+\alpha_l\]
        with $\mu(\alpha_i)=\mu(\alpha)$ and $\CM_{\alpha_i}^\mu\neq \emptyset$. By (4), we have $\alpha_i\in C(\CA)_{\pe}$ for any such partition. 
\end{enumerate}
\end{assumption}

Since we required $\CM^{\mu}_{\alpha}\subseteq \CM_\alpha$ to be an open embedding, we obtain a derived enhancement $\FM_\alpha^\mu\subseteq \FM_\alpha$ which is still an open embedding. Since $\FM_\CA$ is locally of finite type by Proposition \ref{defprop: goodabelian}, it follows that $\FM_\alpha^\mu$ is of finite type.

The concept of $\Theta$-stratifications has been introduced and studied by Halpern--Leistner, see for example \cite[Definition 2.1.2]{HLstructure}. Joyce defines in \cite[Definition 3.3.4]{joyce} a weaker notion which he calls pseudo $\Theta$-stratification \cite[Definition 3.3.4]{joyce}; in some cases, pseudo $\Theta$-stratifications are more natural. We will discuss both of these in Section \ref{subsec: Theta}. For now, we just point out that there are techniques to construct (pseudo) $\Theta$-stratifications in most cases of interest. There are two reasons for imposing the existence of these stratifications: they guarantee the existence of proper good moduli spaces (cf. Theorem \ref{thm: goodproper}) and they are necessary for the non-abelian localization theorem (cf. Theorem \ref{thm: NAloc}), which is the heart of our wall-crossing formula. 

\begin{definition}\label{def: additive} We say that $\mu$ is additive at $\alpha$ if there is a totally ordered abelian group $(V, \leq)$ and a monoid homomorphism $\rho_\alpha: C(\mathcal A) \to V$ such that an object $E$ with $[E]= \alpha$ is $\mu$-semistable if and only if, for all subobjects $F \subseteq E$, $\rho([F]) \leq 0$.
\end{definition}

This additive condition is very mild, and true for all the natural stability conditions considered in enumerative geometry. If we can write $\mu(\alpha)=d(\alpha)/r(\alpha)$ with $d,r$ linear, then we may take 
\[\rho_\alpha(\beta)=d(\beta)r(\alpha)-d(\alpha)r(\beta)\,.\]
We impose it since it is necessary in the theorem ensuring the existence of proper good moduli spaces by Alper, Halpern--Leistner and Heinloth:

\begin{theorem}[{\cite[Theorem 7.27]{AHLH}, \cite[Theorem 3.43]{joyce}}]\label{thm: goodproper} Suppose that $\mathcal A$ is a good abelian category, $\mu$ is as in Assumption~\ref{ass: stability} and $\alpha\in C(\CA)_{\pe}$. Then $\FM_\alpha^\mu$ admits a proper good moduli space.
\end{theorem}

The good moduli space $\CM_{\alpha}^\mu\to M_\alpha^\mu$ factors through the rigidification $\CM_{\alpha}^{\mu, \rig}\to M_\alpha^\mu$, and hence $\CM_{\alpha}^{\mu, \rig}, \FM_{\alpha}^{\mu, \rig}$ also admit the same good moduli space, see~\cite[Subsection 1.0.4]{HLverlinde}.

\subsection{Hearts of $D^b(X)$ and tilt stability}\label{subsec: tild}
The setup in this paper gives some new wall-crossing formulas for certain hearts of $D^b(X)$. We collect here some references where the technical assumptions made so far are verified. 

Let $X$ be a smooth projective variety and $\CA$ is the heart of some $t$-structure on $D^b(\CA)$. Assumption \ref{ass: good1} is shown to hold in \cite[Proposition 6.1.7]{HLstructure} provided that $\CA$ is Noetherian and bounded with respect to the standard $t$-structure. Indeed, we may take $\CC$ to be the heart of the induced $t$-structure on $D(\QCoh(X))$. 

For Assumption \ref{ass: good2} we take $\bfT$ to be the dg enhancement of $D^b(X)$, which is saturated by \cite[Lemma 3.27]{TV07}. Establishing that $\CM_\CA$ is an open substack of the classical truncation of $\FM_{\bfT}$ is a non-trivial task, but there are techniques for doing so in many examples of interest\footnote{In all of these, the results are stated as $\CM_\CA$ being open inside Lieblich's \cite{lieblich} stack of complexes with $\Ext^{<0}(F, F)=0$. This is an open substack of $\CM_{D^b(X)}$, as pointed out in \cite[Corollary 3.21]{TV07}.}: see for example \cite[Section 4]{PiyToda}, \cite[Theorem A.3]{arcarabertram}, \cite[Proposition 6.2.7]{HLstructure} and \cite[Proposition 4.6]{BCR}.

One particular way to obtain non-standard $t$-structures of $D^b(X)$ is to tilt the standard $t$-structure with respect to a torsion pair. Even more concretely, we may do this by tilting with respect to a stability condition on $\Coh(X)$, as in \cite[Section 3]{BMTbogomolov}, which we refer to for further details. Given $\BR$-divisors $\omega, B$ with $\omega$ ample, there is a $t$-structure on $D^b(X)$ with heart $\CA_{\omega, B}$ and a stability condition $\nu_{\omega, B}$ on $\CA_{\omega, B}$. When $X$ is a surface, $(\CA_{\omega, B}, \nu_{\omega, B})$ is a Bridgeland stability condition; if $X$ is a $3$-fold then Bridgeland stability conditions can (conjecturally, but known in many cases) be obtained by tilting $\CA_{\omega, B}$ once again. The fact that $\CA_{\omega, B}$ is a good abelian category follows from the results mentioned above, in particular \cite[Proposition 5.2.2]{BMTbogomolov}, \cite[Proposition 6.1.7]{HLstructure} and \cite[Theorem A.3]{arcarabertram}.

The existence of Harder--Narasimhan filtrations with respect to $\nu_{\omega, B}$ is shown in \cite[Lemma 3.2.4]{BMTbogomolov}. For Assumption \ref{ass: stability} we take the permissible classes $C(\CA_{\omega, B})$ to be the ones with $\nu_{\omega, B}(\alpha)<\infty$; these are thought of as torsion-free classes in $\CA_{\omega, B}$. The stacks $\CM_{\alpha}^{\nu_{\omega, B}}$ being finite type and open in $\CM_\alpha$ is established in \cite{todaK3, PiyToda}; see also \cite[Proposition 6.2.7]{HLstructure}. 

The finiteness condition (5) in Assumption \ref{ass: stability} is known to hold if $X$ is a surface or a 3-fold satisfying the \cite{BMTbogomolov} inequality, see \cite[Conjecture 4.1]{BMS}. For the case of 3-folds see \cite[Proposition 4.1]{FTr0}. Finally, the existence of $\Theta$-stratifications is also known when $X$ is a surface or a 3-fold satisfying the BMT inequality by using \cite[Theorem 6.5.3]{HLstructure}. Regarding the assumptions in loc. cit.:
\begin{enumerate}
    \item Since the set of walls is locally finite \cite[Proposition 12.5]{BMS} we may assume that $\omega, B$ are $\BQ$-divisors.
    \item Generic flatness in this example follows from \cite[Proposition 4.11]{PiyToda}.
    \item Boundedness of quotients for surfaces follows from \cite[Proposition~3.15]{todaK3}.  While \cite[Proposition~3.15]{todaK3} is only for Bridgeland stability conditions, a modification of its proof using \cite[Theorem C.5]{FTr0} shows boundedness of quotients for 3-folds satisfying the BMT inequality.
\end{enumerate}

The last point that is necessary to address is quasi-smoothness. When $X$ is a surface with nef anticanonical this is shown in \cite[Lemma 7.8]{lmquivers}. For Fano 3-folds this will be proved in forthcoming work by the first author. We summarize the discussion above in the following proposition:

\begin{proposition}\label{prop: tilt}
Let $X$ be a surface with nef anticanonical or a Fano 3-fold satisfying the BMT inequality. Then Assumptions \ref{ass: good1}, \ref{ass: good2} and \ref{ass: stability} hold for $(\CA_{\omega, B}, \nu_{\omega, B})$, with the permissible classes being the ones satisfying $\nu_{\omega, B}(\alpha)<\infty$. 
\end{proposition}

\section{Preliminaries: $K$-theory and $K$-homology}
\label{sec: Khomology}

\subsection{$K$-theory of stacks}
\label{subsec: Ktheorystacks}

Let $\FX$ be a derived Artin stack. We briefly review the algebraic $K$-theory of $\FX$, following mostly \cite{khan}, and recall some results that we will need. Denote by $D\QCoh(\FX)$ the (unbounded) category of quasi-coherent complexes on $\FX$ and by $\QCoh(\FX)$ its heart. Inside $D\QCoh(\FX)$ we have the subcategories of perfect and coherent complexes which we denote as
\[\Perf(\FX)\textup{ and }D^b\Coh(\FX)\,.\]
A complex is perfect/coherent if it is perfect/coherent on smooth affine charts (see \cite[Definition 1.5]{khan}). The $t$-structure on $D\QCoh(\FX)$ restricts to $D^b\Coh(\FX)$ and its heart is the abelian category of coherent sheaves $\Coh(\FX)$. A complex being in $D^b\Coh(\FX)$ is equivalent to having coherent cohomology groups, all but finitely many being 0.

When the structure sheaf $\CO_\FX$ is bounded, in the sense that $H^i(\CO_{\FX})=0$ for all but finitely many $i$, there is an inclusion 
\[\Perf(\FX)\subseteq D^b\Coh(\FX)\,.\]
\begin{proposition}[{\cite[Corollary 6.1]{khan}}]\label{prop: quasismoothObounded}
    If $\FX$ is quasi-smooth then $\CO_\FX$ is bounded.
\end{proposition}
The algebraic $K$-groups of $\FX$ are defined as follows:
\begin{definition}\label{def: KG}
    Let $\FX$ be a derived stack. We let
    \begin{align*}
        K^\ast(\FX)&=K_0(\Perf(\FX))\,,\\
        G(\FX)&=K_0(\Coh(X))=K_0(D^b\Coh(X))\,
     \end{align*}
     where $K_0$ denotes the Grothendieck group of a category.
\end{definition}

The notation $K^\ast$ is not standard, but we will use it to indicate that we think of $K^\ast$ as a cohomology theory. When $\CO_\FX$ is bounded, there is a canonical morphism $K(\FX)\to G(\FX)$ induced by the inclusion of the respective categories; if $\FX$ is regular this map is an isomorphism. It it shown in \cite[Corollary 3.4]{khan} that if $\FX$ is Noetherian then $G$-theory is insensitive to the derived structure, i.e. the inclusion of the classical truncation $\iota\colon \CX\coloneqq \FX_{\cl}\to \FX$ induces an isomorphism
\[\iota_\ast\colon G(\CX)\to G(\FX)\,.\]
The analogous statement is not true in general for $K^\ast(\FX)$ \cite{annala}.

\subsubsection{Operations in $K$ and $G$ theory}
\label{subsec: operationsKG}

Given a morphism of derived stacks $f\colon \FX\to \FY$ we have a derived pullback and a derived pushforward
\[Lf^\ast \colon D\QCoh(\FY)\to D\QCoh(\FX)\textup{ and }Rf_\ast \colon D\QCoh(\FX)\to D\QCoh(\FY)\,.\]
The pullback $Lf^\ast$ always preserves perfect complexes. If $f$ is of finite Tor amplitude then it also preserves coherent complexes. Hence we have pullback morphisms
\begin{align*}
f^\ast&\colon K^\ast(\FY)\to K^\ast(\FX) \\
f^\ast&\colon G(\FY)\to G(\FX) \quad \textup{($f$ has finite Tor amplitude)}
\end{align*}
If $f$ is proper and representable, the pushforward also preserves coherent complexes and, under some extra mild conditions, perfect complexes. Hence, we have pushforward maps
\begin{align*}
f_\ast &\colon K^\ast(\FX)\to K^\ast(\FY) \quad \textup{($f$ proper, representable,} \\
& \hspace{4.5cm} \textup{finite Tor amplitude, locally finite type)}\\
f_\ast & \colon G(\FX)\to G(\FY) \quad\quad \textup{($f$ proper, representable)}
\end{align*}

There are some situations where pushforwards exist even without the map being proper, see \cite{HLpreygel}. A particularly important one is the following:

\begin{proposition}[{\cite[Theorem 4.16 (x)]{alper}}]\label{prop: pushforwardgoodmoduli}
    Let $\CX$ be the classical truncation of ~$\FX$. If $f\colon \CX\to X$ is a good moduli space and $\CX$ is Noetherian, then $f_\ast$ preserves coherent complexes, and in particular there is a well-defined pushforward
    \[f_\ast\colon G(\FX)\simeq G(\CX)\to G(X)\,.\]
    The same is true if $X$ is not necessarily an algebraic space but $f$ still satisfies the two conditions in the definition of good moduli space.
\end{proposition}

Another example is Proposition \ref{prop: Ktheoryrig} below, concerning rigidification maps. Since the tensor product restricts to $\Perf(\FX)$, there is a product
\[K^\ast(\FX)\otimes K^\ast(\FX)\to K^\ast(\FX)\,.\]
Tensoring by a perfect complex preserves coherent complexes, which endows $G(\FX)$ with a $K^\ast(\FX)$-module structure
\[\cap \colon K^\ast(\FX)\otimes G(\FX)\to G(\FX)\,.\]
The map from $K$-theory to $G$-theory when $\CO_\FX$ is bounded is precisely $-\cap [\CO_\FX]$.

\subsubsection{Exterior powers}\label{subsubsec: alternating}
\label{subsubsec: exterior}

Exterior powers $\Lambda^n(-)$ endow $K^\ast(\FX)$ with the structure of a $\lambda$-ring. We let
\begin{equation}\label{eq: lambdau}
    \Lambda_{-u}(V)=\sum_{n\geq 0}(-u)^n\Lambda^n(V)\in K^\ast(\FX)\llbracket u \rrbracket\,.\end{equation}
A common theme in the paper will be how to expand $\Lambda_{-u}(V)$ in different directions. We regard \eqref{eq: lambdau} above as the ``expansion of $\Lambda_{-u}(V)$ around $u=0$''. 

\begin{lemma}\label{lem: symmetryVB}
    If $V$ is a vector bundle, then 
    \begin{equation}
        \label{eq: uu-1lambda}
    \Lambda_{-u}(V)=(-u)^{\rk(V)}\Lambda_{-u^{-1}}(V^\vee)\otimes \det(V)\,\end{equation}
    in $K^\ast(\FX)[u]$.
\end{lemma}
\begin{proof}
Let $r=\rk(V)$. This follows from the fact that
\[\Lambda^n(V)\otimes \Lambda^{r-n}(V)\to \Lambda^r(V)=\det(V)\]
defines a perfect pairing, and hence an isomorphism
\[\Lambda^n(V)\simeq \Lambda^{r-n}(V)^\vee\otimes \det(V)\simeq \Lambda^{r-n}(V^\vee)\otimes \det(V)\,.\qedhere\]
\end{proof}
\begin{remark}\label{rmk: inverseeulerclass}
    When $V$ is a vector bundle the class
    \[\Lambda_{-1}(V^\vee)=\sum_{n=0}^{\rk V}(-1)^n\Lambda^n(V^\vee)\]
    plays the role of the inverse Euler class of $V$, for example in the context of torus localization  \cite[Section 2.3]{okounkovlecturesK}
\end{remark}

When $V$ is not a vector bundle, the right hand side of \eqref{eq: uu-1lambda} -- which is a Laurent series in $u^{-1}$ -- might be interpreted as the ``expansion of $\Lambda_{-u}(V)$ around $u=\infty$''. This might be made precise when $V\in K^\ast(\FX)$ can be written as an alternating sum of line bundles $V=\sum_{i=1}^m \epsilon_i L_i$ with $\epsilon_i\in \{-1,1\}$. Then we have an equality 
\begin{align*}\label{eq: lineTSigma}
\Lambda_{-u}(V)&=\prod_{i=1}^m (1-u L_i)^{\epsilon_i}=\prod_{i=1}^m\big( (1-u^{-1} L_i^\vee)^{\epsilon_i}(-uL_i)^{\epsilon_i}\big)\\
&=(-u)^{\rk(V)}\Lambda_{-u^{-1}}(V^\vee)\otimes \det(V)
\end{align*}
as rational functions in $u$ with coefficients in $K^\ast(\FX)$, but the left hand side is the expansion in $K^\ast(\FX)\lpp u \rpp$ while the right hand side is the expansion in $K^\ast(\FX)\lpp u^{-1}\rpp$. 

The expansion around $u=1$,
\[\Lambda_{-u}(V)\in K^\ast(\FX)\lppbig(1-u)^{-1}\rppbig\,,\]
is also used in \cite[Definition 2.1.8]{Liu}. Unlike the previous two, its definition requires that we assume that $V$ can be represented as the difference between the classes of two vector bundles; this is always the case when $\FX$ has the resolution property, for example if $\FX$ is a smooth classical scheme. It is defined by using the splitting principle and expanding the rational function above around $u=1$. In particular, if $L$ is a line bundle,
\[\Lambda_{-u}(-L)=\sum_{n\geq 0}\frac{u^n}{(1-u)^{n+1}}(L-1)^n\,.\]
When $\FX=X$ is a scheme, $(L-1)^n=0$ for sufficiently large $n$ and the expression above is a Laurent polynomial in $1-u$. Indeed, we have the following:

\begin{lemma}[{\cite[2.1.8, 2.1.11]{Liu}}]\label{lem: symmetryexpansion}
    Let $X$ be a finite type scheme and $V\in K_0(\textup{Vect}(X))$. Then there is an element of $K_0(\textup{Vect}(X))[(1-u)^{\pm 1}]$ whose expansion in $u$ is 
    \[\Lambda_{-u}(V)\in K_0(\textup{Vect}(X))\llbracket u\rrbracket \]
    and whose expansion in $u^{-1}$ is 
    \[(-u)^{\rk(V)}\Lambda_{-u^{-1}}(V^\vee)\otimes \det(V)\in K_0(\textup{Vect}(X))\lpp u^{-1}\rpp\,.\]
\end{lemma}

\subsubsection{$K$-theory of $B\BG_m$ and rigidifications}\label{sec: Kthrig}

A quasi-coherent sheaf on $B\BG_m$ corresponds to a representation of $B\BG_m$ on a (possibly infinite dimensional) vector space $V$. Such representation is coherent if and only if it is perfect if and only if $V$ is finite. The $K$-theory of $B\BG_m$ is isomorphic to
\[K^\ast(B\BG_m)\simeq R(\BG_m)\simeq \BZ[u^{\pm 1}]\]
as a ring, where $R(G)$ denotes the representation algebra of a group $G$ and $u$ is the class of the standard representation corresponding to the weight 1 action of $\BG_m$ on $\BC$. The cohomology groups of a quasi-coherent sheaf $V$ are given by
\[H^i(B\BG_m, V)=\begin{cases}V^G&\textup{ if }i=0\\
0&\textup{otherwise}
\end{cases}\]
In particular, if $\pi\colon B\BG_m\to \pt$ then the functor $\pi_\ast=R\pi_\ast$ is exact and preserves perfect complexes. The pushforward
\[\BZ[u^{\pm 1}]\simeq K^\ast(B\BG_m)\xrightarrow{\pi_\ast} K^\ast(\pt)\simeq \BZ\]
is identified with the operator $[u^0]$ of extracting the constant term of a Laurent polynomial.

Consider now a (derived) stack  $\FX$ equipped with a $B\BG_m$ action $\Psi\colon B\BG_m\times \FX\to \FX$. Given a complex $F$ on $\FX$, the $B\BG_m$ action induces a canonical weight decomposition
\[F=\bigoplus_{\mu\in \BZ} F_\mu\]
where
\[\Psi^\ast F=\bigoplus_{\mu\in \BZ} u^\mu\otimes F_\mu\,.\]
A complex is said to be of weight $\mu$ if $F=F_\mu$. The next proposition relates the $K$-theory of a stack and its rigidification (cf. Definition~\ref{def: rigstack}). 

\begin{proposition}\label{prop: Ktheoryrig}
The pullback
\[\pi^\ast \colon K^\ast(\FX^\rig)\to K^\ast(\FX)\]
is injective and its image are the weight 0 classes. Moreover, the pushforward
\[\pi_\ast \colon K^\ast(\FX)\to K^\ast(\FX^\rig)\]
is well defined and we have
\[\pi^\ast \pi_\ast F=F_0\,.\]
\end{proposition}
\begin{proof}
This follows essentially by \cite{BSgerbs}. There, it is shown that the weight decomposition induces an equivalence of categories
\begin{equation}\label{eq: qcohrigdecomposition}
D\QCoh(\FX)\simeq \prod_{\mu\in \BZ}D\QCoh(\FX)_\mu\end{equation}
where $D\QCoh(\FX)_\mu$ denotes the Serre subcategory of weight $\mu$ complexes. Moreover, \cite[Proposition 5.7]{BSgerbs} shows that $\pi^\ast=L\pi^\ast$ gives an equivalence 
\[D\QCoh(\FX^\rig)\isoto D\QCoh(\FX)_0\,\]
with inverse given by $\pi_\ast=R\pi_\ast$. It is clear that the decomposition \eqref{eq: qcohrigdecomposition} restricts to an isomorphism
\begin{equation*}
\Perf(\FX)\simeq \bigoplus_{\mu\in \BZ}\Perf(\FX)_\mu\,.\end{equation*}
From here all the statements in the proposition are clear.
\end{proof}

If $T$ is a $n$-dimensional torus then these statements generalize in a straightforward way. The $K$-theory of $BT$ is
\[K^\ast(BT)\simeq R(T)\simeq \BZ[\Lambda]\]
where $\Lambda\subseteq \mathfrak t^\vee$ is the weight lattice of the torus. A $BT$ action on $\FX$ induces a weight decomposition $F\simeq \bigoplus_{\mu\in \Lambda} F_\mu$ on any sheaf on $\FX$. A splitting $T\simeq \BG_m\times \ldots\times \BG_m$ induces a basis on $\Lambda$ and a corresponding isomorphism
\[K^\ast(T)\simeq \BZ[u_1^{\pm 1}, \ldots, u_n^{\pm 1}]\,.\]

\subsection{$K$-homology}\label{subsec: K-hom}

Under appropriate conditions, a moduli space $M$ with a virtual structure sheaf $\CO_M^\vir$ determines an ``integration'' functional $K^\ast(M)\to \BZ$ by taking the Euler characteristic:
\[V\mapsto \chi(M, \CO^\vir_M\otimes V)\,.\]
Liu defines in \cite{Liu} a notion of $K$-homology, which roughly is the space where such functionals live. The most naive thing would be to simply define $K$-homology as the dual to $K$-theory. However, the lack of a Kunneth decomposition creates issues when trying to define some algebraic structures on $K$-homology. This is solved by working with an operational theory that forces the existence of a Kunneth morphism.

\begin{definition}\label{def: K-hom}
Let $\FX$ be a derived stack. An element $\phi$ in the $K$-homology group $K_\ast(\FX)$ is a collection of $K^\ast(S)$-linear maps
\[\{\phi_S\colon K^\ast(\FX\times S)\to K^\ast(S)\}\]
for every derived stack $S$ which satisfies the following compatibility condition: for any morphism $h\colon S\to S'$, the diagram 
\begin{center}
\begin{tikzcd}
    K^\ast(\FX\times S')\arrow[d, "\phi_{S'}"]\arrow[r, "(\id\times h)^\ast"] & K^\ast(\FX\times S)\arrow[d, "\phi_{S}"]\\
     K^\ast(S')\arrow[r, "h^\ast"] & K^\ast(S)\end{tikzcd}
\end{center}
commutes.
\end{definition}

\begin{remark}
    There are two more conditions in the definition of $K$-homology in \cite{Liu}. The ``equivariant localization'' is irrelevant for us since we are not working equivariantly. The  ``finiteness condition'', on the other hand, is a fundamental difference and we really need to exclude it. For example, for us
    \[K_\ast(B\BG_m)\simeq \Hom_\BZ(K^\ast(B\BG_m), \BZ)\simeq \BZ^\BZ\,;\]
    with the finiteness condition, the $K$-homology of $B\BG_m$ is isomorphic to $\BZ[\phi]$ (cf. \cite[Proposition 2.3.5]{Liu}).
    We will further discuss this finiteness condition in Section \ref{subsec: LiuVA}.
\end{remark}

Note that, by taking $S=\pt$, a class in $K$-homology determines a functional $\phi_\pt\colon K^\ast(\FX)\to \BZ$. If $\FX$ has the property that the Kunneth map
\[K^\ast(\FX)\otimes K^\ast(S)\to K^\ast(\FX\times S)\]
is an isomorphism for every $S$, then $K_\ast(\FX)\simeq K^\ast(\FX)^\vee$. For most purposes, it is enough to think of a $K$-homology class as just a functional, so we will often -- when no subtlety arises -- do so for the sake of clarity of exposition.

We have the following structures on $K$-homology:

\textit{Cap product.} There is a cap product
    \[\cap\colon K^\ast(\FX)\otimes K_\ast(\FX)\to K_\ast(\FX)\]
    obtained from the tensor product on $K^\ast$. At the level of functionals, 
    \[(V\cap \phi)_\pt(W)=\phi_\pt(V\otimes W)\]
    for $V, W\in K^\ast(\FX)$.
    
\textit{Kunneth map.} For any $\FX, \FY$ there is a Kunneth map
    \[\boxtimes \colon K_\ast(\FX)\otimes K_\ast(\FY)\to K_\ast(\FX\times \FY)\,.\]
    It is defined by setting $(\phi\boxtimes \psi)_S$ to be the composition
    \[K^\ast(\FX\times \FY\times S)\xrightarrow{\phi_{\FY\times S}} K^\ast(\FY\times S) \xrightarrow{\psi_{S}} K^\ast(S)\,.\]
Note how this definition uses the operational definition in a fundamental way.

\textit{Pushforward.} Given an arbitrary map $f\colon \FX\to \FY$ there is a pushforward $f_\ast\colon K_\ast(\FX)\to K_\ast(\FY)$. At the level of functionals, this is the dual of the pullback on $K^\ast$.

\textit{Proper pullback.} If $f \colon \FX\to \FY$ is a proper, representable, finite Tor amplitude, locally finite type morphism, then there is a pullback $f^\ast\colon K_\ast(\FY)\to K_\ast(\FX)$. At the level of functionals, this is the dual of the pushforward on $K^\ast$. By Proposition~\ref{prop: Ktheoryrig} we also have pullbacks in $K_\ast$ along rigidification morphisms.

\textit{$G$ to $K_\ast$ and universal invariants.} Suppose that $\FX$ is Noetherian and admits a proper good moduli space. Then there is a canonical map
    \begin{equation}\label{eq: GtoKmorphism}
   G(\FX)\to K_\ast(\FX)\,. \end{equation}
At the level of functionals, it sends a coherent complex $C\in G(\FX)$ to the functional
    \[K^\ast(\FX)\ni V\mapsto \chi(\FX, V\cap C)\in \BZ\,.\]
The operational description requires some work. The coherent complex $C$ is sent to $\phi$ where 
\[\phi_S(V)=(Rp_{1})_\ast (V\otimes p_2^\ast C)\]
and $p_1, p_2$ are the projections of $\FX\times S$ onto $\FX$ and $S$ respectively. Using the fact that pushforward along the good moduli morphism preserves coherent complexes (cf. Proposition \ref{prop: pushforwardgoodmoduli}), it is not hard to see that the complex above is coherent, but being actually perfect requires a proof.

\begin{lemma}\label{lem: gtokhom}
Let as above $\FX$ be a Noetherian derived stack admitting a proper good moduli space, $V$ in $\Perf(\FX\times S)$ and $C$ in $D^b\Coh(\FX)$. Then  $(Rp_{1})_\ast (V\otimes p_2^\ast C)$ is in $\Perf(S)$.
\end{lemma}
\begin{proof}
    Let $\iota\colon \CX=\FX_\cl\to \FX$ be the classical truncation. Since $D^b\Coh(\FX)$ has a bounded $t$-structure with heart $\Coh(\FX)$ and $\iota_\ast\colon \Coh(\CX)\isoto \Coh(\FX)$ is an isomorphism (cf. \cite[Proposition 3.3, Corollary 3.4]{khan}), we may assume that $C$ is obtained as a pushforward from $\CX$. Hence, it is enough to prove the statement for the classical truncation $\CX$ instead of $\FX$.
    
    The proof uses the notion of relatively perfect complexes as in \cite{lieblich} or \cite[Section 0DHZ]{stacks}; morally, we are introducing non-perfectness only along $\CX$, which is what this notion captures. By \cite[Example 0DI1, Lemma 0DI5]{stacks} the complex $p_2^\ast C=Lp_2^\ast C$ is $S$-perfect. By \cite[Lemma 0DI4]{stacks} it follows that $E\coloneqq V\otimes p_2^\ast C$ is also $S$-perfect. Moreover, it is coherent since $V$ is perfect and $p_2^\ast C$ is coherent.

If $\CX$ itself was proper then \cite[Lemma 0DJT]{stacks} would give the desired conclusion. Hence, it is enough to prove that the pushforward along the good moduli map (base changed to $S$) $f\colon \CX\times S\to X\times S$ preserves $S$-perfect complexes. This follows from the same argument with the push-pull formula in the proof of \cite[Lemma 08EV]{stacks} using the fact that $f_\ast=Rf_\ast$ is exact (which is preserved by base change, see \cite[Proposition 3.9 (iv)]{alper}).\qedhere
\end{proof}

It is not hard to see that the map $G(\FX)\to K_\ast(\FX)$ is a homomorphism of $K^\ast(\FX)$-modules.

If furthermore $\CO_\FX$ is bounded, then the image of $[\CO_\FX]\in G(\FX)$ via the map above defines a $K$-theoretic fundamental class $[\FX]$ on $\FX$. Note that if $\FX$ is quasi-smooth then $\CO_\FX$ is bounded \cite[Corollary 6.1]{khan}. When this is the case, the structure sheaf $[\CO_\FX]$ corresponds, under the isomorphism $G(\FX)\simeq G(\CX)$, to the well-known virtual fundamental sheaf $[\CO_{\CX}^\vir]$ on the classical truncation $\CX=\FX_\cl$ \cite{BF, leeQK}. Roughly speaking, the fundamental class $[\FX]$ tells us how to ``integrate against $[\CO_{\CX}^\vir]$'', which is what $K$-theoretic invariants in enumerative geometry are. We summarize this discussion in the following theorem:

\begin{theorem}\label{thm: deltainvariants}
Let $\FX$ be a Noetherian derived stack admitting a proper good moduli space. Then there is a well-defined homomorphism
\[G(\FX)\to K_\ast(\FX)\]
of $K^\ast(\FX)$-modules constructed as above. If furthermore $\CO_\FX$ is bounded (e.g. $\FX$ is quasi-smooth), then there is a $K$-theoretic fundamental class
\[[\FX]=[\FX]_K\in K_\ast(\FX)\]
defined as the image of $[\CO_\FX]\in G(\FX)$ under the previous map which, at the level of functionals, sends
\[K^\ast(\FX)\ni V\mapsto \chi(\FX, V)=\chi(\CX, \CO_{\CX}^\vir\otimes V_{|\CX})\in \BZ\,.\]
\end{theorem}

Let us emphasize that this fact is one of the places where we use $K$-theory, rather than cohomology, in a fundamental way. Khan \cite{khanVFC} defines a virtual fundamental class in Borel--Moore homology of $\FX$, which is analogous to $[\CO_\FX]\in G(\FX)$ -- they are even comparable through a Grothendieck--Riemann--Roch type formula, cf. \cite[Theorem 3.23]{khanVFC}. However, as far as we know, unlike in $G$-theory, there is no pushforward along good moduli maps for Borel--Moore homology. Hence there is no good way to integrate cohomology classes against such fundamental class, and therefore no good way to define a fundamental class in the singular homology of $\FX$.

We should also point out that, unless $\FX$ is actually a scheme, the class $[\FX]$ in general does not satisfy Liu's finiteness condition. See Example \ref{ex: deltanotregular}.

\begin{remark}\label{rmk: kunnethdelta}
    It is an easy exercise to verify that the $K$-theoretic fundamental class is well behaved with respect to products and the Kunneth map, in the sense that
    \[[\FX\times \FZ]=[\FX]\boxtimes [\FZ]\,.\]
\end{remark}

\section{The $K$-Hall algebra and generalized $K$-theoretic invariants}
\label{sec: Khallinvariants}

In this section we will introduce an associative algebra structure on $K_\ast(\FM_\CA^\rig)$, when $\CA$ is a good abelian category. This algebra is the analog of the motivic Hall algebra (or of the quantum torus) in $K$-theory, and we call it the $K$-Hall algebra. 

As in the motivic setting, we will use this product to define generalized $K$-theoretic invariants -- $\varepsilon$ classes -- by taking a formal logarithm of $\delta$ classes defined directly via the stacks of semistable objects. The $\varepsilon$ classes will be shown to agree with Joyce--Liu classes in Sections \ref{sec: framing} and \ref{sec:Liuvert}.

\begin{remark}
    Let us emphasize that, despite sharing some similarities, the algebra that we introduce is not the $K$-theoretic Hall algebra from \cite{tudorKHall}. The $K$-theoretic Hall algebra is also an associative algebra, but the underlying vector space is $G(\FM)$ and the product is defined by a push-pull construction on the stack of extensions. Another difference is that the $K$-theoretic Hall algebra requires some extra condition, such as $\FM_\CA$ being quasi-smooth or $(-1)$-shifted symplectic, while ours does not. We hope that the distinction ``$K$-Hall algebra'' versus ``$K$-theoretic Hall algebra'' will be enough not to confuse the reader. It would also be interesting to understand the relation between the two constructions when $\CA$ has homological dimension $\leq 2$. 
\end{remark}

\subsection{The $K$-Hall algebra product}\label{subsec: K-Hal-prod}

We will start by introducing some ingredients necessary for the definition of the $K$-Hall product. On the stack $\FM\times \FM$ there is a $B\BG_m\times B\BG_m$ action, where each $B\BG_m$ acts on each factor. Let us consider the stack
\[(\FM\times \FM)^\rig\]
obtained by rigidifying with respect to the diagonal $B\BG_m\subseteq B\BG_m\times B\BG_m$. This stack still admits a 
\[(B\BG_m\times B\BG_m)/B\BG_m\simeq B\BG_m\times 1\,.\]
action. If we further rigidify with respect to this action, we obtain $\FM^\rig\times \FM^\rig$; we denote by
\[\pi\colon (\FM\times \FM)^\rig \to \FM^\rig\times \FM^\rig \]
this second rigidification map. Since $\Sigma\colon \FM\times \FM\to \FM$ intertwines the diagonal $B\BG_m$ action on $\FM\times \FM$ with the $B\BG_m$ action on $\FM$, there is a rigidified map
\[\Sigma^\rig\colon (\FM\times \FM)^\rig\to \FM^\rig\,.\]

The last ingredient is the following quasi-coherent complex on $\FM\times \FM$, which is central in the non-abelian localization theorem \cite{TW, HLNAL}:
\begin{equation}
    \label{eq: gammacomplex}
\Gamma_-\coloneqq \Lambda_{-1}\big(\Ext_{21}^\vee+\Ext_{12}\big)\otimes \det(\Ext_{12})^\vee[\rk_{12}]\,,
\end{equation}
where $\rk_{12}$ is the locally constant function $\rk(\Ext_{12})$ on $\FM\times \FM$. Since $\Ext_{12}, \Ext_{21}$ have weights $(-1, 1)$ and $(1, -1)$, respectively, with respect to the $B\BG_m\times B\BG_m$ action, they are weight 0 for the diagonal $B\BG_m$ and hence descend to $(\FM\times \FM)^\rig$. Thus the same is true for $\Gamma_-$.

\begin{definition}[$K$-Hall algebra]\label{def: KHall}
Let $\CA$ be a good abelian category. We define the $K$-Hall algebra\footnote{The algebra could be defined with $\BZ$ coefficients, without tensoring by $\BQ$. However, rational coefficients are necessary for the definition of generalized $K$-theoretic invariants.} 
\[\BK(\CA)=K_\ast(\FM^\rig_\CA)_\BQ\]
with product 
\[\ast\colon K_\ast(\FM^\rig)\otimes K_\ast(\FM^\rig)\to K_\ast(\FM^\rig)\]
defined by
\[\phi\ast \psi=\Sigma_\ast^\rig\big([\Gamma_-]\cap \pi^\ast(\phi\boxtimes \psi)\big)\,.\]
\end{definition}

\begin{remark}\label{rmk: variationsHallproduct}
    This definition and the proof of associativity below only depend on the structures $\Sigma, \Phi, \Ext_{12}$, and the compatibilities between them described in Section \ref{subsec: abcat}. In particular, the same construction works if we take any substack of $\FM_\CA$ closed under direct sum and the $B\BG_m$ action or if we take $\FM_\bfT$ for a dg category $\bfT$.
\end{remark}

Let us unpack the definition and justify that it makes sense. We first recall that $\boxtimes$ is the Kunneth map defined in Section \ref{subsec: K-hom}, so $\phi\boxtimes \psi\in K_\ast(\FM^\rig\times \FM^\rig)$. The pullback $\pi^\ast$ is well-defined in~$K_\ast$ since pushforwards of rigidification maps are well-defined in $K^\ast$, cf. Proposition~\ref{prop: pushforwardgoodmoduli}. 

Since $\Gamma_-$ is only quasi-coherent, the operation $[\Gamma_-]\cap -$ in not well defined in $K_\ast$. Nevertheless, the end expression still makes sense as an element of $K_\ast$. For simplicity, let us argue that at the level of functionals. The important observation is that, despite $\Gamma_-$ not being perfect (or even coherent), its weight $\mu$ part is perfect and vanishes for $\mu\gg 0$. Here, we mean weight with respect to the $B\BG_m$ action
\[\Psi_1\colon B\BG_m\times (\FM\times \FM)^\rig\to (\FM\times \FM)^\rig\]
acting on the first $\FM$. Note that $\Ext_{12}$, $\Ext_{21}$ have weights $-1$ and $1$ with respect to this action. Hence
\begin{align*}
\Gamma_-(u)\coloneqq \Psi_1^\ast \Gamma_-&=\Lambda_{-1}\big(\Ext_{21}^\vee\otimes u^{-1}+\Ext_{12}\otimes u^{-1}\big)\otimes \det(\Ext_{12}\otimes u^{-1})^\vee[\rk_{12}]\\
&=(-u)^{\rk_{12}} \Lambda_{-u^{-1}}(\Ext_{21}^\vee)\otimes  \Lambda_{-u^{-1}}(\Ext_{12})\otimes \det(\Ext_{12})^\vee\,.
\end{align*}

In other words, even though $\Gamma_-$ itself is not perfect, the class of $\Gamma_-(u)$ is in the ring
\[K^\ast\big((\FM\times \FM)^\rig\big)\lpp u^{-1}\rpp\,\]
of Laurent series in $u^{-1}$ with coefficients in $K^\ast$, where $u$ denotes the class of the tautological line bundle on $B\BG_m$.

\begin{lemma}\label{lem: astwelldefined}
    Let $V$ be a perfect complex in $\FM^\rig$. Then 
    \[\pi_\ast(\Gamma_-\otimes (\Sigma^\rig)^\ast V)\]
    is perfect in $\FM^\rig\times \FM^\rig$. 
\end{lemma}
\begin{proof}
By the proof of Proposition \ref{prop: pushforwardgoodmoduli} applied to $\CX=(\FM\times \FM)^\rig$, the complex we are considering is perfect if and only if the weight 0 part of $\Gamma_-\otimes (\Sigma^\rig)^\ast V$ is perfect. But this follows from the fact that 
\begin{enumerate}
    \item $(\Sigma^\rig)^\ast V$ is perfect, hence its weight $\mu$ part is also perfect and it is zero except for finitely many $\mu$.
    \item $\Gamma_-$ has perfect weight $\mu$ parts which are 0 for $\mu>\rk(\Ext_{12})$.\qedhere
\end{enumerate}
\end{proof}

One upshot of the considerations above is that the $K$-Hall product (at the level of functionals) is characterized by the fact that the dual ``coproduct''\footnote{This is not really a coproduct, since there is no Kunneth map $K^\ast(\FM^\rig\times \FM^\rig)\to K^\ast(\FM^\rig)\otimes K^\ast(\FM^\rig)$.} fits into the diagram
\begin{center}
    \begin{tikzcd}
        K^\ast(\FM)\arrow[r]\arrow[d, "\rho^\ast"] &K^\ast(\FM\times \FM) \arrow[d, "(\rho\times \rho)^\ast"]\\
          K^\ast(\FM^\rig)\arrow[r, "\ast^\vee"]& K^\ast(\FM^\rig\times \FM^\rig)\\
    \end{tikzcd}
\end{center}
where $\rho\colon \FM\to \FM^\rig$ is the rigidification map and the top arrow is given by the morphism
\begin{equation}
    \label{eq: hallcoproduct}[u^0]\Big(\Gamma_-(u)\otimes \Psi_1^\ast\Sigma^\ast(-)\Big)\,.
\end{equation}
Above, $[u^0](\ldots)$ denotes the constant coefficient of $\ldots$. For a more symmetric version, the same is also true if we replace the top arrow by 
\begin{align*}
    \label{eq: hallcoproduct2}[u_1^0u_2^0]\Big(\Gamma_-(u_1, u_2)\otimes \Psi_{12}^\ast\Sigma^\ast(-)\Big)\,.
\end{align*}
where $\Psi_{12}\colon B\BG_m\times B\BG_m\times \FM\times \FM\to \FM\times \FM$ is the product of the two actions and 
\begin{align*}\Gamma_-(u_1, u_2)&\coloneqq \Psi_{12}^\ast \Gamma_-\\
&=(-u_1/u_2)^{\rk_{12}}\Lambda_{-u_2/u_1}(\Ext_{21}^\vee)\otimes  \Lambda_{-u_2/u_1}(\Ext_{12})\otimes \det(\Ext_{12})^\vee\,.\end{align*}

\begin{remark}\label{rmk: eulerclassTSigma}
The $K$-theory class of the derived normal bundle to the morphism $\Sigma$ is the complex $\BN_\Sigma=-\Ext_{12}-\Ext_{21}$. The inverse of its $K$-theoretic Euler class (cf. Remark \ref{rmk: inverseeulerclass}) is
\[\Gamma\coloneqq \Lambda_{-1}(\BN_\Sigma^\vee)^{-1}=\Lambda_{-1}\big(\Ext_{12}^\vee+\Ext_{21}^\vee\big)\,.\]
Note that
\[\Gamma(u)\coloneqq  \Psi_1^\ast \Gamma=\Lambda_{-u^{-1}}(\Ext_{12}^\vee)\otimes \Lambda_{-u}(\Ext_{21}^\vee)\]
is precisely the complex appearing in the definition of the vertex algebra in \cite{Liu}, where it is interpreted in terms of its $1-u$ expansion, in the sense of Lemma \ref{lem: symmetryexpansion}. Unlike $\Gamma_-$, the complex $\Gamma$ does not have the property that its weight $\mu$ part is perfect for every $\mu$ since it is a product of a Laurent series in $u^{-1}$ with a Laurent series in $u$. The complex $\Gamma_-(u)$ can be regarded as the $u^{-1}$ expansion of $\Gamma(u)$, following the discussion in Section \ref{subsubsec: alternating} and Lemma \ref{lem: symmetryexpansion}.
\end{remark}

\subsubsection{Associativity}\label{subsec: associativity}
We will now prove the associativity of $\ast$. The argument is standard, but requires somewhat cumbersome bookkeeping.

\begin{theorem}\label{thm: associativity}
    The $K$-Hall algebra $\BK(\CA)$ is an associative algebra.
\end{theorem}
\begin{proof}
We fix
\[\varphi_\alpha\in K_\ast(\FM^\rig_\alpha)\,,\,\varphi_\beta\in K_\ast(\FM^\rig_\beta)\,,\,\varphi_\gamma\in K_\ast(\FM^\rig_\gamma)\,.\]
We will use $\alpha, \beta, \gamma$ as indices in the different objects involved in the definition of $\ast$ to keep track of which copy we refer to. For example, $\Sigma_{\alpha, \beta}$ refers to the restriction of the direct sum map to $\FM_\alpha\times \FM_\beta\to \FM_{\alpha+\beta}$ and $\Gamma_{\alpha, \beta}$ refers to the restriction of $\Gamma_-$ to $\FM_\alpha\times \FM_\beta$ or $(\FM_\alpha\times \FM_\beta)^\rig$; we ease the notation by writing only $\Sigma_{\alpha, \beta}$ instead of $\Sigma^\rig_{\alpha, \beta}$. To calculate $\varphi_\alpha\ast(\varphi_\beta\ast \varphi_\gamma)$ we consider the following diagram, where the square is cartesian:
\begin{center}
\begin{tikzcd}
    (\FM_\alpha\times \FM_\beta\times \FM_\gamma)^\rig\arrow[r, "\overline \pi_{\alpha, \beta+\gamma}"]  \arrow[dd, bend right=70, swap,"\Sigma_{\alpha, \beta, \gamma}"]\arrow[rr, bend left=20, "\pi_{\alpha, \beta, \gamma}"]\arrow[d, "\overline \Sigma_{\beta, \gamma}"]&
    \FM_\alpha^\rig\times (\FM_\beta\times \FM_\gamma)^\rig \arrow[r, "\pi_{\beta,\gamma}"]\arrow[d, "\Sigma_{\beta, \gamma}"]&
    \FM_\alpha^\rig\times \FM_\beta^\rig\times \FM_\gamma^\rig\\
     (\FM_\alpha\times \FM_{\beta+\gamma})^\rig \arrow[r, "\pi_{\alpha, \beta+\gamma}"] \arrow[d, "\Sigma_{\alpha, \beta+\gamma}"]&
     \FM_\alpha^\rig\times \FM_{\beta+\gamma}^\rig & \\
     \FM_{\alpha+\beta+\gamma}^\rig
\end{tikzcd}
\end{center}
The map $\Sigma_{\alpha, \beta, \gamma}$ is the triple direct sum map $\FM_{\alpha}\times \FM_\beta\times \FM_\gamma\to \FM_{\alpha+ \beta+ \gamma}$ (or its rigidification), and by associativity of the direct sum it can be written as
\[\Sigma_{\alpha, \beta, \gamma}\coloneqq \Sigma_{\alpha, \beta+\gamma}\circ \Sigma_{\beta,\gamma}=\Sigma_{\alpha+\beta, \gamma}\circ\Sigma_{\alpha, \beta}\,.\]
The map $\pi_{\alpha, \beta, \gamma}$ is the rigidification with respect to the natural $BT$ acting on  $(\FM_\alpha\times \FM_\beta\times \FM_\gamma)^\rig$ -- more precisely, $T$ is the 2-dimensional torus obtained by modding out the diagonal $\BG_m$ from $\BG_m^3$. We have the associativity type relation
\[\pi_{\alpha, \beta, \gamma}\coloneqq \pi_{\beta, \gamma}\circ \overline \pi_{\alpha, \beta+\gamma}=\pi_{\alpha, \beta}\circ \overline \pi_{\alpha+\beta, \gamma}\,.\]

Denote $\varphi_{\alpha\beta\gamma}=\varphi_\alpha\boxtimes\varphi_\beta\boxtimes \varphi_\gamma\in K_\ast(\FM_\alpha^\rig\times \FM_\beta^\rig\times \FM_\gamma^\rig)$. Then
\begin{align}\label{eq: proofassociativity}
    \varphi_\alpha\ast(\varphi_\beta\ast \varphi_\gamma)&=(\Sigma_{\alpha, \beta+\gamma})_\ast\left(\Gamma_{\alpha, \beta+\gamma}\otimes \pi_{\alpha, \beta+\gamma}^\ast(\Sigma_{\beta, \gamma})_\ast\left(\Gamma_{\beta, \gamma}\otimes \pi_{\beta, \gamma}^\ast(\varphi_{\alpha\beta\gamma})\right)\right)\\
    &\nonumber =(\Sigma_{\alpha, \beta+\gamma})_\ast\left(\Gamma_{\alpha, \beta+\gamma}\otimes (\overline\Sigma_{\beta, \gamma})_\ast\overline\pi_{\alpha, \beta+\gamma}^\ast\left(\Gamma_{\beta, \gamma}\otimes \pi_{\beta, \gamma}^\ast(\varphi_{\alpha\beta\gamma})\right)\right)\\
    &\nonumber =(\Sigma_{\alpha, \beta, \gamma})_\ast \left(\overline \Sigma_{\beta, \gamma}^\ast(\Gamma_{\alpha, \beta+\gamma})\otimes \overline\pi_{\alpha, \beta+\gamma}^\ast(\Gamma_{\beta, \gamma})\otimes \pi_{\alpha, \beta, \gamma}^\ast(\varphi_{\alpha\beta\gamma})\right)
\end{align}

Note that, by the properties of the Ext complex, we have
\begin{align*}\Ext_{<}&\coloneqq \Ext_{\alpha, \beta}+\Ext_{\alpha, \gamma}+\Ext_{\beta, \gamma}=\Sigma_{\beta,\gamma}^\ast(\Ext_{\alpha, \beta+\gamma})+\Ext_{\beta, \gamma}\\
\Ext_{>}&\coloneqq \Ext_{\beta, \alpha}+\Ext_{\gamma, \alpha}+\Ext_{\gamma, \beta}=\Sigma_{\beta,\gamma}^\ast(\Ext_{\beta+\gamma, \alpha})+\Ext_{ \gamma, \beta}
\end{align*}
on $\FM_{\alpha}\times \FM_\beta\times \FM_\gamma$, and therefore
\[\overline \Sigma_{\beta, \gamma}^\ast(\Gamma_{\alpha, \beta+\gamma})\otimes \overline\pi_{\alpha, \beta+\gamma}^\ast(\Gamma_{\beta, \gamma})=\Lambda_{-1}\big(\Ext_{<}+\Ext_{>}^\vee\big)\otimes \det(\Ext_{<})^\vee[\rk \Ext_{<}]\,,\]
which we can plug in \eqref{eq: proofassociativity}. An entirely analogous calculation arrives at the same expression for $(\varphi_\alpha\ast \varphi_\beta)\ast \varphi_\gamma$. 
\end{proof}

The proof of associativity provides a formula for the $n$-fold product. To state it, we introduce the following notation:
\begin{enumerate}
    \item $\Sigma_{[n]}$ is the $n$-fold direct sum map $\FM^{\times n}\to \FM$.
    \item Denote by $(\FM^{\times n})^\rig$ the diagonal rigidication of $\FM^{\times n}$, and consider the rigidification 
    \[\pi_{[n]}\colon (\FM^{\times n})^\rig \to (\FM^\rig)^{\times n}\,\]
    with respect to the remaining $(n-1)$-dimensional torus.
    \item Define the complex $\Gamma_{[n], -}$ on $(\FM^{\times n})^\rig$ by
    \[\Gamma_{[n], -}\coloneqq \Lambda_{-1}\big(\Ext_{<}+\Ext_{>}^\vee\big)\otimes \det(\Ext_{<})^\vee[\rk_{<}]\,,\]
    where 
    \[\Ext_{<}=\sum_{1\leq i<j\leq n}\Ext_{ij}\quad\textup{ and }\quad\Ext_{>}=\sum_{1\leq j<i\leq n}\Ext_{ij}\,\]
    and $\Ext_{ij}$ is the pullback of the complex $\Ext_{12}$ along the projection $\FM^{\times n}\to \FM^{\times 2}$ onto the $i$-th and $j$-th component.
\end{enumerate}

\begin{proposition}\label{prop: nproduct}
    We have
\[\phi_1\ast \ldots\ast\phi_n=\big(\Sigma^\rig_{[n]}\big)_\ast\big([\Gamma_{[n], -}]\cap \pi_{[n]}^\ast (\phi_1\boxtimes \ldots\boxtimes \phi_n)\big)\,.\]
\end{proposition}
\begin{proof}
    The $n=3$ case is shown along the proof of associativity. The general case can be shown by induction on $n$ and the same calculation as in the proof of associativity.\qedhere
\end{proof}

\subsection{Generalized $K$-theoretic invariants}
\label{subsec: epsiloninvariants}
In analogy with the generalized motivic Donaldson--Thomas invariants, we use the $K$-Hall algebra product to define generalized $K$-theoretic invariants as a formal logarithm of the $K$-theoretic invariants constructed directly from the stacks of semistable objects.

\begin{definition}\label{def: K-invariants}
    Let $\CA$ be a good abelian category, let $\mu$ be a stability condition as in Assumption \ref{ass: stability} and let $\alpha\in C(\CA)_\pe$. Then by Theorem~\ref{thm: deltainvariants} and Theorem~\ref{thm: goodproper}, we have classes
    \[\delta_\alpha^\mu\coloneqq j_\ast[\FM_{\alpha}^{\mu, \rig}]\in \BK(\CA)\,\]
    where $j\colon \FM_\alpha^{\mu, \rig}\to \FM_\CA^\rig$ is the open inclusion, and $j_\ast$ the pushforward in $K$-homology.
    We then define the generalized $K$-theoretic invariants $\varepsilon_\alpha^\mu\in \BK(\CA)$ by
      \begin{equation}
      \label{eq: epsilonfromdelta}\varepsilon_\alpha^\mu\coloneqq \sum_{k\geq 1}\frac{(-1)^{k-1}}{k}\sum_{\substack{\alpha_1+\ldots+\alpha_k=\alpha\\\mu(\alpha_i)=\mu(\alpha)}}\delta^\mu_{\alpha_1}\ast\ldots \ast\delta^\mu_{\alpha_{k}}\,.
      \end{equation}
\end{definition}

The relation expressing $\varepsilon$ classes as a ``logarithm'' of $\delta$ classes can be formally inverted to express $\delta$ classes as an ``exponential'' of $\varepsilon$ classes:
\begin{equation}
      \label{eq: deltafromepsilon}
  \delta_\alpha^\mu=\sum_{k\geq 1}\frac{1}{k!}\sum_{\substack{\alpha_1+\ldots+\alpha_k=\alpha\\\mu(\alpha_i)=\mu(\alpha)}}\varepsilon^\mu_{\alpha_1}\ast\ldots \ast\varepsilon^\mu_{\alpha_{k}}\,.\end{equation}

Note that parts (5) and (4) of Assumption \ref{ass: stability} guarantee that the sums in the previous formulas are finite, and that $\alpha_i$ are also permissible, so that $\delta_\alpha^\mu$ is defined. If there is no non-trivial decomposition $\alpha_1+\ldots+\alpha_k=\alpha$ with $\mu(\alpha_i)=\mu(\alpha)$, then every object in $\FM_{\alpha}^\mu$ is stable and hence $M_\alpha^\mu=\FM_\alpha^{\mu, \rig}$ is the good moduli space of~$\FM_{\alpha}^\mu$. In this case, $\varepsilon_\alpha^\mu=\delta_\alpha^\mu=[M_\alpha^\mu]$ encode the usual $K$-theoretic invariants of the good moduli space.

\begin{remark}
    When $\FM_\alpha$ is itself finite type and quasi-smooth for every $\alpha$ (i.e., $\CA$ has homological dimension $\leq 2$), then we might take the trivial stability condition in Definition \ref{def: K-invariants} and obtain invariants $\delta_\alpha, \varepsilon_\alpha$ which are intrinsically associated to the abelian category $\CA$. If $\mu$ is a stability condition and for some $s\in T$ all classes with fixed slope $\mu(\alpha)=s$ are permissible, then we might think of the invariants $\delta_\alpha^\mu, \varepsilon_\alpha^\mu$ as being intrinsically attached to the exact subcategory $\CA^\mu_s\subseteq \CA$ of $\mu$ semistable objects with fixed slope $s$.
\end{remark}

We will show later that, in general, the classes $\varepsilon_\alpha^\mu$ are precisely the $K$-theoretic Joyce--Liu classes, denoted by $Z_\alpha(\mu)$ in \cite{Liu}, when the latter are defined. Note that the definition of Joyce--Liu classes is only available when there is a framing functor, and it is a hard theorem that they are actually independent of the choice of framing functor (cf. \cite[Proposition 9.12]{joyce} and \cite[Theorem 4.2.5]{Liu}). On the other hand, in our approach we do not need a framing functor, which extends the definition to new cases -- for example, moduli spaces of (weak) Bridgeland stable objects on Fano surfaces or 3-folds -- and makes the independence of the framing functor obvious. We see this as one of the strengths of our new approach. For example, with our definition the following fact is basically trivial:

\begin{proposition}\label{prop: autoequivalence}
    Let $B\colon \CA\isoto\CA'$ be an isomorphism of abelian categories and let $\mu$ be a stability condition on $\CA$ as in Assumption \ref{ass: stability}. Let $B^\ast\mu$ be the stability condition on $\CA'$ defined by $B^\ast\mu(\alpha')=\mu(B^{-1}(\alpha'))$. Then
    \[B_\ast \delta_\alpha^\mu=\delta_{B(\alpha)}^{B^\ast\mu}\quad\textup{ and }\quad B_\ast \varepsilon_\alpha^\mu=\varepsilon_{B(\alpha)}^{B^\ast\mu}\,.\]
\end{proposition}
\begin{proof}
    The equality for the $\delta$ classes is trivial once we notice that $B(E)$ is $B^\ast\mu$ semistable if and only if $E$ is $\mu$ semistable. Then the equality for $\varepsilon$ classes follows from the definition and the observation that $B_\ast\colon \BK(\CA)\to \BK(\CA')$ is an algebra homomorphism.
\end{proof}

A typical scenario where these type of results might be useful is when $B$ is an automorphism of a derived category, $\CA$ is the heart of a $t$-structure and $\CA'$ the image of $\CA$ by $B$. In the Calabi--Yau setting, automorphisms of derived categories have often been used to constrain enumerative invariants, see for example \cite{todaPTrationality, OS, BM, thomasrefinedK3}.

In contrast with the trivial proof given above, in Joyce's setting these results might be quite subtle. Since the definition of invariants depends on a framing functor, equalities as in Proposition \ref{prop: autoequivalence} are only immediate if we choose compatible framing functors on both sides. So a general statement requires the hard result that the choice of framing functor is actually not important; that requires \cite[Assumptions 5.1(g), 5.2(h)]{joyce}, which could potentially be difficult to verify in examples.

\subsection{Example: $\mathsf{Vect}$}\label{subsec: exampleVect} We consider now the simplest possible abelian category, the category of vector spaces $\CA=\mathsf{Vect}$. The stack of vector spaces is
\[\FM_\mathsf{Vect}=\bigsqcup_{n\geq 0} B\GL_n\]
where $n\in C(\mathsf{Vect})=\BZ_{\geq 0}$ corresponds to the dimension of the vector space. Its rigidification is
\[\FM_\mathsf{Vect}^\rig=\bigsqcup_{n\geq 0} B\PGL_n\,.\]
Note that $K^\ast(B\PGL_n)$ is the representation ring of $\PGL_n$, and $K_\ast(B\PGL_n)$ is the dual to the representation ring. Given a representation $V$ of $\PGL_n$, the Euler characteristic of the corresponding vector bundle on $B\PGL_n$ is 
\[\chi(B\PGL_n, V)=\dim(V^{\PGL_n})\,.\]
This follows from the fact that representations of $\PGL_n$ are completely reducible and Whitehead's lemma, which states that $H^i(\PGL_n, V)=0$ for any $i$ and any non-trivial irreducible representation $V$. So the functional $\delta_n$ is the functional that sends a representation to the dimension of its fixed part or, in other words, to the multiplicity of the trivial representation in its decomposition into irreducible representations. For $n=1$ the group $\PGL_1$ is trivial and $\varepsilon_1=\delta_1$ is the isomorphism $K^\ast(B\PGL_1)\isoto \BZ$ sending $V$ to $\dim(V)$.

In the following proposition we show that $\varepsilon_n=0$ for $n>1$. Note that in Joyce's \cite{joyce} approach to cohomological invariants the analogous statement is trivial since the class he defines has negative homological degree $2(1-n^2)$. However, from our perspective this is a non-trivial statement. One way to prove it is via the connection to the homological invariants which we will show later in Theorem \ref{thm: homliftsjoyce}. We present here a direct proof using the Weyl character formula. 

\begin{proposition}\label{prop: epsilonVect}
    If $\CA=\mathsf{Vect}$ then $\varepsilon_n=0$ for $n>1$. In other words, we have
    \[\dim(V^{\PGL_n})=\frac{1}{n!}(\underbrace{\varepsilon_1\ast\ldots \ast\varepsilon_1}_{ n\textup{ times}})(V)\]
    for every representation $V$ of $\PGL_n$.
\end{proposition}

\begin{proof}
Let $\rho^\ast V$ be the induced representation of $\GL_n$ for which the diagonal $\BG_m$ acts trivially. Let $T=(B\BG_m)^{\times n}$ be the maximal torus of $\GL_n$ and let 
\[\Sigma_{[n]}\colon \FM_1^{\times n}=BT\to B\GL_n=\FM_n\]
be the induced map. By Proposition \ref{prop: nproduct} and the discussion in Section \ref{subsec: K-Hal-prod} we have 
\[(\varepsilon_1\ast\ldots \ast\varepsilon_1)(V)=[u_1^0\ldots u_n^0]\big(\Gamma_{[n], -}\otimes \Sigma^\ast_{[n]}\rho^\ast V\big)\]
where 
\[\Gamma_{[n], -}=\prod_{1\leq i\neq j\leq n}\left(1-\frac{u_i}{u_j}\right)\in \BZ[u_1^{\pm 1}, \ldots, u_n^{\pm 1}]\simeq K^\ast(BT)\,;\]
note that $\Ext_{<}$ is $\sum_{1\leq i<j\leq n}u_j/u_i$ and similarly for $\Ext_{>}$, and recall Remark \ref{rmk: eulerclassTSigma}. Note also that 
\[\Sigma^\ast \rho^\ast V=\sum_{\mu\in \BZ^n} \dim(V_\mu)u^\mu \in \BZ[u_1^{\pm 1}, \ldots, u_n^{\pm 1}]\]
is the character of the representation $V$. Recall that the root system of $\fsl_n$ is $\Delta=\{e_i-e_j\}_{1\leq i\neq j\leq n}$, where $e_1, \ldots, e_n$ is the standard basis of $\Ft^\vee$. Therefore, we can rewrite $\Gamma_{[n], -}$ more conceptually as
\[\Gamma_{[n], -}=\prod_{\alpha\in \Delta}(u^{\alpha/2}-u^{-\alpha/2})\,.\]
Note that this is almost the Weyl denominator, except that we take the product over all roots and not just the positive ones. The proof is now concluded using the following lemma for general semisimple Lie algebras.
\end{proof}

\begin{lemma}\label{lem: weylcharacter}
Let $\Fg$ be a semisimple Lie algebra and $L$ an irreducible representation of $\Fg$. Then we have
\[[u^0]\left(\prod_{\alpha\in \Delta}(u^{\alpha/2}-u^{-\alpha/2})\cdot \ch(L)\right)=\begin{cases}
    |W|&\textup{ if }L\textup{ is the trivial representation}\\
    0& \textup{ otherwise}
\end{cases}\]
where $\Delta\subseteq \Ft^\vee$ is the root system of $\Fg$, $W$ the Weyl group, and $\ch(L)\in \BZ[\Ft^\vee]$ is the character of $L$.
\end{lemma}
\begin{proof}
    Choose a set of positive roots $\Delta_+$ and let $\Delta_-=\Delta\setminus \Delta_+$ be the negative roots. Let $\lambda$ be the highest weight vector of $L$, which is integral and dominant. We recall the Weyl character formula:
    \[\ch(L)=\frac{\sum_{\sigma\in W}\sgn(\sigma)u^{\sigma(\lambda+\rho)}}{\prod_{\alpha\in \Delta_+}\big(u^{\alpha/2}-u^{-\alpha/2}\big)}\]
    where $\rho=\frac{1}{2}\sum_{\alpha\in \Delta_+}\alpha$ is the Weyl vector. Using this and the Weyl denominator formula for the negative roots, we have
   \begin{align*} \prod_{\alpha\in \Delta}(u^{\alpha/2}-u^{-\alpha/2})\cdot \ch(L)&=\prod_{\alpha\in \Delta_-}(u^{\alpha/2}-u^{-\alpha/2})\left(\sum_{\sigma\in W}\sgn(\sigma)u^{\sigma(\lambda+\rho)}\right)\\
   &=\left(\sum_{\sigma'\in W}\sgn(\sigma')u^{-\sigma'(\rho)}\right)\left(\sum_{\sigma\in W}\sgn(\sigma)u^{\sigma(\lambda+\rho)}\right)\\
   &=\sum_{\sigma, \sigma'\in W}\sgn(\sigma\sigma') u^{\sigma(\lambda+\rho)-\sigma'(\rho)}\,.
   \end{align*}
   If $\lambda=0$ then the terms with $\sigma=\sigma'$ are precisely the ones contributing to the constant coefficient, which shows what we wanted. We now argue that the constant coefficient is equal to 0 for $\lambda\neq 0$ by showing that there are no $\sigma, \sigma'$ such that $\sigma(\lambda)=\sigma'(\rho)-\sigma(\rho)$; without loss of generality we may assume that $\sigma'=\id$, so suppose that $\sigma(\lambda)=\rho-\sigma(\rho)$. Since $\lambda$ is dominant we have
   \[0\leq \langle \rho, \lambda\rangle=\langle \sigma(\rho), \sigma(\lambda)\rangle=\langle \sigma(\rho),\rho\rangle-\lvert\lvert\sigma(\rho)\rvert\rvert^2\leq 0\]
   where the last inequality follows from Cauchy-Schwarz, and we have used twice that $\sigma$ is an isometry. Hence $\langle \rho, \lambda\rangle =0$, which implies that $\langle \alpha, \lambda\rangle=0$ for every positive root $\alpha$ and therefore $\lambda=0$.\qedhere
\end{proof}

\begin{remark}
   Our $K$-Hall algebra admits a generalization to arbitrary stacks in the spirit of ``intrinsic Donaldson--Thomas theory'' from \cite{intrinsicDT1, intrinsicDT2}, which will be pursued elsewhere.  The fact that Lemma \ref{lem: weylcharacter} holds for a general semisimple Lie algebra translates to a vanishing statement of $\varepsilon$ invariants for stacks $BG$. 
\end{remark}

\begin{example}\label{ex: deltanotregular}
   In \cite{Liu}, the author introduces a slightly different notion of $K$-homology where he requires a finiteness condition, see Definition \ref{def: regularKhom} and the discussion that follows it. Let us check that $\varepsilon_1\ast \varepsilon_1$, or equivalently $\delta_2$, does not satisfy the finiteness condition. Let $V$ be the canonical representation of $\PGL_2$, which has weight $(1, -1)$. Then
\[(\varepsilon_1\ast\varepsilon_1)_\pt\big((V-2)^N\big)=[u^0]\big((2-u-u^{-1})(u+u^{-1}-2)^N\big)=(-1)^N\binom{2N+2}{N+1}\]
   is non-zero for every $N$. Since $\rk(V)=2$, this shows that the finiteness condition~\eqref{eq: finiteness} does not hold. In particular, $\delta_2$ is not in the image of the pushforward along a map $Z\to B\PGL_2$ from a finite type scheme $Z$.
\end{example}

\section{Non-abelian localization and the wall-crossing formula}
\label{sec: NAL}

In this section we state and prove the wall-crossing formula for $\delta$ and $\varepsilon$ invariants. The fundamental tool is the virtual non-abelian localization theorem of Halpern--Leistner \cite{HLremarks, HLNAL}. 

\subsection{$\Theta$-stratifications}
\label{subsec: Theta}
The non-abelian localization is stated in the language of derived $\Theta$-stratifications, which we briefly introduce now following \cite[Section~1]{HLderived}; the reader can refer to loc. cit. for details on this topic.

Let $\mathfrak X$ be a derived algebraic stack over $\mathbb C$ which is locally of finite type. We will mostly be interested in the case of $\mathfrak X$ being the moduli stack of objects in some abelian category, and we specialize the discussion to that case in Section~\ref{subsec: stratificationabelian}.

Now consider $\Theta\coloneqq \mathbb A^1/\mathbb G_m$, and the derived stacks:

\begin{enumerate}
    \item[i)] $\operatorname{Grad}(\mathfrak X) \coloneqq  \operatorname{Maps}(B\mathbb G_m, \FX)$ which is called the \textit{stack of the graded points of~$\mathfrak X$};
\item[ii)] $\operatorname{Filt}(\mathfrak X) \coloneqq \operatorname{Maps}(\Theta, \mathfrak X)$, which is called the \textit{stack of the filtered points of~$\mathfrak X$}.
\end{enumerate} 

We have the following canonical maps of derived stacks:
\begin{equation}\label{eq: filtgraddiagram}
    \begin{tikzcd}
       \Filt(\FX)\arrow[rr,"\ev_1"]\arrow[d, "\gr"', bend right] & & \FX\\
        \Grad(\FX)\arrow[u, "\sigma"', bend right]\arrow[urr, "\Sigma"']&
    \end{tikzcd}
\end{equation}
The canonical morphism $\gr$ is the restriction via the inclusion $0/\mathbb G_m \to \Theta$; the evaluation map $\operatorname{ev}_1$ is given by restricting the map to the open substack $(\mathbb A^1 \setminus \{0\})/\BG_m$ inside $\Theta.$ The morphism $\sigma$ is a section of $\gr$, i.e. $\gr\circ \sigma=\id_\Grad$, and is induced by the projection $\Theta\to B\BG_m$. Finally, $\Sigma$ is the composition $\ev_1\circ \sigma$.

\begin{definition}[$\Theta$-stratum] A \textit{derived $\Theta$-stratum} inside $\mathfrak X$ is a union of connected components $\mathfrak S \subset \Filt(\FX)$ so that the morphism $\ev_1: \mathfrak S \to \mathfrak X$ is a closed immersion.
\end{definition}

\begin{definition}[$\Theta$-stratification] Let $(\Gamma, \leq)$ be a totally ordered set with a minimal element $0$. A derived $\Theta$-stratification of $\FX$ indexed by $\Gamma$ consists of:
\begin{enumerate}
    \item a collection of open derived substacks $\FX_{\leq c}$ for $c \in \Gamma$, such that $\mathfrak X_{\leq c} \subseteq \FX_{\leq c'}$ for $c \leq  c'$;
\item  for each $c\in \Gamma$, a derived $\Theta$-stratum $\FS_c\subseteq \Filt(\mathfrak \FX_{\leq c})$ such that 
\[\mathfrak X_{\leq c} \setminus \operatorname{ev}_1(\mathfrak S_c) = \bigcup \limits_{c'<c} \mathfrak X_{\leq c'}\,.\]
\item for every point $x$ in $\mathfrak X$, the set $\{c \in \Gamma \ | \ x \in \FX_{\leq c}\}$ admits a minimal element.
\end{enumerate}
\end{definition}

A consequence of (3) is that the open sets $\FX_{\leq c}$ exhaust $\FX$. The open set $\FX_{\leq 0}$ is called the semistable locus of the $\Theta$-stratification, and denoted by $\FX^{\textup{ss}}$; its complement is called the unstable locus. We will use the notation
\[\FX_{=c}\coloneqq \ev_1(\FS_c)\,,\quad \FX_{<c}\coloneqq \FX_{\leq c}\setminus \FX_{=c}\overset{(2)}{=}\bigcup \limits_{c'<c} \mathfrak X_{\leq c'}\,.\]

By the definition of $\Theta$-stratification, every point in $\FX$ comes with a distinguished filtration $f\in \FS_c\subseteq \Filt(\FX)$ such that $\ev_1(f)=x$. This distinguished filtration is called the \textit{Harder-Narasimhan filtration} of $x$; in examples, they turn out to be the Harder-Narasimhan filtrations in the usual sense.

%\begin{lemma}\cite[Lemma 2.1.4]{HLstruc} Let $\mathfrak M$ be a stack as above. Let $\mathfrak M_{\leq \alpha}$ be a (weak) $\Theta$-stratification of $\mathfrak M$. Then, for any unstable point $p \in |\mathfrak M|$, there is a unique $c \in \Gamma$, and a point $f \in |\mathfrak S_c|$ so that $p \in |\mathfrak M_{\leq \alpha}|$, and $\operatorname{ev}_1(f) = p$.
%\end{lemma}

Halpern-Leistner also introduces a notion of the \textit{tcenter of a stratum}, which plays an important role for us:

\begin{definition}[Center] For a derived $\Theta$-stratum $\FS$ in $\mathfrak X$, one defines its center $\FZ$ as 
\[\FZ=\sigma^{-1}(\FS)\subseteq \Grad(\FX)\,.\] 
\end{definition}
Note that the center $\FZ$ is itself a union of connected components of $\Grad(\FX)$ (connected components of $\Grad(\FX)$ and $\Filt(\FX)$ are in bijection, c.f. \cite[Lemma 1.3.8]{HLstructure}) and $\FS=\gr^{-1}(\FZ)$.

Joyce defined in \cite[Section 3.3.4]{joyce} a slight weakening of the notion of $\Theta$-stratification, which he calls pseudo $\Theta$-stratification. 

\begin{definition}[Pseudo $\Theta$-stratification]
 Let $(\Gamma, \leq)$ be a partially ordered set with a minimal element $0$ such that every $c, c'\in \Gamma$ have a greatest lower bound, denoted by $c\vee c'$. A derived pseudo $\Theta$-stratification of $\FX$ indexed by $\Gamma$ consists of:
\begin{enumerate}
    \item a collection of open derived substacks $\FX_{\leq c}$ for $c \in \Gamma$, such that $\FX_{\leq c}\cap \FX_{\leq c'}=\FX_{\leq (c\vee c')}$. In particular, $\mathfrak X_{\leq c} \subseteq \FX_{\leq c'}$ for $c \leq  c'$;
\item  for each $c\in \Gamma$, there exists a $\Theta$-stratum $\FS_c\subseteq \Filt(\mathfrak \FX_{\leq c})$ such that 
\[\mathfrak X_{\leq c} \setminus \operatorname{ev}_1(\mathfrak S_c) = \bigcup \limits_{c'<c} \mathfrak X_{\leq c'}\,.\]
\item for every point $x$ in $\mathfrak X$, the set $\{c \in \Gamma \ | \ x \in \FX_{\leq c}\}$ admits a minimal element.
\end{enumerate} 
\end{definition}

The two main differences between a pseudo $\Theta$-stratification and a $\Theta$-stratification are that we no longer require the indexing set $\Gamma$ to be totally ordered, and a $\Theta$-stratum $\FS_c$ is no longer part of the data, we just require that it exists. This is a fairly mild difference and most of results concerning $\Theta$-stratifications also apply to pseudo $\Theta$-stratification. In some cases, it is easier and more natural to construct a pseudo-$\Theta$ stratification, so we will formulate our results in that language.

\subsection{The non-abelian virtual localization theorem}
\label{subsec: NAloc}

Let $\FX$ be a quasi-smooth derived stack which admits a $\Theta$-stratification. For simplicity, and because it is enough for our purposes, we will assume that the $\Theta$-stratification is finite (i.e. $\Gamma$ is finite) and that $\FX$ admits a proper good moduli space, as well as all the centers~$\FZ_c$.\footnote{As explained in \cite{HLNAL}, it is possible in some cases to allow $\FX$ to be non finite-type (e.g. the moduli of all bundles on a curve) and the stratification to be infinite. One of the remarkable aspects of the theorem, which we do not explore at all here, is that sometimes $\chi(\FX, V)$ can be defined even in that setting.}

Note that the stack of graded points comes with a canonical $B\BG_m$ action. This action preserves the center $\FZ_c\subseteq \Grad(\FX_{\leq c})$ of $\FS_c$. We regard the center $\FZ_c$ as a substack of $\FX$ via the locally closed immersion $\Sigma\colon \FZ_c\to \FX_{\leq c}\subseteq \FX$. We consider now the decomposition
\[\BT_\FX|_{\FZ_c}=\BT_c^-\oplus\BT_c^0\oplus \BT_c^+\]
of the restriction of the derived tangent bundle to $\FZ_c$ into its negative/zero/positive weight parts, according to the canonical $B\BG_m$ action. 

\begin{remark}\label{rmk: interpretationL-+}
The complexes $\BT_c^{-/0/+}$ have natural interpretations in terms of the derived normal bundles of the immersions in \eqref{eq: filtgraddiagram}. By \cite[Lemma 1.3.2, Lemma 1.5.5]{HLderived} we have
\begin{align*}
    \BT_c^0&\simeq \BT_{\FZ_c}\\
    \BT_c^-&\simeq \sigma^\ast\BT_{\FS_c/\FX}[1]\\
    \BT_c^+&\simeq \BT_{\FZ_c/\FS_c}[1]
\end{align*}
\end{remark}

For $c\in \Gamma$ we introduce the following complex on $\FZ_c$:
\begin{equation}\label{eq: complcorr}E_c=\Sym \big((\mathbb T_{c}^+)^\vee \oplus \mathbb T_{c}^-\big) \otimes \operatorname{det}(\mathbb T_{c}^-)^{\vee}[-\operatorname{rk} \mathbb T_{c}^-].\end{equation}
Note that $E_c$ is not perfect or coherent in general, but it has the property that its weight $\mu$ part is coherent for every $\mu$ and vanishes for $\mu\gg 0$; see the discussion in Section \ref{subsec: K-Hal-prod} regarding $\Gamma_-$.

\begin{theorem}[{\cite[Corollary 3.16]{HLNAL}}, Virtual non-abelian localization]\label{thm: NAloc} Let $\FX$ be a finite type quasi-smooth derived stack admitting a derived $\Theta$-stratification indexed by a finite set~$\Gamma$. Assume that $\FX$, $\FZ_c$ have proper good moduli spaces.

Then, for any $V \in K^\ast(\FX)$, the following formula holds:
\begin{equation}\label{eq: NAloc}\chi(\FX, V) = \chi(\mathcal \FX^{\textup{ss}}, V|_{\mathcal \FX^{ss}}) + \sum_{c\in \Gamma\setminus \{0\}} \chi(\FZ_{c}, V|_{\FZ_{c}} \otimes E_{c})\end{equation}
\end{theorem}

Note that $\FX^{\textup{ss}}=\FZ_0$ and $E_0=\CO_{\FX^{\textup{ss}}}$, so on the right hand side we could have written just a sum over all $c\in \Gamma$. All the terms are well defined by Theorem \ref{thm: deltainvariants}, since $\FX$ being quasi-smooth implies that the centers $\FZ_c$ are also quasi-smooth by Remark \ref{rmk: interpretationL-+} (see \cite[Lemma 3.1.4]{HLderived} for further details). This holds even though $E_c$ is not coherent by the same argument as in the proof of Lemma \ref{lem: astwelldefined}.

\begin{remark}\label{rmk: relativeNAL}
For the wall-crossing formula formulated in operational $K$-homology we will also use the following relative version of the non-abelian localization, as in \cite[Proposition 6.10]{HLHgauged}. If $\FX$ satisfies the conditions of Theorem \ref{thm: NAloc}, $S$ is some derived stack and $V\in K^\ast(\FX\times S)$ then we have an equality
\[(\FX\times S\to S)_\ast V=\sum_{c\in \Gamma}(\FZ_c\times S\to S)_\ast (V|_{\FX_c\times S}\otimes E_c) \]
in $K^\ast(S)$.
\end{remark}

\begin{remark}\label{rmk: pseudoNAL}
Theorem \ref{thm: NAloc} still holds for pseudo $\Theta$-stratifications, as long as we pick some strata and corresponding centers. The proof in \cite{HLNAL} shows the result when there is just one stratum and then inducts on the number of strata. To induct on the number of strata when $\Gamma$ is just partially ordered we take a maximal element $c\in \Gamma$ so that there is no $c'>c$. Then 
\[\FX\setminus \FX_{=c}=\bigcup_{c'\neq c}\FX_{=c'}=\bigcup_{c'\neq c}\FX_{\leq c'}\,,\]
so in particular $\FX_{=c}$ is closed in $\FX$ and $\FX\setminus \FX_{=c}$ inherits a pseudo $\Theta$-stratification indexed by $c'\in \Gamma\setminus \{c\}$ with the same centers.
\end{remark}

\subsection{(Pseudo) $\Theta$-stratifications on $\CM_\CA$}\label{subsec: stratificationabelian}
Let us now consider the case of stacks which come from a good abelian category; we start in the underived setting, and later consider the derived enhancement. It is shown in \cite[Proposition 7.12, Corollary 7.13]{AHLH} that the stacks of graded points and filtered points correspond, respectively, to $\BZ$-graded objects of $\CA$ and $\BZ$-filtered objects of $\CA$:
\begin{align}
\label{eq: gradedfiltpoints}    \Grad(\CM_\CA)&=\bigsqcup_{c, w}\CM_{\alpha_1}\times\ldots\times  \CM_{\alpha_n}\\
\nonumber    \Filt(\CM_\CA)&=\bigsqcup_{c, w}\mathcal{E}xt_{c, w}\,.
\end{align}
In either case, the union is over all possible choices of $n\geq 1$, $c=(\alpha_1, \ldots, \alpha_n)$ with $\alpha_i\in C(\CA)\setminus \{0\}$ and integer weights
\[w_1>\ldots>w_n\,,\]
and the stack $\mathcal{E}xt_{c, w}$ parametrizes $\BZ$-indexed (descending) filtrations
\[\ldots \subseteq \tilde E_{1}\subseteq \tilde E_0\subseteq \tilde E_{-1}\subseteq \ldots\]
such that $\tilde E_w=0$ for $w>w_1$ and $[\tilde E_w/\tilde E_{w+1}]=\alpha_i$ if $w=w_i$ and 0 otherwise; in particular, $\tilde E_w=E$ stabilizes for $w<w_n$ and $[E]=\alpha_1+\ldots+\alpha_n$. Note that, by setting $E_i=\tilde E_{w_i}$, the data of such a filtration is the same as the data of 
\[0=E_0\subsetneq E_1\subsetneq \ldots\subsetneq E_n=E\]
together with the choice of weights. In a similar way, a point in  $\Grad(\CM_\CA)$ is canonically described as an object $E$ together with a decomposition $E=\bigoplus_{w\in \BZ} \tilde F_w$, with $[\tilde F_{w}]=\alpha_i$ for $w=w_i$ and $0$ otherwise, but this is the same data as $E=F_1\oplus \ldots \oplus F_n$ and a choice of weights. The stacks of graded and filtered points on $\CM_\alpha$ are the union of components such that $\alpha_1+\ldots+\alpha_n=\alpha$. The map $\ev_0$ described in the previous section sends a filtration $\tilde E_\bullet$ to its associated graded $\bigoplus_{w\in \BZ} \tilde E_w/\tilde E_{w-1}$ and $\ev_1$ sends the filtration to the total object~$E$. The map $\sigma$ sends a graded point $\bigoplus_{w\in \BZ} \tilde F_w$ to the filtered point defined by $\tilde E_w=\bigoplus_{w'\geq w}\tilde F_{w'}$.

We will assume the existence of a pseudo $\Theta$-stratification adapted to $\mu$. Let 
\[\HN_\alpha(\mu)=\{(\alpha_1, \ldots, \alpha_n)\colon \alpha_1+\ldots+\alpha_n=\alpha \textup{ and }\mu(\alpha_1)>\ldots>\mu(\alpha_n)\}\]
be the set of possible $\mu$-HN types. Any object in $\CA$ has an associated $\mu$-HN type $c=(\alpha_1, \ldots, \alpha_n)$ where $\alpha_i=[E_i/E_{i-1}]$ are the types of its $\mu$-HN factors in \eqref{eq: HNfiltration}. 

\begin{definition}\label{def: stratificationadaptedto}
    We say that there is a pseudo $\Theta$-stratification of $\CM_\alpha$ adapted to $\mu$ if there is a partial ordering $\geq$ of $\HN_\alpha(\mu)$ which satisfies\footnote{In the right hand side of \eqref{eq: ineqHNtypes}, $\alpha_i+\alpha_j$ is put in a position that makes it a $\mu$-HN type, i.e. it is put between $\alpha_k$ and $\alpha_{k+1}$ such that $\mu(\alpha_k)>\mu(\alpha_i+\alpha_j)>\mu(\alpha_{k+1})$. In the case that $\mu(\alpha_k)=\mu(\alpha_i+\alpha_j)$ for some $k$, the right hand side should be interpreted as $(\ldots, \alpha_i+\alpha_k+\alpha_j, \ldots)$.}
     \begin{equation}\label{eq: ineqHNtypes}
    (\alpha_1, \ldots, \alpha_i, \ldots, \alpha_{j}, \ldots, \alpha_n)\geq (\alpha_1, \ldots, \alpha_{i}+\alpha_{j}, \ldots, \alpha_n)\, \end{equation}
   and a pseudo $\Theta$-stratification of $\CM_\alpha$ indexed by $(\HN_\alpha(\mu), \geq)$ for which the $\BC$-points of $\CM_{=c}^\mu$ correspond to objects of $\CA$ of $\mu$-HN type $c$.
\end{definition}

  Condition \eqref{eq: ineqHNtypes} should be interpreted as saying that if $[F_i]=\alpha_i$ then the $\mu$-HN type of $\bigoplus_{i=1}^nF_i$ is $\geq\!(\alpha_1, \ldots, \alpha_n)$, which is the form in which we will use this condition in the proof of Proposition \ref{prop: centers}. It implies in particular that $c=(\alpha)$ is the minimal element of $\HN_\alpha(\mu)$. Since $E$ is semistable if and only if its $\mu$-HN type is $(\alpha)$, the definition implies that $\mu$-semistability agrees with semistability in the sense of the $\Theta$-stratification.

We believe that in all cases of interest there is a natural choice of partial ordering making this hold. The main point is that the partial ordering should be chosen in a way that $\CM_{\leq c}^\mu$ is open. In some cases, it is possible to upgrade the pseudo-$\Theta$ stratification to a honest $\Theta$-stratification; see \cite[Theorem 2.2.2]{HLstructure} for some technical conditions that make this possible. Doing so amounts to choosing a total ordering, instead of a partial ordering, and a choice of weights for each $c$ (i.e., a choice of $\Theta$-stratum).

\subsubsection{Centers}\label{subsubsec: centers}

Given $c\in \HN_\alpha(\mu)$ and a choice of weights $w$, we have a $\Theta$-stratum 
\[\CS_{c,w}\subseteq \CE xt_{c,w}\subseteq \Filt(\CM_\alpha)\]
which is the preimage of $\FM_{=c}^\mu$ along the map $\ev_1\colon \CE xt_{c,w}\to \FM_\alpha$. The next proposition identifies the centrum $\CZ_{c,w}$ of this stratum:

\begin{proposition}\label{prop: centers} For every $c=(\alpha_1, \ldots, \alpha_n)\in \HN_\alpha(\mu)$ and choice of weights $w$, the center $\CZ_{c,w}$ is a product of stacks of $\mu$-semistable objects:
\[\CZ_{c,w}\simeq \CM_{\alpha_1}^\mu\times \ldots \CM_{\alpha_n}^\mu\,.\]
\end{proposition}
\begin{proof}
By definition, the center is a substack of the connected component of $\Grad(\CM_\alpha)$ which is isomorphic to $\CM_{\alpha_1}\times \ldots \times \CM_{\alpha_n}$. Let $E=\bigoplus_{i=1}^n F_i$ be a point in such connected component. By \cite[Corollary 2.3.6]{HLstructure}, $E=\bigoplus_{i=1}^n F_i$ is in the center if and only if the total object $E$ is in $\CM_{\leq c}^\mu$. If $F_1, \ldots, F_n$ are all semistable, then
\[0\subseteq F_1\subseteq F_1\oplus F_2\subseteq \ldots\subseteq \bigoplus_{i=1}^n F_i=E\]
is the $\mu$-HN filtration of $E$, so $E$ is indeed in $\CM_{\leq c}^\mu$. On the other hand, if not all the $\CM_{\alpha_i}$ are semistable, then \eqref{eq: ineqHNtypes} implies that the $\mu$-HN type of $E$ is $>c$, so $E=\bigoplus_{i=1}^n F_i$ is not in the center.\qedhere
\end{proof}

Since the strata $\CS_{c,w}$ and the centers $\CZ_{c,w}$ are identified for different $w$, we will drop it from the notation and just write $\CS_c=\CS_{c,w}$ and $\CZ_c=\CZ_{c,w}$. 

\subsubsection{Derived enhancement}\label{subsubsec: derivedenhancementabelian}

We explain now how to upgrade the statements above to the derived setting. The stacks of filtered and graded points of a derived stack inherit a natural derived enhancement (cf. \cite[Theorem 5.1.1]{HLpreygel} or \cite[Theorem 1.2.1]{HLderived}).

Recall from Section \ref{subsec: abcat} that the stack $\CM_\CA$ admits a derived enhancement $\FM_\CA$, which is obtained from the open embedding of $\CM_\CA$ into the classical truncation of $\FM_{\bfT}$, where $\bfT$ is a saturated dg category. By \cite[Proposition 8.26]{KPS}, the derived stacks of graded and filtered points on $\FM_{\bfT}$ admit descriptions analogous to \eqref{eq: gradedfiltpoints}; hence, the same is true for the derived stack $\FM_\CA$. 

The pseudo $\Theta$-stratification on $\CM_\CA$ considered in Definition \ref{def: stratificationadaptedto} induces a derived pseudo $\Theta$-stratification on $\FM_\CA$ by \cite[Lemma 1.2.3]{HLderived}. In particular, this means that for each $c\in \HN_\alpha(\mu)$ we have derived enhancements $\FS_c, \FZ_c$ of $\CS_c, \CZ_c$, and maps of derived stacks as follows:
\begin{center}
    \begin{tikzcd}
        \FS_c\arrow[r, "\sim"',"\ev_1"]\arrow[d, "\gr"', bend right]& \FM_{=c}^\mu\subseteq \FM_\alpha\\
        \FZ_c\arrow[u, "\sigma"', bend right]\arrow[ur, "\Sigma"']&
    \end{tikzcd}
\end{center}
By the previous considerations, the description of the center in Proposition \ref{prop: centers} can be enhanced to
\[\FZ_{c}\simeq \FM_{\alpha_1}^\mu\times \ldots \times \FM_{\alpha_n}^\mu\,.\]
Note that $\Sigma=\ev_1\circ \sigma$ is precisely the direct sum map.

Let us now identify the complex $E_c$ appearing in the non-abelian localization formula with the class $\Gamma_-$ in the definition of the $K$-Hall algebra. 

Denote by $\Ext_{ij}$ the pullback of $\Ext_{12}$ via the map $\FZ_c\to \FM_{\alpha_i}^\mu\times \FM_{\alpha_j}^\mu$. In particular, when $i=j$, the map above factors through the diagonal $\FM_{\alpha_i}^\mu\to \FM_{\alpha_i}^\mu\times \FM_{\alpha_i}^\mu$ and $\Ext_{ii}$ is the pullback of $\BT_{\FM_{\alpha_i}^\mu}[-1]$ to $\FZ_c$ by \eqref{eq: derivedtangentFM}. By bilinearity of the Ext complex, we have
\[\Sigma^\ast \BT_{\FM_\alpha}=\sum_{1\leq i,j\leq n}\Ext_{ij}[1]\,.\]

Given a weight $w$, the canonical $B\BG_m$ action on the center can be described as
\[B\BG_m\times \prod_{i=1}^n \FM_{\alpha_i}^\mu\xrightarrow{(-)^w}
\prod_{i=1}^n (B\BG_m\times\FM_{\alpha_i}^\mu)\xrightarrow{\Phi\times \ldots\times \Phi} \prod_{i=1}^n \FM_{\alpha_i}\,.\]
Here, $(-)^w\colon B\BG_m\to (B\BG_m)^n$ denotes the map given by $(-)^{w_i}$ in the $i$-th copy of $B\BG_m$. Note that $\Ext_{ij}$ has weight $w_j-w_i$ with respect to this action. Therefore, we have 
\[\BT_c^0=\Ext_=[1]=\BT_{\FZ_c}\,,\,\BT_{c}^{+}=\Ext_{>}[1]\,,\, \BT_{c}^{-}=\Ext_{<}[1]\,.\]
Observe that, although the $B\BG_m$ action depends on the choice of weights, the splitting of $\Sigma^\ast \BT_{\FM_\alpha}$ above into zero/positive/negative weights does not. Hence the complex $E_c$ is precisely
\begin{align}
    E_c&=\Sym(\Ext_{>}^\vee[-1]+\Ext_<[1])\otimes \det(\Ext_{<}[1])[\rk_<]\\
    &=\Lambda_{-1}\big(\Ext_>^\vee+\Ext_<\big)\otimes \det(\Ext_<)^\vee[\rk(\Ext_<)]=\Gamma_{[n],-} \nonumber
\end{align}

\subsection{Dominant wall-crossing formula} \label{subsec: dominantWC} We now prove the dominant wall-crossing formula, which compares invariants defined with respect to two stability conditions $\mu_0, \mu$ where $\mu_0$ dominates $\mu$. The typical scenario where a stability dominates another is when $\mu_0$ is on some wall in the space of stability conditions and $\mu$ is in an adjacent chamber. For us, dominance means the following:

\begin{definition}[Dominance]\label{def: dominance}
Let $\mu_0, \mu$ be as in Assumption \ref{ass: stability}. We say that $\mu_0$ dominates $\mu$ at $\alpha\in C(\CA)_\pe$ if the following condition holds for every $E$ with $[E]=\alpha$: $E$ is $\mu_0$-semistable if and only if 
\[\mu_0(F_1)=\ldots=\mu_0(F_n)\,.\]
where $F_1, \ldots, F_n$ are the $\mu$-HN factors of $E$.
\end{definition}

The more standard definition of dominance (see \cite[Definition 4.10]{JO06III}) is that $\mu, \mu_0$ should satisfy
\[\mu(\beta)\geq \mu(\gamma)\Rightarrow \mu_0(\beta)\geq \mu_0(\gamma)\]
for every $\beta, \gamma\in C(\CA)$. This implies dominance as in Definition \ref{def: dominance}, but it is sometimes too restrictive in applications, see for example \cite[Section 4.1]{FTr0}. The following notion of dominance is more flexible, and implicitly used in loc. cit.

\begin{definition}[Numerical dominance]
Let $\mu_0, \mu$ be as in Assumption \ref{ass: stability}. We say that $\mu_0$ numerically dominates $\mu$ at $\alpha\in C(\CA)_\pe$ if the following holds:
\begin{enumerate}
    \item If $E$ is $\mu_0$-semistable of class $\alpha$ and $0=E_0\subsetneq E_1\subsetneq \ldots\subsetneq E_n=E$ is the $\mu$-HN filtration of $E$, then $[E_i/E_j]\in C(\CA)_\pe$ is permissible for any $i>j$.
    \item We have
    \[\mu(\beta)\geq \mu(\gamma)\Rightarrow \mu_0(\beta)\geq \mu_0(\gamma)\]
    for every $\beta, \gamma\in C(\CA)$ such that $\beta+\gamma=[E_i/E_j]$, where $E_i$ are steps in the $\mu$-HN filtration of some $\mu_0$-semistable object in class $\alpha$ as above, $\gamma\in C(\CA)_\pe$, and the stack $\FM_\gamma^{\mu_0}$ are non-empty.
\end{enumerate}
\end{definition}

\begin{proposition}
    If $\mu_0$ numerically dominates $\mu$ then $\mu_0$ dominates $\mu$. 
\end{proposition}
\begin{proof}
Let $I\subseteq C(\CA)_\pe$ be the set of types of the form $[E_i/E_j]$ for some $E$ which is $\mu_0$-semistable of type $\alpha$. In particular, $\alpha\in I$. We start by proving that if $[E']\in I$ then $E'$ being $\mu$-semistable implies that $E'$ is $\mu_0$-semistable. Suppose that $E'$ is not $\mu_0$-semistable and let $E''\subseteq E'$ be its maximally destabilizing subobject with respect to $\mu_0$; in particular, $E''$ is $\mu_0$-semistable. Using the condition with $\gamma=[E'']$ (note that $\gamma\in C(\CA)_\pe$ by Assumption \ref{ass: stability}(4)) and $\beta=\alpha'-\gamma=[E'/E'']$ it follows that $E''$ also destabilizes $E'$ with respect to $\mu$, a contradiction.

Since extensions of $\mu_0$-semistable objects with the same $\mu_0$-slope are still $\mu_0$-semistable, if the $\mu$-HN factors $F_i=E_{i+1}/E_i$ -- which are $\mu$-semistable by definition, and hence $\mu_0$-semistable -- of $E$ have the same $\mu_0$-slope then $E$ is $\mu_0$-semistable. On the other hand, we have
\[\mu(F_i)>\mu(F_{i+1})\Rightarrow \mu_0(F_i)\geq \mu_0(F_{i+1})\]
for each $i=0, \ldots, n-1$; note that $[F_i]+[F_{i+1}]=[E_{i+2}/E_i]\in I$. But if one of the inequalities is strict then 
\[\mu_0(E_i)=\mu_0(F_1+\ldots+F_i)\geq \mu_0(F_i)>\mu_0(F_{i+1})\geq \mu_0(F_{i+1}+\ldots+F_n)=\mu_0(E/E_i)\]
which would mean that $E$ is not $\mu_0$-semistable.\qedhere
\end{proof}

We now prove the dominant wall-crossing formula:

\begin{theorem}[Dominant wall-crossing]\label{thm: dominantwc}
    Let $\mu_0, \mu$ be stability conditions as in Assumption \ref{ass: stability} such that $\mu_0$ dominates $\mu$. Then we have the equality
    \begin{equation}
        \label{eq: dominantwc}
    \delta_\alpha^{\mu_0}=\sum_{(\alpha_1, \ldots, \alpha_n)\in \HN_\alpha(\mu/\mu_0)}\delta_{\alpha_1}^{\mu}\ast\ldots \delta_{\alpha_n}^{\mu}\end{equation}
    in $\BK(\CA)$ where
    \begin{equation}\label{eq: HNmumu0}
        \HN_\alpha(\mu/\mu_0)\coloneqq \{(\alpha_1, \ldots, \alpha_n)\in \HN_\alpha(\mu)\colon \mu_0(\alpha_i)=\mu_0(\alpha),\, i=1, \ldots, n\}\,.
    \end{equation}
\end{theorem}
\begin{proof}
    By the definition of dominance, the stack of $\mu_0$-semistables is a union of strata
    \[\FM^{\mu_0}_\alpha=\bigsqcup_{c\in \HN_\alpha(\mu/\mu_0)}\FM^\mu_{=c}\,.\]
    The $\mu$ pseudo $\Theta$-stratification on $\FM_\alpha$ restricts to a pseudo $\Theta$-stratification of $\FM^{\mu_0}_\alpha$ indexed by $\HN_\alpha(\mu/\mu_0)$. By Assumption \ref{ass: stability}, this is a finite set. 

    We now apply the virtual non-abelian localization theorem, cf. Theorem \ref{thm: NAloc}. For simplicity, we argue at the level of functionals; see Remark \ref{rmk: relativeNAL} for the operational version. Let $V\in K^\ast(\FM_\alpha^\rig)$. 

    Then
    \begin{align*}
\delta_\alpha^{\mu_0}(V)&=\chi(\FM_{\alpha}^{\mu_0, \rig}, V)=\chi(\FM_{\alpha}^{\mu_0}, \pi^\ast V)\\
&=\sum_{(\alpha_1, \ldots, \alpha_n)\in \HN_\alpha(\mu/\mu_0)}\chi\big(\FM_{\alpha_1}^{\mu}\times \ldots \times \FM_{\alpha_n}^\mu,  \Sigma^\ast \pi^\ast V\otimes E_c\big)\\
&=\sum_{(\alpha_1, \ldots, \alpha_n)\in \HN_\alpha(\mu/\mu_0)}\chi\big((\FM_{\alpha_1}^{\mu}\times \ldots \times \FM_{\alpha_n}^\mu)^\rig,  (\Sigma^\rig)^\ast  V\otimes E_c\big)
    \end{align*}
The first equality is just the definition of $\delta_\alpha^{\mu_0}$; second and fourth equalities are obtained from Proposition \ref{prop: Ktheoryrig}; the third equality is the non-abelian localization theorem applied to the $\mu$ pseudo $\Theta$-stratification of $\FM_\alpha^{\mu_0}$ (see Remark \ref{rmk: pseudoNAL}). Finally, we have by Proposition \ref{prop: nproduct} and Remark \ref{rmk: kunnethdelta} the equality
    \begin{align*}
        \chi(\FM_{\alpha_1}^{\mu}\times \ldots \times \FM_{\alpha_n}^\mu,  \Sigma^\ast \pi^\ast V\otimes E_c)=(\delta_{\alpha_1}^\mu\ast\ldots \delta_{\alpha_n}^\mu)(V)
    \end{align*}
    which finishes the proof of \eqref{eq: dominantwc}.
\end{proof}

\subsection{General wall-crossing formula}\label{subsec: generalWC}

It is explained in \cite[Section 11]{joyce} how a general wall-crossing formula for stability conditions connected by a continuous path can be deduced from the dominant wall-crossing formula. The statement of the general wall-crossing formula involves the combinatorial coefficients
\begin{align*}
    S(\alpha_1, \ldots, \alpha_n; \mu, \mu')&\in \{0, -1, 1\}\,,\, U(\alpha_1, \ldots, \alpha_n; \mu, \mu')\in \BQ\,,\,\\
    \tilde U(\alpha_1, \ldots, \alpha_n; \mu, \mu')&\in \BQ
    \end{align*}
defined in \cite[Section 4]{JO06IV} for $\alpha_1, \ldots, \alpha_n\in C(\CA)$ and $\mu,\mu'$ two stability conditions. 

\begin{theorem}[Continuous path wall-crossing]\label{thm: generalwc}
    Let $\mu, \mu'$ be two stability conditions which can be connected by a continuous path crossing finitely many walls, in the precise sense that it satisfies \cite[Assumption 5.3]{joyce}\footnote{Since we are stating a formula for $\delta$ invariants as well, we should replace the non-vanishing condition on $U$ coefficients in \cite[Assumption 5.3]{joyce} by non-vanishing of $S$ coefficients.}. Then we have the following wall-crossing formulas in $\BK(\CA)$:
    \begin{align}\label{eq: deltawc}\delta_{\alpha}^{\mu'}&=\sum_{\alpha_1+\ldots+\alpha_l=\alpha}S(\alpha_1, \ldots, \alpha_n; \mu, \mu')\cdot \delta_{\alpha_1}^{\mu}\ast \ldots\ast \delta_{\alpha_n}^{\mu}\\
    \varepsilon_{\alpha}^{\mu'}&=\sum_{\alpha_1+\ldots+\alpha_l=\alpha}U(\alpha_1, \ldots, \alpha_n; \mu, \mu')\cdot \varepsilon_{\alpha_1}^{\mu}\ast \ldots\ast \varepsilon_{\alpha_n}^{\mu}\label{eq: epsilonwc}\\
    &\nonumber=\sum_{\alpha_1+\ldots+\alpha_l=\alpha}\tilde U(\alpha_1, \ldots, \alpha_n; \mu, \mu')\cdot [[\ldots [\varepsilon_{\alpha_1}^{\mu}, \varepsilon_{\alpha_2}^\mu], \ldots, ],\varepsilon_{\alpha_n}^{\mu}]
    \end{align}

    In every case, the sum runs over $\alpha_i\in C(\CA)_{\pe}$ with $\FM^{\mu'}_{\alpha_i}\neq 0$ and non-zero $S, U$ coefficients, and with those restrictions they are finite sums by assumption. In the last line, the bracket $[u,v]$ is the commutator $u\ast v-v\ast u$ in the associative algebra $\BK(\CA)$.
\end{theorem}
\begin{proof}
The argument of \cite[Section 11]{joyce} reduces the statement to that of the dominant wall-crossing formula. For the reader's convenience, we quickly summarize how it goes. Let $\mu_t$, $t\in [0,1]$ be a continuous path of stability conditions with $\mu_0=\mu$, $\mu_1=\mu'$. Then \cite[Assumption 5.3]{joyce} implies that we can find
\[0=t_0<t_1<\ldots<t_{N-1}<t_N=1\]
such that:
\begin{enumerate}
    \item The coefficients $S(\alpha_1, \ldots, \alpha_n; \mu, \mu_t)$ are constant for $t\in (t_i, t_{i+1})$;
    \item The moduli stacks $\FM_\alpha^{\mu_t}$, and hence the classes $\delta_{\alpha}^{\mu_t}$, are constant for $t\in (t_i, t_{i+1})$;
    \item If $s\in (t_{i-1}, t_{i+1})$ then $\mu_{t_i}$ numerically dominates $\mu_{s}$. Moreover, the coefficient $S(\alpha_1, \ldots, \alpha_n; \mu_s, \mu_{t_i})$ is equal to $1$ if $(\alpha_1, \ldots, \alpha_n)\in \HN_\alpha(\mu_s/\mu_{t_i})$ and 0 otherwise.
\end{enumerate}
Then the wall-crossing \eqref{eq: deltawc} for the pairs $(\mu, \mu')=(\mu_s, \mu_{t_i})$ is the dominant wall-crossing formula in Theorem \ref{thm: dominantwc}. By \cite[Lemma 11.5]{joyce}, using the formal properties of the coefficients in \cite[Theorem 4.5]{JO06IV}, the dominant wall-crossing formulas can be combined to prove
\eqref{eq: deltawc} for the pairs $(\mu, \mu')=(\mu_0, \mu_{t})$ for every $t\in [0,1]$, by induction on the $i$ such that $t\in [t_i, t_{i+1})$. 

The wall-crossing formula for the $\varepsilon$ invariants \eqref{eq: epsilonwc} formally follows from \eqref{eq: deltawc}, see the last step in the proof of \cite[Theorem 5.2]{JO06IV}. The fact that \eqref{eq: epsilonwc}  can be written entirely in terms of commutators is shown in \cite[Theorem 5.4]{JO06IV}.
\end{proof}

\subsection{Changing the heart}\label{subsec: changinghearts}
We shall now briefly discuss the alterations necessary for when we consider stability conditions when the underlying abelian categories are different hearts of the same triangulated category $\textup{Ho}(\bfT)$, such as the case of tilt stability discussed in Section \ref{subsec: tild}. First of all, the natural place to write wall-crossing formulas is in the Hall algebra
\[\BK(\bfT)\coloneqq K_\ast(\FM_{\bfT})_\BQ\,,\]
see Remark \ref{rmk: variationsHallproduct}. The inclusion of $\FM_\CA$ into $\FM_\bfT$ induces algebra homomorphisms $\BK(\CA)\to \BK(\bfT)$, and hence the classes $\varepsilon_\alpha^{ \mu}$ may be regarded as being in $\BK(\bfT)$. 

\begin{definition}\label{def: dominationheart}
    We say that $\sigma_0=(\CA_0, \mu_0)$ dominates $\sigma=(\CA, \mu)$ if every $\sigma$-semistable object is in $\CA_0$ and an object in $\CA_0$ is $\sigma_0$-semistable if and only if it lies in $\CA\cap \CA_0$ and its $\sigma$-HN factors have the same $\mu_0$ slope. 
\end{definition}

With this notion of dominance, the exact same proof of the dominant wall-crossing formula goes through, showing that \eqref{eq: dominantwc} holds in $\BK(\bfT)$. 

Consider now the case of tilt stability $\sigma_{\omega, B}=(\CA_{\omega, B}, \nu_{\omega, B})$ discussed in Section~\ref{subsec: tild}. Given a fixed topological type $\alpha$, the moduli stack $\FM^{\sigma}_\alpha$ is empty unless $\ch_1 (\alpha)\cdot \omega^2\geq \ch_0(\alpha) \omega^2\cdot B$, by construction of $\CA_{\omega, B}$. When equality holds, $\alpha$ is not permissible, so we define the space of stability conditions\footnote{In \cite[Figure 1]{FTr0} this is the region to the left (assuming $\ch_0(E)>0$) of the vertical line through $\Pi(E)=\Pi(\alpha)$.}
\[S_\alpha=\{\sigma_{\omega, B}\colon\, \omega \textup{ ample and }\ch_1 (\alpha)\cdot \omega^2> \ch_0(\alpha) \omega^2\cdot B\}\,.\]

\begin{corollary}
    Suppose we are in the setting of Proposition \ref{prop: tilt} and either $\ch_0(\alpha)>0$ or $\ch_0(\alpha)=0$ and $\ch_1(\alpha)$ is effective. If $\sigma, \sigma'\in S_\alpha$ then we have wall-crossing formulas as in Theorem \ref{thm: generalwc} comparing $\varepsilon_\alpha^\sigma$ and $\varepsilon_\alpha^{\sigma'}$ in $\BK(D^b(X))$.
\end{corollary}
\begin{proof}
    Given $\sigma=\sigma_{\omega, B}$, the abelian category $\CA_{t\omega, B}$ does not depend on $t>0$ by construction. When $t\gg 0$, all $\sigma_{t\omega, B}$ semistable objects are in $\Coh(X)$, and indeed stability is equivalent to $(\omega, B-K/2)$-twisted Gieseker stability. Then, we can compare $\sigma_{\omega, B}$ and $\sigma_{\omega', B'}$ by comparing 
    \[\sigma_{\omega, B}\leftrightsquigarrow \sigma_{t\omega, B} \leftrightsquigarrow \sigma_{t\omega', B'}\leftrightsquigarrow \sigma_{\omega', B'}\,.\]
    Each of the 3 wall-crossing formulas is an application of the general wall-crossing formula \eqref{eq: epsilonwc} for the abelian categories $\CA_{\omega, B}, \Coh(X), \CA_{\omega', B'}$, respectively.\qedhere
\end{proof}

\section{Framing functor and pair invariants}
\label{sec: framing}

The original approach of \cite{Liu} -- following \cite{joyce} -- to generalized $K$-theoretic invariants does not use the stack of semistable objects, but instead a framing functor that ``stabilizes'' the problem. In this section, we recover their definition of invariants, but now as a theorem (cf. Theorem \ref{thm: comparisonJL}). Once we compare our algebraic setup to theirs (cf. Theorems \ref{thm: liealgebrahom} and \ref{thm: homliftsLiealgebras}), the main result of this section establishes that our $\varepsilon$ classes match their classes (cf. Proposition \ref{prop: nopolesknown}(3), Theorem \ref{thm: homliftsjoyce}).

\subsection{Framing functor and stack of pairs}
\label{subsec: framingdefinitions}

Throughout this section, we fix $\alpha\in C(\CA)_\pe$ and $\mu$ a stability condition. We denote by $C_\alpha\subseteq C_\pe(\CA)$ the subset of all $\alpha'$ for which there exists $\alpha''$ so that $\alpha=\alpha'+\alpha''$ with $\mu(\alpha')=\mu(\alpha)=\mu(\alpha'')$ and $\FM_{\alpha'}^\mu, \FM_{\alpha''}^\mu\neq \emptyset$. We allow $\alpha''=0$, so that $\alpha\in C_\alpha$. By Assumption \ref{ass: stability}, $C_\alpha$ is finite. In other words, $C_\alpha$ is the set of topological types which appear as $\mu$-Jordan--Hölder factors of $\mu$-semistable objects in class $\alpha$. 

Recall from our assumptions (cf. Definition \ref{defprop: goodabelian}) that $\CA$ comes with an embedding into the homotopy category of a dg category $\bfT$ (for example $D^b(X)$). Let $\Perf$ be the dg category of perfect $\BC$-complexes.

\begin{definition}\label{def: framing}
    A framing functor for $(\CA, \alpha, \mu)$ is a functor of dg categories $\Phi\colon \bfT\to \Perf$ with the following properties:
    \begin{enumerate}
        \item If $F\in \CA$ then $H^i(\Phi(F))=0$ for $i<0$.
        \item If $[F]\in \FM_\alpha^\mu$ then $H^i(\Phi(F))=0$ for $i\neq 0$, i.e. $\Phi(F)$ is a vector space. Moreover, $\Phi(F)\neq 0$. 
    \end{enumerate}
\end{definition}

Note that a framing functor for $\alpha$ is automatically also a framing functor for any $\alpha'\in C_\alpha$. By functoriality of the construction in \cite{TV07}, a framing functor induces a map of stacks
\[\FM_\CA\to \FM_{\bfT}\to \FM_{\Perf}\,\]
which preserves the monoidal structure and the $B\BG_m$ action on the stacks. The stack $\FM_{\Perf}$ carries a universal perfect complex, and we denote by $\CV$ the pullback of this perfect complex to $\FM_\CA$. Note that $\CV$ is connective (i.e. $h^{<0}(\CV)=0$) by (1) and its restriction to $\FM_{\alpha'}^\mu$ is a vector bundle by (2).

Let 
\[\lambda\colon C(\CA)\xrightarrow{\,}\pi_0(\FM_\bfT)\xrightarrow{\Phi} \pi_0(\FM_{\Perf})\simeq \BZ\]
be the homomorphism of monoids sending $\alpha$ to $\rk(\CV_{|\FM_{\alpha}})$. Note that $\lambda(\alpha')>0$ for $\alpha'\in C_\alpha$. We will denote by $\Phi^0$ the left exact functor sending an object $F$ of $\CA$ to the vector space $H^0(\Phi(F))$. 

We let $\CB$ be the exact subcategory of $\CA$ of objects $F$ such that $H^i(F)=0$ for $i\neq 0$, which in particular contains all the $\mu$-semistable objects in class $\alpha'\in C_\alpha$. Let $\FM_\CB\subseteq \FM_\CA$ be the stack which parametrizes objects in $\CB$; in other words, $\FM_\CB$ is the locus where the perfect complex $\CV$ is concentrated in degree 0, which is open in $\FM_\CA$. Note that $\CB$ is closed under direct sums and quotients, but not necessarily kernels, so it is not necessarily an abelian category.

The following examples of framing functors have been considered in \cite{joyce, GJT, bu, parabolic}.

\begin{example}
Let $Q$ be a quiver, $\CA=\Rep_Q$ and fix some positive integers $a_v\in \BZ_{>0}$ for each vertex $v\in Q_0$. Then the functor that maps a representation $\{V_v\}_{v\in Q_0}$ of $Q$ to $\bigoplus_{v\in Q_0} V_v^{\oplus a_v}$ is a framing functor for every $\alpha, \mu$. In this case, $H^i(\Phi(F))=0$ for every $i\neq 0$ and $F\in \Rep_Q$. 
\end{example}

\begin{example}
    Suppose that $X=\Coh(X)$ and $\alpha\in C(\Coh(X))$ and $\mu$ satisfies Assumption \ref{ass: stability} (for example $\mu$ is slope or Gieseker stability). Let $\CO_X(1)$ be an ample line bundle. By \cite[Lemma 1.7.2]{HL}, every object in $\CB$ is $N$-regular for sufficiently large $N$, and hence the functor
    \[R\Gamma(-\otimes \CO_X(N))\colon D^b(X)\to \Perf\]
    is a framing functor.
\end{example}

\begin{example}
    Suppose that $\CA=\Coh(C)$ where $C$ is a curve and fix a point $p\in C$; denote by $\iota\colon \{p\}\to C$ the inclusion. Then
    \[L\iota^\ast(-)^\vee\colon D^b(C)\to D^b(\{p\})\simeq \Perf \]
    is a framing functor for slope stability.
\end{example}

We let $\CA_\Phi$ be the mapping cylinder of $\Phi^0$ (cf. \cite{mozgovoyquiverrepab}), i.e. the abelian category which parametrizes triples $(F, U, f)$ where $U\in \Vect$, $F\in \CA$ and $f\colon U\to \Phi^0(F)$ is a morphism in $\Vect$; we call an element of $\CA_\Phi$ a $\Phi$-pair, and we sometimes denote it by $[U\to \Phi^0(F)]$. Similarly, define $\CB_\Phi$ to be the mapping cylinder of $\Phi\colon \CB\to \Vect$, which is the exact subcategory of $\CA_\Phi$ where $F\in \CB$. The topological type of a $\Phi$-pair is $([F], \dim U)\in C(\CA)\times \BZ_{\geq 0}$.

We now define the derived stack $\FP=\FP_{\CA, \Phi}$ which parametrizes objects in $\CB_\Phi$. Recall that $\FM_{\CB}$ has the vector bundle $\CV$ whose fiber over $F$ is $\Phi(F)$, and let $\CU$ be the universal vector bundle on
\[\FM_\Vect=\bigsqcup_{n\geq 0}B\GL_n\,.\]
Then $\FP$ is defined to be the total space of the vector bundle
\[\FP\coloneqq \textup{Tot}_{\FM_\CB\times \FM_{\Vect}}\big(\CU^\vee\otimes \CV\big)\,.\]

We still denote by $\CU, \CV$ the pullbacks to $\FP$; on $\FP$ there is a universal linear map $\CU\to \CV$. The stack $\FP$ admits a decomposition into connected components
\[\FP=\bigsqcup_{(\alpha, d)\in C(\CA)\times \BZ_{\geq 0}} \FP_{(\alpha, d)}\,,\]
some of which are possibly empty (for example if $\lambda(\alpha)<0$). Note that 
\[\FP_{(\alpha, 0)}\simeq \FM_{\CB, \alpha}\coloneqq \FM_\CB \cap \FM_\alpha\] and in general there is a projection map $\Pi\colon \FP_{(\alpha, d)}\to \FM_{\CB, \alpha}$. The fibers of this map are isomorphic to
\[\BA^{d\cdot \lambda(\alpha)}/\GL_d\,\]
and in particular this is a smooth map.

We define an Ext complex on $\FP\times \FP$ as
\begin{equation}
    \label{eq: extPhi}\Ext_{12}^\Phi=\Pi^\ast \Ext_{12}\oplus \big[\underset{0}{\CU_1^\vee\otimes \CU_2}\to \underset{1}{\CU_1^\vee\otimes \CV_2}\big]\,.
    \end{equation}
Despite $\CB_\Phi$ not being an abelian category, the stack $\FP$ and the complex $\Ext_{12}^\Phi$ have all the structure and properties discussed in Section \ref{subsec: abcat}; see \cite[Section 5.2]{joyce} for a detailed discussion. In particular, \eqref{eq: derivedtangentFM} holds and the relative derived tangent bundle of $\Pi$ is
\[\BT_{\Pi}=\big[\underset{-1}{\CU^\vee\otimes \CU}\to \underset{0}{\CU^\vee\otimes \CV}\big]\,.\]
This follows from the fact that $\BT_{\FM_\Vect}=(\CU^\vee\otimes \CU)[1]$. The restriction of $\BT_\Pi$ to the locus of $\FP$ which parametrizes $\Phi$-pairs with $f\colon U\to \Phi(V)$ injective is a vector bundle. The $K$-Hall algebras $\BK(\CB)=K_\ast(\FM_\CB)\otimes_\BZ \BQ$ and $\BK(\CB_\Phi)=K_\ast(\FP)\otimes_\BZ \BQ$ are defined in the same way as for an abelian category.

\subsection{Stability conditions on $\CA_\Phi$}
\label{subsec: stabilityframing}

From the stability condition $\mu$ on $\CA$, we define 2 stability conditions on the abelian category of pairs $\CA_\Phi$. The first is the ``naive'' stability condition
\[\mu_0\colon C(\CA_\Phi)\setminus \{(0,0)\}\to T\cup\{\infty\}\]
where $\infty>t$ for every $t\in T$, and is defined by
\[\mu_0(\beta, d)=\begin{cases}\mu(\beta)&\textup{ if }\beta\neq 0\\
\infty &\textup{ if }\beta=0\,.
\end{cases}\]
The second one is the Joyce--Song stability condition. To define it, we consider some rank function\footnote{The choice of rank function not really play any role, as we will see in Proposition \ref{prop: characterizingstabilitypairs}. In examples there is often a natural choice, which is the denominator of $\mu$.} $\rk\colon C(\CA)\setminus 0\to \BZ_{>0}$ which satisfies 
\[\mu(\beta)=\mu(\beta')\Rightarrow  \rk(\beta)+\rk(\beta')=\rk(\beta+\beta')\,.\]
Then we let 
\[\mu_0\colon C(\CA_\Phi)\setminus \{(0,0)\}\to (T\times \BQ)\cup\{\infty\}\,,\]
where $T\times \BQ$ is given the lexicographic order and $\infty$ is the maximum, given by
\[\mu_+(\beta, d)=\begin{cases}\big(\mu(\beta), d/\rk(\beta)\big)&\textup{ if }\beta\neq 0\\
\infty &\textup{ if }\beta=0\,.
\end{cases}\]

Stability can be characterized as follows:

\begin{proposition}\label{prop: characterizingstabilitypairs}
    Let $\alpha'\in C_\alpha$. 
\begin{enumerate}[label=(\alph*)]
    \item A pair $[F\to 0]$ of topological type $(\alpha', 0)$ is $\mu_0$-(semi)stable if and only if it is $\mu_+$-(semi)stable if and only if $F$ is $\mu$-(semi)stable.
    \item A pair $[U\xrightarrow{f} \Phi^0(F)]$ of topological type $(\alpha', 1)$ (in particular with $\dim U=1$) is $\mu_0$-(semi)stable if and only if $F$ is $\mu$-(semi)stable and $f\neq 0$.
    \item A pair $[U\xrightarrow{f} \Phi^0(F)]$ of topological type $(\alpha', 1)$ is $\mu_+$-semistable if and only if it is $\mu_+$ stable if and only if $F$ is semistable, $f\neq 0$, and $f$ does not factor as $U\to \Phi^0(F')\subsetneq \Phi^0(F)$ for some subobject $F'\neq 0$ with $\mu(F')=\mu(F/F')$.
    \item The stability condition $\mu_0$ dominates $\mu_+$ at $(\alpha', 1)$, in the sense of Definition \ref{def: dominance}.
\end{enumerate}
\end{proposition}
\begin{proof}
Part (a) is trivial. Part (b) is also easy, since clearly $\mu_0$-stability implies $\mu$-stability of $F$ and, if $F$ is $\mu$-semistable, the only possible destabilizing subobject of $U\to V$ is $U\to 0$, which is a subobject if and only if $f=0$. Part (c) is standard, see \cite[Example 5.6]{joyce}. 

For part (d), obviously $\mu_+$-semistablity implies $\mu_0$-semistability. Suppose that $[U\xrightarrow{f} \Phi^0(F)]$ is $\mu_0$-semistable but not $\mu_+$-semistable. Then its $\mu_+$-HN filtration is
\[0\subsetneq [U\xrightarrow{f} \Phi^0(F')]\subsetneq [U\xrightarrow{f} \Phi^0(F)]\]
where $F'$ is the maximal subobject of $F$ with $\mu(F')=\mu(F/F')$ through which $f$ factors. Then the $\mu_+$-HN factors both have
\[\mu_0([U\xrightarrow{f} \Phi^0(F')]=\mu(F')=\mu(F)=\mu(F/F')=\mu_0([F/F'\rightarrow 0])\,.\]
The other implication is also easy to establish.
\end{proof}

Recall the definition \eqref{eq: HNmumu0} of $\HN_\alpha(\mu_+/\mu_0)$. It follows from the proof of (d) above that
\[\HN_\alpha(\mu_+/\mu_0)=\{(\alpha, 1)\}\cup \{\big((\alpha', 1), (\alpha'',0)\big)\colon \alpha'+\alpha''=\alpha,\, \mu(\alpha')=\mu(\alpha'')\}\,.\]

Given $\alpha'\in C_\alpha$ we denote by
\[\FP_{(\alpha',1)}^{\mu_+}\subseteq \FP_{(\alpha',1)}^{\mu_0}\subseteq \FP_{(\alpha',1)}\]
the stacks of $\mu_+$ and $\mu_0$-semistable pairs. By Proposition \ref{prop: characterizingstabilitypairs}(b), $\FP_{(\alpha', 1)}^{\mu_0}$ is a projective bundle over $\FM_{\alpha'}^\mu$:
\[\Pi\colon \FP_{(\alpha', 1)}^{\mu_0}=\BP(\CV)\to \FM_{\alpha'}^\mu\,.\]

By Proposition \ref{prop: characterizingstabilitypairs}, there are no $\mu_+$-strictly semistable pairs, and in particular $\varepsilon_{(\alpha, 1)}^{\mu_+}=\delta_{(\alpha, 1)}^{\mu_+}\in \BK(\CB_{\Phi})$ are defined without any trouble. The next result expresses the $\varepsilon$ classes inductively in terms of the classes coming from the moduli spaces of Joyce--Song pairs. Indeed, in the theories of Joyce and Liu, this theorem is actually their definition of invariants. 

\begin{theorem}\label{thm: comparisonJL}
Let $\Phi$ be a framing functor as above and let $\mu_+$ be Joyce--Song stability on $\CA_\Phi$. Then, we have the equality
    \begin{equation}\label{eq: framingdefinitionepsilon}
        \Pi_\ast\big(\Lambda_{-1}(\BT_\Pi^\vee)\cap \varepsilon_{(\alpha, 1)}^{\mu_+}\big)=\sum_{\substack{\alpha_{1}+\ldots+\alpha_n=\alpha\\
        \mu(\alpha_i)=\mu(\alpha)}}\frac{(-1)^{n-1}}{n!}\lambda(\alpha_1)\big[\big[\ldots\big[\varepsilon_{\alpha_1}^\mu,\varepsilon_{\alpha_2}^\mu\big],\ldots, \varepsilon_{\alpha_n}^\mu\big]\,. 
    \end{equation}
    in $\BK(\CA)$.
\end{theorem}

\subsection{Morphisms between $K$-Hall algebras}
\label{subsec: morphismsKHall}

In the proof of Theorem \ref{thm: comparisonJL}, we will need as an input the property that $\Pi_\ast(\Lambda_{-1}(\BT_\Pi^\vee)\cap -)$ is (for practical purposes) a homomorphism. Indeed, this is a somewhat general phenomena that could be of interest in other scenarios, so we formulate it in general. In what follows, the main example to keep in mind is the forgetful functor $\CC=\CA_\Phi\to \CA$ and the induced projection of stacks $\Pi\colon \FP\to \FM$.

Consider a morphism of abelian categories $\CC\to \CA$ and assume that it lifts to a map of derived stacks $f\colon \FM_\CC\to \FM_\CA$ which is compatible with the direct sum maps and the $B\BG_m$ actions. We also write $f$ for the induced map $\FM_\CC^\rig\to \FM^\rig_\CA$. Define the $K$-theory class
\begin{equation}\label{eq: differenceExt}
F_{12}=(f\times f)^\ast \Ext^\CA_{12}-\Ext^\CC_{12}\in K^\ast(\FM_\CC\times \FM_\CC)\,.
\end{equation}
Similarly, define $F_{21}$. Note that $F_{12}$ descends to $(\FM_\CC\times \FM_\CC)^\rig$. The relative derived tangent bundle of the morphism $f$ is given by the restriction
\[\BT_f=\BT_{\CC}-f^\ast \BT_\CA=\Delta^\ast F_{12}\in K^\ast(\FM_{\CC})\,\]
to the diagonal $\Delta\colon \FM_\CC\to \FM_\CC\times\FM_\CC$, and it descends to $\FM_\CC^\rig$. 

The following proposition, constructing a homomorphism of associative algebras, is analogous to the homomorphism of vertex algebras and Lie algebras in \cite[Theorem 2.12]{GJT}.

\begin{proposition}\label{prop: homalgebras}
    Suppose that the $K$-theory class $F$ is represented by a vector bundle on $\FM_\CC\times \FM_\CC$. Then the map
    \begin{align*}
\Upsilon\colon \BK(\CC)&\rightarrow \BK(\CA)\\
\phi&\mapsto f_\ast(\Lambda_{-1}(\BT_f^\vee)\cap \phi)
\end{align*}
is a homomorphism of associative algebras. 

More generally, if $Z_1, Z_2\subseteq \FM_\CC$ are such that the restriction of $F_{12}$ to $Z_i\times Z_j$ is represented by a vector bundle and $\phi, \psi$ are pushed forward from $Z_1, Z_2$, respectively, then 
\[\Upsilon(\phi\ast \psi)=\Upsilon(\phi)\ast\Upsilon(\psi)\,.\]
\end{proposition}
\begin{proof}
Assume first that $F_{12}$ is globally a vector bundle. Let us denote by $F_{ij}$ the pullback of $F_{12}$ via the map $\FM_\CC\times \FM_\CC\xrightarrow{p_i\times  p_j}\FM_\CC\times \FM_\CC$ where $p_i$ is the projection in the $i$-th component for $i=1,2$. In particular, $F_{ii}=p_i^\ast \BT_f$. Then, we have
\begin{equation}
    \label{eq: sigmarelativetangent}
    \Sigma^\ast \BT_{f}=F_{11}+F_{22}+F_{12}+F_{21}\,.\end{equation}
Since by assumption $F_{12}$ is a vector bundle, Proposition \ref{lem: symmetryVB} gives an equality
\begin{equation}
    \label{eq: symmetryEulerclass}
\Lambda_{-1}(F_{12}^\vee)=\Lambda_{-1}(F_{12})\otimes \det(F_{12})^\vee[\rk F_{12}]\,.
\end{equation}
Combining \eqref{eq: differenceExt}, \eqref{eq: sigmarelativetangent} and \eqref{eq: symmetryEulerclass} we get the equality
\begin{equation}
    \label{eq: keyhomomorphism}
\Lambda_{-1}(\Sigma^\ast \BT_{f}^\vee)\otimes (f\times f)^\ast\Gamma_-^{\CA}=p_1^\ast\Lambda_{-1}(\BT_f)\otimes p_2^\ast\Lambda_{-1}(\BT_f)\otimes \Gamma_-^{\CC}\end{equation}
in $K^\ast(\FM_{\CC}\times \FM_{\CC})$, where we denote by $\Gamma_-^\CA, \Gamma_-^\CC$ the complexes \eqref{eq: gammacomplex} for the abelian categories $\CA$ and $\CC$. Together with the compatibilities
\[f\circ\Sigma_\CC=\Sigma_\CA\circ (f\times f)\textup{ and }(f\times f)\circ \pi_\CC=\pi_\CA\circ (f\times f)\,,\]
where $\Sigma_\CA, \Sigma_\CC, \pi_\CA, \pi_\CC$ are the maps in Definition \ref{def: KHall} for $\CC, \CA$, we obtain $\Upsilon(\phi\ast \psi)=\Upsilon(\phi)\ast\Upsilon(\psi)$ by unraveling the definitions.

In the more general case, if $j_1\colon Z_1\hookrightarrow \FM_\CC$ then $\Upsilon(\phi)$ is, more precisely, defined by
\[\Upsilon(\phi)\coloneqq (f\circ j_1)_\ast(\Lambda_{-1}(j_1^\ast \BT_f^\vee)\cap \phi)\,.\]
Equality \eqref{eq: keyhomomorphism} still holds after restriction to $Z_1\times Z_2$ and the proof goes through by push-pull to $Z_1\times Z_2$. \qedhere
\end{proof}

The following is a straightforward corollary:

\begin{corollary}\label{cor: inclusionhomomorphism}
    Let $\CC, \CA$ be good abelian categories such that $\CC\subseteq \CA$ is a full and faithful subcategory. Then pushforward along the inclusion $\FM_{\CC}^\rig\hookrightarrow \FM_{\CA}^\rig$ induces a homomorphism
    \[\BK(\CC)\to \BK(\CA)\,.\]
\end{corollary}
\begin{proof}
    Proposition \ref{prop: homalgebras} applies with $F=0$. 
\end{proof}

\subsection{Proof of Theorem \ref{thm: comparisonJL}}
\label{subsec: proofcomparisonJL}

The basic idea of the proof is to combine the projective bundle $\FP_{(\alpha, 1)}^{\mu_0, \rig}\to \FM_\alpha^{\mu, \rig}$ with the wall-crossing formula relating $\mu_0$ and $\mu_+$ stability. We will use the morphism $\Upsilon$ from the previous section given by 
\[\Upsilon(\phi)= \Pi_\ast(\Lambda_{-1}(\BT_\Pi^\vee)\cap \phi)\in \BK(\CA)\]
for $\phi\in \BK(\CB_\Phi)$ for which it is defined. We also have homomorphisms $\BK(\CB)\to \BK(\CA)$ and $\BK(\CB)\to \BK(\CB_\Phi)$ induced by the respective inclusions of categories, cf. Corollary \ref{cor: inclusionhomomorphism}.

We will abbreviate $\varepsilon_\alpha=\varepsilon_\alpha^\mu$, $\varepsilon_{(\alpha, 1)}^+=\varepsilon_{(\alpha, 1)}^{\mu_+}$ and $\varepsilon_{(\alpha, 1)}^0=\varepsilon_{(\alpha, 1)}^{\mu_0}$, and use similar notation for $\delta$ classes. By abuse of notation, the classes $\varepsilon_\alpha$ can be regarded either in $\BK(\CB)$, $\BK(\CA)$ or $\BK(\CB_\Phi)$. Following this abuse, we have by Proposition \ref{prop: characterizingstabilitypairs}(a)
\[\varepsilon_\alpha=\varepsilon_{(\alpha,0)}^0=\varepsilon_{(\alpha, 0)}^+\in \BK(\CB_\Phi)\,.\]

\textit{Step 1: Projective bundle formula.} By Proposition \ref{prop: characterizingstabilitypairs}(b) and the discussion that follows it, $\FP_{(\alpha, 1)}^{\mu_0, \rig}$ is the projectivization of the restriction of $\CV$ to $\FM_\alpha^{\mu, \rig}$, which is a vector bundle of rank $\lambda(\alpha)$. By \cite[Lemma 2.1.16]{Liu} it follows that

\begin{equation}\label{eq: projectivebundleformula}
    \Upsilon(\delta_{(\alpha, 1)}^{\mu_0})=\lambda(\alpha)\delta_\alpha^{\mu}\,.
\end{equation}
Note that \cite[Lemma 2.1.16]{Liu} is stated for schemes with an obstruction theory, but the projective bundle formula in \cite[Propositions 3.1(iii), 3.2]{khan} can be used to extend it to derived stacks. Equation \eqref{eq: projectivebundleformula} also holds if we replace $\alpha$ by any $\alpha'\in C_\alpha$. 

\textit{Step 2: Wall-crossing.} By the dominant wall-crossing formula (cf. Theorem \ref{thm: dominantwc}) and Proposition \ref{prop: characterizingstabilitypairs}(d), we have the wall-crossing formula between naive and Joyce--Song stabilities:
\begin{equation}\label{eq: wcframing}
\delta_{(\alpha,1)}^0 = \delta_{(\alpha, 1)}^+ +\sum_{\substack{\alpha'+\alpha''=\alpha\\
\mu(\alpha')=\mu(\alpha'')}} \delta_{(\alpha',1)}^+ * \delta_{\alpha''}\end{equation}
in $\BK(\CB_\Phi)$. 

\textit{Step 3: Applying $\Upsilon$.} We now apply $\Upsilon$ to \eqref{eq: wcframing} to obtain
\begin{equation}\label{eq: wcframingupsilon}\lambda(\alpha)\delta_{\alpha} = \Upsilon(\varepsilon_{(\alpha,1)}^+) +\sum_{\substack{\alpha_1+\alpha_2=\alpha\\
\mu(\alpha_1)=\mu(\alpha_2)}} \Upsilon(\varepsilon_{(\alpha_1,1)}^+) * \delta_{\alpha_2}\,.\end{equation}

Here, we are using \eqref{eq: projectivebundleformula} on the left hand side and the general version of Proposition \ref{prop: homalgebras}. To see that we can apply it, note that in this case the $K$-theory class $F_{12}$ is given by $\CU_1^\vee\otimes \CV_2-\CU_1^\vee\otimes \CU_2$ 
(cf. \eqref{eq: extPhi}) and we may take $Z_1=\FM_{(\alpha_1, 1)}^{\mu_0}$ and $Z_2=\FM_{(\alpha_2, 0)}^{\mu}$. Indeed, the restriction of $\CU_1^\vee\otimes \CU_2$ is trivial when we restrict to $Z_i\times Z_j$ except in the case $(i, j)=(1,1)$. When $(i, j)=(1,1)$, $F_{12}$ can be represented by the cokernel of $\CU_1^\vee\otimes \CU_2\to \CU_1^\vee\otimes \CV_2$, which is a vector bundle since $\CU\to \CV$ is injective over $\FM^{\mu_0}_{(\alpha_1, 1)}$.

\textit{Step 4: Combinatorics.}
Observe that \eqref{eq: wcframingupsilon} determines recursively the invariants $\delta_\alpha$, and hence $\varepsilon_\alpha$, from $\Upsilon(\varepsilon_{(\alpha',1)}^+)$. The same is true for the formula \eqref{eq: framingdefinitionepsilon}, so it remains to show that the two are equivalent, which is a combinatorial statement. We define classes $\tilde \varepsilon_{\alpha'}$, for $\alpha'\in C_\alpha$, to be the classes that satisfy \eqref{eq: framingdefinitionepsilon} (for any $\alpha'\in C_\alpha$ in place of $\alpha$), i.e.
\begin{align*}
       \Upsilon\big(\varepsilon_{(\alpha, 1)}^{+}\big)&=\sum_{\substack{\alpha_{1}+\ldots+\alpha_n=\alpha\\
        \mu(\alpha_i)=\mu(\alpha)}}\frac{(-1)^{n-1}}{n!}\lambda(\alpha_1)\big[\big[\ldots\big[\tilde \varepsilon_{\alpha_1},\tilde \varepsilon_{\alpha_2}\big],\ldots\big], \tilde \varepsilon_{\alpha_n}\big]\\
        &=\sum_{\substack{\alpha_{1}+\ldots+\alpha_n=\alpha\\
        \mu(\alpha_i)=\mu(\alpha)}}\sum_{p=1}^n\frac{(-1)^{n-p}}{n!}\lambda(\alpha_p)\tilde \varepsilon_{\alpha_1}\ast  \tilde\varepsilon_{\alpha_2}\ast\ldots\ast \tilde \varepsilon_{\alpha_n}\,.
\end{align*}
The second equality is \cite[(9.54)]{joyce}. Define invariants $\tilde \delta_\alpha$ in terms of $\tilde \varepsilon_\alpha$. We ought to show that $\tilde \delta_{\alpha}=\delta_\alpha$, which amounts to show that $\tilde \delta_\alpha$ satisfies \eqref{eq: wcframingupsilon}. For $\alpha'+\alpha''=\alpha$ and $\mu(\alpha')=\mu(\alpha'')$ we have
\[\Upsilon(\varepsilon_{(\alpha',1)}^+) * \tilde \delta_{\alpha''}=\sum_{\substack{\alpha_{1}+\ldots+\alpha_k=\alpha'\\
        \mu(\alpha_i)=\mu(\alpha)}} \sum_{\substack{\alpha_{k+1}+\ldots+\alpha_{k+l}=\alpha''\\
        \mu(\alpha_i)=\mu(\alpha)}}\sum_{p=0}^k\frac{(-1)^{k-p}}{k!l!}\binom{k-1}{p-1}\lambda(\alpha_p)\tilde \varepsilon_{\alpha_1}\ast \ldots\ast\tilde \varepsilon_{\alpha_{k+l}}\,.\]
Hence, the term $\tilde \varepsilon_{\alpha_1}\ast \ldots\ast\tilde \varepsilon_{\alpha_{n}}$ appears in 
\begin{align}\label{eq: wcupsilontilde}\Upsilon(\varepsilon_{(\alpha,1)}^+) +\sum_{\substack{\alpha'+\alpha''=\alpha\\
\mu(\alpha')=\mu(\alpha'')}} \Upsilon(\varepsilon_{(\alpha',1)}^+) * \tilde \delta_{\alpha''}
\end{align}
with coefficient
\[\sum_{p=1}^n \lambda(\alpha_p)\sum_{\substack{k\geq p, \,l\geq 0\\k+l=n}}\frac{(-1)^{k-p}}{k!l!}\binom{k-1}{p-1}=\frac{1}{n!}\sum_{p=1}^n\lambda(\alpha_p)=\frac{1}{n!}\lambda(\alpha)\,.\]
Note that the isolated term $\Upsilon(\varepsilon_{(\alpha,1)}^+)$ corresponds to allowing $l=0$. The identity used is deduced from the observation that the coefficient of $x^n$ in
\[(1+x)^n\frac{x^p}{(1+x)^p}=x^p(1+x)^{n-p}\]
is 1. Then \eqref{eq: wcupsilontilde} is equal to $\lambda(\alpha)\tilde \delta_\alpha$, which concludes the proof.

\section{Comparison with Liu's vertex algebra}\label{sec:Liuvert}

In previous work by Joyce \cite{joyce} (in cohomology) and Liu \cite{Liu} (in $K$-theory), the Lie algebra where wall-crossing formulas are written is obtained from a (multiplicative) vertex algebra. In this section we compare our construction to theirs, matching the wall-crossing formulas obtained from either method.

\subsection{Vertex algebras and multiplicative vertex algebras}\label{subsec: multiplivativeva}

We start by recalling the notion of (graded) vertex algebras (cf. \cite{FB-Z, Ka98, liFVA, LLVA, joyce}). To distinguish from multiplicative vertex algebras, we will sometimes call regular vertex algebras ``additive''.  Given a vector space $V$ we write $V\lpp z\rpp=V\llbracket z \rrbracket[z^{-1}]$ for the ring of Laurent series with coefficients in $V$.

\begin{definition}\label{def: grva} An (additive) super\footnote{We consider the ``super'' case for additive vertex algebras since the main example we have in mind comes from the homology of a stack, that admits a natural $\BZ/2$-grading.} vertex algebra is the data of $(V, \mathbf 1, D, Y)$ where:
\begin{enumerate}[label=(\roman*)]
        \item $V$ is a vector space over $\BQ$ with a $\BZ/2$-grading;
        \item $\mathbf 1 \in V$ is the \textit{identity element}, with parity 0;
        \item $D: V\to V$ is the \textit{translation operator}, which preserves the parity. 
        \item $Y$ is the \textit{state-field correspondence}: $Y: V\otimes V\to V\lpp z\rpp$. We write
        \[Y(a, z)b=\sum_{n\in \BZ}(a_n b)z^{-1-n}\]
        where $a_nb\in V$ and it is 0 for $n\gg 0$.
    \end{enumerate}
    These objects are required to satisfy the following properties for all $a,b, c\in V$: 
    \begin{enumerate}[label=(\roman*)]\label{vertalg}
        \item Vacuum: $Y(\mathbf 1, z) a = a$;
        \item Associativity: $Y(a, z - w)Y(b, z)c \equiv Y(Y(a, z)b, w)c$ where $\equiv$ means that both sides are obtained from expanding the same element in
        \[V\llbracket z, w\rrbracket[z^{-1},w^{-1},(z-w)^{-1}]\,.\]
        \item Skew-symmetry: $Y(a, z)b=(-1)^{|a||b|}e^{zD}Y(b, -z) a$ where $|a|\in \{0,1\}$ denotes the $\BZ/2$-grading of $a$. 
       % \item $Y(a, z) \mathbf 1 = e^{zD}a;$
    \end{enumerate}
    The vertex algebra is said to be $\BZ$-graded if there is a $\BZ$-grading on $V$ upgrading the $\BZ/2$-grading such that $\1$ is degree 0, $D$ is of degree 2, and $Y$ is of degree 0 where the variable $z$ is regarded as having degree $-2$.
\end{definition}

\begin{remark}
The equivalence $\equiv$ in the associativity axiom is equivalent to the existence of some integer $N>0$ such that 
\[(z-w)^N Y(a, z - w)Y(b, z)c =(z-w)^N Y(Y(a, z)b, w)c\,.\]
\end{remark}
%\begin{remark} This, perhaps, should be understood most properly in the language of the expansions from the~\cite[Subsection 3.1.3]{Liu}. Namely, there is a diagram of the ``expansion maps''
% https://q.uiver.app/#q=WzAsNCxbMCwwLCJcXG1hdGhiYiBaKCh1KSkoKHYpKSJdLFsyLDAsIlxcbWF0aGJiIFpbdSwgdiwgKHUtdileey0xfV0iXSxbMywwXSxbNCwwLCJcXG1hdGhiYiBaKCh2KSkoKHUpKSJdLFsxLDAsIlxcaW90YV92IiwwLHsic3R5bGUiOnsidGFpbCI6eyJuYW1lIjoiaG9vayIsInNpZGUiOiJib3R0b20ifX19XSxbMSwzLCJcXGlvdGFfdSIsMix7InN0eWxlIjp7InRhaWwiOnsibmFtZSI6Imhvb2siLCJzaWRlIjoidG9wIn19fV1d
%\[\begin{tikzcd}
	%{\mathbb Z\lppz_1))\lppz_2))} && {\mathbb Z[z_1, z_2][\textcolor{red}{z_1^{-1}, z_2^{-1}}, (z_1-z_2)^{-1}]} & {} & {\mathbb Z\lppz_2))\lppz_1))}
	%\arrow["{\iota_1}", hook', from=1-3, to=1-1]
	%\arrow["{\iota_2}"', hook, from=1-3, to=1-5],
%\end{tikzcd}\]
%(where $\iota_u$ and $\iota_v$ are, of course, different, and, in particular, differ on $(z_1 -z_2)^{-1}$). What the weak commutativity de-facto claims is that $Y(u, z_1)Y(v, z_2)w$ and $\pm Y(v, z_2)Y(u, z_1)$ are the images, under the maps $\iota_1$ and $\iota_2$, of the same element.
%\end{remark}

There are several different ways to formulate the vertex algebra axioms. For example, the ones above imply the locality axiom
\[Y(a, z)Y(b, w)c\equiv (-1)^{|a||b|}Y(b, w)Y(a, z)c\]
which is sometimes included in the definition.

Note that the translation operator always appears through
\[D(z)\coloneqq e^{zD}\colon V\to V\llbracket z \rrbracket\]
which satisfies
\[D(z)D(w)=D(z+w)\textup{ and }D(0)=\id\,.\]
The formal variable $z$ in the definition of vertex algebra should be thought as being ``additive''. Indeed, there is a generalization of vertex algebras where this additivity is replaced by any group law \cite{liFVA}. The multiplicative case plays a central role in the formulation of $K$-theoretic wall-crossing in \cite{Liu}.

\begin{definition}\label{def: multva} A multiplicative vertex algebra is the data of $(V, \mathbf 1, D(u), Y)$ where:
\begin{enumerate}[label=(\roman*)]
        \item $V$ is a vector space over $\BQ$;
        \item $\mathbf 1 \in V$ is the \textit{identity element}, with parity 0;
        \item $D(u): V\to V\llbracket 1-u \rrbracket$ is the \textit{translation operator}, satisfying $D(u)D(v)=D(uv)$ and $D(1)=\id$.
        \item $Y$ is the \textit{state-field correspondence} $Y: V\otimes V\to V\lpp1-u\rpp$. 
    \end{enumerate}
    These objects are required to satisfy the following properties for all $a,b, c\in V$: 
    \begin{enumerate}[label=(\roman*)]\label{vertalg}
        \item Vacuum: $Y(\mathbf 1, u) v = v;$
        \item Associativity: $Y(Y(a, u)b, v)c\equiv Y(a, uv)Y(b, v)c$, where $\equiv$ means that both sides are the expansions of the same element in
$$V\llbracket 1-u, 1-v\rrbracket [(1-u)^{-1}, (1-v)^{-1},(1-uv)^{-1}].$$
        \item Skew-symmetry: $Y(a, u)b = D(u)Y(b, u^{-1})a$;
    \end{enumerate}
\end{definition}

It turns out that vertex algebras associated to a group law can be effectively reduced to the usual ones (over a field of characteristic 0) by a simple change of variable. We state this only for the additive/multiplicative cases.
\begin{proposition}[{\cite[Proposition 3.6]{liFVA}}]\label{prop: equiv}The change of variables $u=e^z$ determines an equivalence of categories between the categories of multiplicative and additive vertex algebras. More precisely, if $(V, \1, D(u), Y)$ is a multiplicative vertex algebra then $(V, \1, D', Y')$ is an additive vertex algebra, where $D'(u)=D(e^z)$ and $Y'(a, z)=Y(a, e^z)$. Here, we are implicitly using the isomorphism
\[\BQ\llbracket 1-u\rrbracket \simeq \BQ\llbracket z\rrbracket\]
which sends $1-u\mapsto 1-e^z\in z\BQ\llbracket z\rrbracket$. 
\end{proposition}

Given a (super) additive vertex algebra $V$, there is an associated (super) Lie algebra, constructed by Borcherds \cite{Borcherds}. Its underlying vector space is the quotient
\[\widecheck V\coloneqq V/\im(D)\]
and the Lie bracket is defined by
\[[\bar a, \bar b]\coloneqq \overline{\Res_{z=0}Y(a,z)b}\,.\]
When $V$ is a multiplicative vertex algebra, an analogous construction appears in \cite[3.2.13]{Liu}. In the multiplicative case, one defines
\[\widecheck{V}\coloneqq V/\im(D(u)-\id)\]
where $\im(\id-D(u))$ is the subspace of $V$ generated by the coefficients of $(D(u)-\id)a$ for all $a \in V$, and
\[[\overline a, \overline b]\coloneqq \overline{\Res_{u=1}u^{-1}f(u)}\,.\]

\begin{remark}\label{rmk: residuechangevariables}
Note that if $u=e^z$ then
    \[\Res_{u=1}u^{-1}f(u)=\Res_{z=0}f(z)\,,\]
    so the Lie algebras obtained by regarding $V$ as either a multiplicative or additive vertex algebra, through Proposition \ref{prop: equiv}, are the same.
\end{remark}

\begin{remark}\label{rmk: polesresidue}
If $f(u)\in \BQ[(1-u)^{\pm 1}]$ is a rational function with poles only at 1, then the residue theorem implies that 
\begin{align*}\Res_{u=1}u^{-1}f(u)&=-\Res_{u=\infty}u^{-1}f(u)-\Res_{u=0}u^{-1}f(u)\\
&=\Res_{u=0}uf(1/u)-\Res_{u=0}u^{-1}f(u)=[u^0](f_--f_+)\end{align*}
where $f_-, f_+$ are the Taylor expansions of $f$ in $u^{-1}$ and $u$, respectively, and $[u^0](\ldots)$ denotes the $u^0$ coefficient. The latter expression is the definition that Liu uses in the equivariant setting, where the equality above does not hold due to the presence of non-trivial roots of unity as poles. See \cite[Appendix A]{Liu} for a more detailed discussion of these residue maps.
\end{remark}

\subsection{Joyce's vertex algebra}\label{subsec: JoyceVA}

In what follows, $\CA$ is a good abelian category as in Definition \ref{defprop: goodabelian} and $\FM=\FM_{\CA}$ is the stack of objects of $\CA$ or, more generally, any stack that satisfies the conclusion of Proposition \ref{defprop: goodabelian}. Joyce defines the structure of an additive super vertex algebra on $H_\ast(\FM_{\CA})$ which we now recall. 

The $\BZ/2$-grading on $H_\ast(\FM)$ is induced by the homological grading. The vacuum vector $\textbf 1$ is given as the image of $1 \in \mathbb Q$ under the natural map 
\[\mathbb Q =  H_\ast(\{0\})= H_\ast(\FM_0)  \to H_\ast(\FM)\,.\]
The translation operator is defined by
\[D(a)=\Psi_\ast(t\otimes a)\]
where $t$ is the generator of $H_2(B\BG_m)$ and $\Psi$ the $B\BG_m$ action on $\FM$. Note that if $z$ is the generator of $H^2(B\BG_m)$ and we interpret the map $\Psi_*$ as an element of
\[\Hom_\BQ\big(H_\ast(B\BG_m)\otimes H_\ast(\FM), H_\ast(\FM)\big)\simeq \Hom_\BQ\big(H_\ast(\FM), H_\ast(\FM)\llbracket z\rrbracket\big)\,,\] 
then the translation operator is defined by the formula $\Psi_\ast = e^{zD}$.

Finally, the state-field correspondence is defined as follows: for $a \in H_\ast(\FM_{\alpha})$ and $b \in H_\ast(\FM_{\beta})$,
\begin{align} 
Y(a, z)b &= (-1)^{\chi(\alpha, \beta)}z^{\chi(\alpha, \beta)+\chi(\beta, \alpha)}\Sigma_\ast\big((e^{zD}\otimes \id)c_{z^{-1}}(\Theta)\cap(a\otimes b)\big)
\end{align}
where 
\[\Theta=\Ext_{12}^\vee+\Ext_{21}\,.\]
Moreover, the $\BZ/2$-grading can be upgraded to a $\BZ$-grading by declaring the degree $j$ part of the vertex algebra to be
\[\bigoplus_{\alpha\in C(\CA)}H_{j-2\chi(\alpha, \alpha)}(\FM_\alpha)\,.\]

\begin{theorem}[{\cite{Jo17}}]\label{thm: joyceVA}
The data above defines a graded vertex algebra on $H_\ast(\FM_{\CA})$.
\end{theorem}

Via the construction of Borcherds explained in Section \ref{subsec: multiplivativeva}, one obtains a Lie algebra structure on $\widecheck H_\ast(\FM)$, which is where wall-crossing formulas for cohomological invariants are written. According to \cite[Proposition 3.24]{Jo17}, the natural projection $H_\ast(\FM) \to H_\ast(\FM^{\rig})$ is surjective and its kernel is precisely $\im(D)$, so we have an identification $\widecheck H_\ast(\FM)\simeq H_\ast(\FM^\rig)$.

\subsection{Liu's vertex algebra in $K$-homology}
\label{subsec: LiuVA}
\cite{Liu} constructs the analogue of the vertex algebra explained before in the $K$-theoretic world. It naturally leads to a multiplicative vertex algebra. The underlying vector space of this multiplicative vertex algebra is not the entire $K$-homology introduced in Section \ref{subsec: K-hom} due to a subtle but very important point regarding convergence, but we need to restrict to a smaller subset. There is some freedom in how to do it, as discussed in \cite[Section 2.2.3]{Liu}, but we will stick to what Liu calls concrete $K$-homology. We will instead call it regular $K$-homology, in analogy with historical terminology in the motivic setting surveyed in Section \ref{subsec: motivic}.
\begin{definition}[Regular $K$-homology]\label{def: regularKhom}
  Let $\FX$ be a derived stack. We let  
\[K_\ast^\reg(\FX)\coloneqq \bigcup_{\substack{Z\in \mathcal C\\
f\colon Z\to \FX}} \textup{im}\big(G(Z)\rightarrow K_\ast(Z)\xrightarrow{f_\ast}K_\ast(\FX)\big)\subseteq K_\ast(\FX)\]
where the union is over all $Z$ in the class $\mathcal C$ of classical, proper, finite type schemes with the resolution property\footnote{The resolution property \cite{totaroresolution} implies that $K^\ast(Z)=K_0(\Vectb(Z))$. It holds for quasi-projective schemes, so in particular all projective schemes are contained in $\mathcal C$. Imposing the resolution property is necessary for somewhat technical reasons in \cite{Liu}, but it would be desirable to remove it since it is unknown if some moduli spaces of complexes satisfy it.}, together with a map $f\colon Z\to \FX$. The morphism $G(Z)\to K_\ast(Z)$, for $Z$ proper, is explained in Section \ref{subsec: K-hom}. 
\end{definition}

A fundamental difference between $K^\reg_\ast(\FX)$ and $K_\ast(\FX)$ is that elements of $K^\reg$ satisfy the finiteness condition \cite[Definition 2.2.2 (iii)]{Liu}. Letting $I(\FX)\subseteq K^\ast(\FX)$ be the augmentation ideal of rank 0 complexes on $\FX$ and $\phi\in K_\ast^\reg(\FX)$, we have 
\begin{equation}
    \label{eq: finiteness}\phi_\pt(I(\FX)^N)=0 \quad \textup{ for }N\gg 1\,.
    \end{equation}
Indeed, this follows from the fact that $I(Z)^N=0$ for sufficiently large $N$ if $Z$ is a finite type scheme with the resolution property \cite[Lemma 2.1.7]{Liu}.

Liu defines a \textit{multiplicative} vertex algebra structure on $K_\ast^\reg(\FM_{\mathcal A})$. The vacuum $\1\in K_\ast^\reg(\FM_0)=K_\ast^\reg(\{0\})$ is defined by 
\[(\1)_S: K^*(S \times \{0\}) \xrightarrow{\operatorname{id}} K^*(S)\,.\]
The translation operator is characterized by the following property: for $\phi \in K_\ast^\reg(\FM), V \in K^\ast(\FM \times S),$ 
\[(D(u)\phi)_S(V) = \phi_S(\Psi^\ast V)\in K^\ast(S)[u^{\pm 1}]\]
where $\Psi^\ast\colon K^\ast(\FM\times S)\to K^\ast(\FM\times B\BG_m\times S)\simeq K^\ast(\FM\times S)[u^{\pm 1}]$. The fact that such $D(u)$ can be regarded as a power series in $1-u$ with coefficients being endomorphisms of $K_\ast^\reg(\FM)$ is established in \cite[Lemma 3.3.7]{Liu}. It will be convenient to consider also $D(u)^{(1)}$ and $D(u)^{(2)}$ as operators $K_\ast^\reg(\FM\times \FM)\to K_\ast^\reg(\FM\times \FM)\llbracket 1-u\rrbracket$ which are defined similarly  using the $B\BG_m$ action on the first or second copy of $\FM$, respectively.

The state-field correspondence is defined as follows for 
$\phi, \psi \in K_*^\reg(\FM)$:
\[Y(\phi, u)\psi = \Sigma_*\big(D(u)^{(1)}(\Gamma(u)\cap (\phi \boxtimes \psi)\big)\] where
\[\Gamma(u) = \Lambda_{-u}(\operatorname{Ext}_{12}^{\vee}) \otimes \Lambda_{-u^{-1}}(\operatorname{Ext}_{21}^{\vee})\,.\]
Let us clarify what exactly is meant by $\Gamma(u)\cap (\phi \boxtimes \psi)$. By definition of $K_*^\reg(\FM)$, there is some $Z\in \CC$ and $f\colon Z\to \FM\times \FM$ such that $\phi\boxtimes \psi=f_\ast \xi$, and we use it to define $\Gamma(u)\cap (\phi \boxtimes \psi)$ as
\[f_\ast\big((\Lambda_{-u^{-1}}(f^\ast\operatorname{Ext}_{12}^{\vee}) \otimes \Lambda_{-u}(f^\ast\operatorname{Ext}_{21}^{\vee}))\cap \xi\big)\in K_\ast^\reg(\FM\times \FM)[(1-u)^{\pm 1}]\]
where the exterior powers are expanded as explained in Lemma \ref{lem: symmetryexpansion}. In particular, observe that if $V\in K^\ast(\FM\times S)$ then 
\[(Y(\phi, u)\psi)_S(V)\in K^\ast(S)[(1-u)^{\pm 1}]\,\]
is a Laurent polynomial in $1-u$, rather than a Laurent series.

\begin{theorem}[{\cite{Liu}}]
The data above defines a multiplicative vertex algebra structure on $K_\ast^\reg(\FM_{\CA})_\BQ$.
\end{theorem}

\subsection{The relation to the $K$-Hall algebra}

Combining the material of the previous sections, one gets a Lie algebra structure on 
\[\widecheck K_\ast^\reg(\FM)_\BQ\coloneqq K_*^\reg(\FM)_\BQ/\operatorname{im}(D(z)-\id)\,.\]

As explained in \cite[Lemma 3.3.17]{Liu}, the morphism $\pi_*: K_*^\reg(\FM)\to K_*^\reg(\FM^\rig)$ factors through $\widecheck K_\ast^\reg(\FM)$. Unlike in the homology case, it is unclear if the induced map $\widecheck K_\ast^\reg(\FM)\to K_*^\reg(\FM^\rig)$ is an isomorphism. 
We have the following diagram involving the different versions of $K$-homology of $\FM$ and $\FM^\rig$:

\begin{center}
\begin{tikzcd}
    K_\ast^\reg(\FM)\arrow[rr, hookrightarrow]\arrow[d, twoheadrightarrow]& \, &K_\ast(\FM)\arrow[d]\\
    \widecheck K_\ast^\reg(\FM)\arrow[r]&K_\ast^\reg(\FM^\rig) \arrow[r, hookrightarrow]& K_\ast(\FM^\rig)
\end{tikzcd}
\end{center}
Note that the the left bottom corner of the diagram (after $\otimes \BQ$) is the Joyce--Liu Lie algebra and the right bottom corner is the the $K$-Hall algebra from Section \ref{subsec: K-Hal-prod}, which is an associative algebra. We regard $K_\ast(\FM^\rig)_\BQ$ also as a Lie algebra with Lie bracket given by the commutator of the $K$-Hall product in Definition \ref{def: KHall}. Our main result from this section is the following:

\begin{theorem}\label{thm: liealgebrahom}
    The morphism $\widecheck K_\ast^\reg(\FM)_\BQ\to K_\ast(\FM^\rig)_\BQ$ is a homomorphism of Lie algebras.
\end{theorem}
\begin{proof}
    Let $\phi, \psi\in K_\ast^\reg(\FM)$ and let $V\in K^\ast(\FM^\rig\times S)$. We denote by $\tilde \phi, \tilde \psi$ their images in $K_\ast(\FM^\rig)$. We have by definition
    \begin{align}\label{eq: proofhomomorphism1}
        (Y(\phi, u)\psi)_S(\pi^\ast V)=(\phi\boxtimes \psi)\left(\Gamma(u)\otimes\Psi_1^\ast\Sigma^\ast\pi^\ast V\right)\,,
    \end{align}
    which is an element of $K^\ast(S)[(1-u)^\pm]$, that we denote by $y$. By Lemma \ref{lem: symmetryexpansion} (see also Remark \ref{rmk: eulerclassTSigma}), the expansion $y_-$ of $y$ in $K^\ast(S)\lpp u^{-1}\rpp$ is precisely what appears in the expression \eqref{eq: hallcoproduct}, and hence
    \[[u^0]y_-=(\tilde \phi\ast \tilde \psi)_S(V)\,.\]
    In a similar way we conclude that $[u^0]y_+=(\tilde \psi\ast \tilde \phi)_S(V)$. The conclusion then follows from Remark \ref{rmk: polesresidue}.
\end{proof}
In light of the previous theorem, the following definition is natural:
\begin{definition}
    We define the Lie subalgebra $\BK^\reg(\CA)$ of $\BK(\CA)$ (with the $K$-Hall commutator) as
    \[\BK^\reg(\CA)\coloneqq \im\big(K_\ast^\reg(\FM)_\BQ\to K_\ast(\FM^\rig)_\BQ\big)\,.\]
\end{definition}

This Lie subalgebra is, from our point of view, a natural analogue of $\BM^\reg(\CA)$ in the motivic setting presented in Section \ref{subsec: motivic}. As in the motivic setting, $\BK^\reg(\CA)$ is not an associative subalgebra; indeed, in Example \ref{ex: deltanotregular} it is observed that 
\[\varepsilon_1\in \BK^\reg(\Vect)\textup{ but }\varepsilon_1\ast \varepsilon_1\notin  \BK^\reg(\Vect)\,.\]
The following is a natural analogue of the no-pole theorem:

\begin{conjecture}[No-poles\footnote{We warn the reader that this conjecture, with our current definition of $\BK^\reg(\CA)$, might very well fail in cases where the good moduli spaces are proper but not projective, due to the requirement of schemes in class $\mathcal C$ to have the resolution property. We believe the finiteness part of the conjecture should still hold even if that is not the case.}]\label{conj: nopole}
We have $\varepsilon_\alpha^\mu\in \BK^\reg(\CA)$. In particular, $\varepsilon_\alpha^\mu$ satisfies the finiteness condition \eqref{eq: finiteness}.
\end{conjecture}
It would be interesting to have a direct proof of this conjecture. The conjecture does hold in the following situations:
\begin{proposition}\label{prop: nopolesknown}
We have $\varepsilon_\alpha^\mu\in \BK^\reg(\CA)$ if one of the following holds:
\begin{enumerate}
    \item There are no $\mu$-strictly semistable objects in class $\alpha$ and the moduli space $M_\alpha^\mu$ satisfies the resolution property.
    \item There is a path of stability conditions from $\mu$ to $\mu'$ satisfying the conditions of Theorem \ref{thm: generalwc} and the conjecture holds for the stability condition $\mu'$.
    \item There is a framing functor, as in Section \ref{subsec: framingdefinitions}, and the moduli spaces of Joyce--Song pairs satisfy the resolution property. Moreover, when that is the case, $\varepsilon_\alpha^\mu$ are the image of Liu's classes $z_\alpha(\mu)\in \widecheck K_\ast^\reg(\FM)_\BQ$ via the morphsim \[\widecheck K_\ast^\reg(\FM)_\BQ\to K_\ast(\FM^\rig)_\BQ\,.\]
\end{enumerate}
\end{proposition}
\begin{proof}
Part (1) is clear since when there are no $\mu$-strictly semistable objects the classes $\varepsilon_\alpha^\mu=\delta_\alpha^\mu$ are, by definition, the image of $\CO_{M_\alpha^\mu}^\vir$ by the map
\[G(M_\alpha^\mu)\to K_\ast(M_\alpha^\mu)\to K_\ast(\FM_\alpha^{\mu, \rig})\,.\]
Part (2) is an immediate consequence of $\BK^\reg(\CA)$ being a Lie subalgebra and the wall-crossing formula of Theorem \ref{thm: generalwc}. 

The claim that our classes are the image of Liu's classes follows from comparing Theorem \ref{thm: comparisonJL} with \cite[Theorem 4.2.5]{Liu}, and part (3) follows.
\end{proof}

\section{From $K$-theory to cohomology}
\label{sec: Ktocoh}

So far, our approach to generalized invariants and wall-crossing has been entirely in the realm of $K$-theory. We now explain how these relate to wall-crossing formulas for cohomological invariants, and how $K$-theory formulas can be used to deduce cohomological ones. The main result of this section is that Joyce's homological invariants, when defined, are homological lifts of our classes $\varepsilon_\alpha^\mu$, in the sense of Definition \ref{def: homlift}. We conjecture that (canonical) homological lifts exist also in the absence of a framing functor. Moreover, Theorem \ref{thm: homliftsLiealgebras} and Proposition \ref{prop: injectivityhomlift} together imply that in general homological lifts of $\varepsilon$ classes satisfy automatically the same wall-crossing formulas, at least when we restrict ourselves to (algebraic) tautological insertions -- this reproves Joyce's cohomological wall-crossing formulas for (algebraic) tautological integrals.

\subsection{The Riemann--Roch morphism and homological lifts}\label{subsec: RRhomlift}

The virtual Riemann--Roch theorem \cite[Corollary 3.4]{FTriemannroch} (see also \cite[Theorem 6.12]{khan} for a version that does not require an embedding into a smooth ambient space) is the statement that, for a proper scheme $M$ with a 2-term perfect obstruction theory, we can write $K$-theoretic invariants as intersection numbers:
\[\chi(M, \CO_M^\vir\otimes V)=\int_{[M]^\vir}\ch(V)\td(T^\vir_M)\,.\]
This motivates the following definition:

\begin{definition}\label{def: RRmorphism}
Let $\FX$ be a derived stack\footnote{We remind the reader that derived stacks in this paper are assumed to have perfect tangent complex, so that $\td(\BT_\FX)$ makes sense.} (for instance, $\FX=\FM_\CA$ for a good abelian category $\CA$). We define the morphism
\begin{align*}
\tau=\tau_\FX\colon K^\ast(\FX)&\to H^\ast(\FX)\\
\alpha &\mapsto \ch(\alpha) \td(\BT_{\FX})
\end{align*}
and its dual version
\begin{align*}
\tau^\vee \colon H_\ast(\FX)&\to K^\ast(\FX)_\BQ^\vee\,.
\end{align*}
\end{definition}

When $\FM=\FM_\CA$, the derived tangent bundle $\BT_{\FM}$ is given by the restriction of $\Ext_{12}[1]$ to the diagonal \eqref{eq: derivedtangentFM}. The virtual Riemann--Roch says that, if $\mu$-stable is equivalent to $\mu$-semistable for objects of class $\alpha$, then
\[(\varepsilon_\alpha^\mu)_\pt=\tau^\vee([M_\alpha^\mu]^\vir)\,.\]

We remind the reader that, as explained in Section \ref{subsec: JoyceVA}, when $\FM=\FM_\CA$ is the moduli stack of objects in an abelian category, $H_\ast(\FM)$ carries the structure of a vertex algebra and $K_\ast^\reg(\FM)$ carries the structure of a (multiplicative) vertex algebra. Ideally, we would have liked to state that an upgrade of $\tau^\vee$ defines a homomorphism of vertex algebras. Unfortunately, there is no way to lift $\tau^\vee$ to a map $H_\ast(\FM)\to K_\ast(\FM)$, so strictly speaking this is not possible. On the other hand, $\tau^\vee$ is not compatible with Kunneth maps since there is no Kunneth map on $(K^\ast)^\vee$. We introduce the notion of homological lift, which works around these technical points.

\begin{definition}[Homological lift]\label{def: homlift}
    Let $\phi\in K_\ast(\FX)$. We say that $A\in H_\ast(\FX)$ is a homological lift of $\phi$ if for any derived stack $S$ the diagram
    \begin{center}
        \begin{tikzcd}
            K^\ast(\FX\times S)\arrow[d, "\tau_{\FX\times S}"]\arrow[r, "\phi_S"] &K^\ast(S)\arrow[d, "\tau_S"]\\
            H^\ast(\FX\times S)\arrow[r, "- \slash A"]& H^\ast(S)
        \end{tikzcd}
    \end{center}
            commutes, where $\slash$ is the slant product.
\end{definition}
\begin{remark}\label{rmk: homliftfiniteness}
Admitting a homological lift is a strong condition. For example, if $\phi$ admits a homological lift then it satisfies the finiteness condition that $\phi_\pt(I(\FX)^N)=0$ for all sufficiently large $N$. This is a consequence of the fact that $\tau$ maps $I(\FX)$ to $H^{\geq 2}(\FX)$. 
\end{remark}
Taking $S=\pt$ in the definition, it follows that $A$ being a homological lift of $\phi$ implies that $\tau^\vee(A)=\phi_\pt$. However, being a homological lift has better compatibility with the Kunneth maps in (K-)homology, as illustrated by the following lemma:
\begin{lemma}\label{lem: homlift}The following properties of homological lifts hold:
\begin{enumerate}
    \item If $A_i\in H_\ast(\FX_i)$ is a homological lift of $\phi_i\in K_\ast(\FX_i)$ then $A_1\boxtimes \ldots\boxtimes A_n\in H_\ast(\FX_1\times \ldots\times \FX_n)$ is a homological lift of $\phi_1\boxtimes \ldots\boxtimes \phi_n\in K_\ast(\FX_1\times \ldots\times \FX_n)$. In particular,
    \[\tau^\vee(A_1\boxtimes \ldots\boxtimes A_n)=(\phi_1\boxtimes \ldots\boxtimes \phi_n)_\pt\,.\]
    \item If $A\in H_\ast(\FX)$ is a homological lift of $\phi\in K_\ast(\FX)$ and $V\in K^\ast(\FX)$ then $\ch(V)\cap A$ is a homological lift of $V\cap \phi$.
    \item If $f\colon \FX\to \FY$ and $A\in H_\ast(\FX)$ is a homological lift of $\phi\in K_\ast(\FX)$ then $f_\ast\big(\td(\BT_f)\cap A)\in H_\ast(\FY)$ is a homological lift of $f_\ast\phi\in K_\ast(\FY)$. 
    \item If $X$ is a quasi-smooth proper algebraic space, then the homological virtual fundamental class $[X]_H\in H_\ast(X)$ is a homological lift of the $K$-theoretic fundamental class $[X]_K\in K_\ast(X)$ (cf. Theorem \ref{thm: deltainvariants}).
\end{enumerate}
\end{lemma}
\begin{proof}
    The first 3 properties are obtained straight from the definition. The last one is a consequence of the virtual Grothendieck--Riemann--Roch for the projection maps $X\times S\to S$, which we prove in Lemma \ref{lem: GRR}. 
    To apply the lemma, two observations are necessary: first, under the assumptions on $X$ the projection $\pi_S\colon X\times S\to S$ admits pushforwards on $K^\ast$, and the $K$-theoretic fundamental class $[X]_K$ consists precisely of the collection of all pushforwards $\pi_{S\ast}\colon K^\ast(X\times S)\to K^\ast(S)$. Secondly, Khan's virtual fundamental class 
    \[[X\times S/S]\in H^\ast_{\textup{BM}}(X\times S/S)\]
    is the pullback of the virtual fundamental class $[X]_H\in H_\ast^{\textup{BM}}(X)=H_\ast(X)$, by the base change formula \cite[Theorem 3.13]{khanVFC}.\qedhere
\end{proof}

\begin{lemma}[Virtual Grothendieck--Riemann--Roch]\label{lem: GRR}
Let $\FX, \FZ$ be derived stacks and $f\colon \FX\to \FZ$ be a quasi-smooth proper representable morphism. Then the following diagram commutes
\begin{center}
\begin{tikzcd}
    K^\ast(\FX)\arrow[r, "f_\ast"] \arrow[d, "\tau"] & 
    K^\ast(\FZ)\arrow[d, "\tau"] \\
    H^\ast(\FX)\arrow[r, "f_{!}"] & H^\ast(\FZ)
    \end{tikzcd}
\end{center}
where $f_!$ is the composition
\[H^\ast(\FX)\xrightarrow{\circ [\FX/\FZ]} H_\ast^{\textup{BM}}(\FX/\FZ)\xrightarrow{f_\ast} H^\ast(\FZ)\]
and $[\FX/\FZ]\in H_\ast^{\textup{BM}}(\FX/\FZ)$ is Khan's virtual fundamental class, see \cite[Section 3.4]{khanVFC} for more details.\footnote{Technically, \cite{khanVFC} works with cohomology theories in the algebraic category, but here we are regarding $[\FX/\FZ]$ in Borel--Moore homology, in the sense of the 6 functor formalism in the topological category, by applying the Betti realization morphism, as discussed in the proof. Let us also remark that the cohomology of derived stacks defined through the 6 functor formalism on the topological category agrees with the cohomology of the topological realization of the stack \cite[Proposition 2.8]{khanequivcoh}.}
\end{lemma}
\begin{proof}
    The Chern character map $\ch\colon K^\ast(-)\to H^\ast(-)$ can be factored as the composition
    \[K^\ast(-)\rightarrow K^\ast(-)_\BQ\xrightarrow{(A)} K^{\textup{ét}}(-)_\BQ\xrightarrow{(B)} KH^{\textup{ét}}(-)_\BQ\xrightarrow{\ch} H^\ast_{\textup{mot}}(-)\xrightarrow{(C)} H^\ast(-)\,.\]
Here, $K^{\textup{ét}}(-)_\BQ$ and $KH^{\textup{ét}}(-)_\BQ$ are the étale localizations of $K^{\ast}(-)_\BQ$ and $KH(-)_\BQ$, $KH$ is the homotopy invariant version of $K$-theory, and $H^\ast_{\textup{mot}}$ is the motivic cohomology with coefficients in $\bigoplus_{n\in \BZ}\BQ(n)[2n]$. We refer to \cite[Sections 4, 5]{khan} and \cite[Example 2.10, 2.13]{khanVFC} for details. A Grothendieck--Riemann--Roch theorem for $KH^{\textup{ét}}(-)_\BQ\xrightarrow{\ch} H^\ast_{\textup{mot}}(-)$ is proven in \cite[Corollary 3.25]{khanVFC}. The maps $(A), (B)$ are easily seen to commute with pushforward of maps, when defined. The map $(C)$ is induced by the Betti realization (see \cite[Section 2]{ayoub}), which also commutes with pushforwards as shown in loc. cit. Note that we are defining $[\FX/\FZ]\in H^{\textup{BM}}_\ast(\FX/\FZ)$ as the image of $[\FX/\FZ]\in H^{\textup{BM,mot}}_\ast(\FX/\FZ)$ under the Betti realization morphism  $H^{\textup{BM,mot}}_\ast(\FX/\FZ)\to  H^{\textup{BM}}_\ast(\FX/\FZ)$, so it is clear that $f_!$ commutes with $(C)$. 
\end{proof}

Recall that we have Lie brackets on $H_\ast(\FM^\rig)$ and on $K_\ast(\FM^\rig)$. The former is obtained from the vertex algebra structure on $H_\ast(\FM)$ and the isomorphism $H_\ast(\FM^\rig)\simeq \widecheck H_\ast(\FM)$, and the latter is the commutator of the $K$-Hall product. 

\begin{theorem}\label{thm: homliftsLiealgebras}
    If $A_1, A_2\in H_\ast(\FM^\rig)$ are homological lifts of $\phi_1, \phi_2\in K_\ast(\FM^\rig)$ then $[A_1, A_2]$ is a homological lift of $[\phi_1, \phi_2]$.
\end{theorem}

\subsection{Proof of compatibility of Lie algebras and homological lifts}
\label{subsec: homliftsLie}

This Section is dedicated to the proof of Theorem \ref{thm: homliftsLiealgebras}. We start by observing that we have the following diagram:
\begin{center}
    \begin{tikzcd}
        K^\ast(\FM^\rig\times S)\arrow[rr, "{[\phi_1, \phi_2]_S}", bend left]\arrow[d, "\tau"]\arrow[r, "b_K"] &K^\ast(\FM^\rig\times \FM^\rig\times S)\arrow[d, "\tau"] \arrow[r, "(\phi_1\boxtimes\phi_2)_S"]& K^\ast(S) \arrow[d, "\tau"]\\
        H^\ast(\FM^\rig\times S)\arrow[r, "b_H"]\arrow[rr, "{\slash [A_1, A_2]}", bend right]& H^\ast(\FM^\rig\times \FM^\rig\times S)\arrow[r, "\slash(A_1\boxtimes A_2)"] &H^\ast(S)
    \end{tikzcd}
\end{center}
The maps $b_K, b_H$ above are, roughly speaking, the dual maps to the Lie brackets on $K$-homology and homology, respectively; we will describe them explicitly below. The square on the right commutes by hypothesis and Lemma \ref{lem: homlift}(1), so it is enough to prove that the left square commutes as well. For ease of notation, we show it when $S=\pt$, but the general case is the same. 

First of all, the map $b_H$, when restricted to $H^\ast(\FM_{\alpha+\beta}^\rig)\to H^\ast(\FM_{\alpha}^\rig\times \FM_{\beta}^\rig)$, is given by\footnote{To be more precise, the formula above defines a map $H^\ast(\FM_{\alpha+\beta})\to H^\ast(\FM_{\alpha}\times \FM_{\beta})$ which descends to $b_H\colon H^\ast(\FM_{\alpha+\beta}^\rig)\to H^\ast(\FM_{\alpha}^\rig\times \FM_{\beta}^\rig)$. The same applies to the desctiption of $m_K$ below.}
\[b_H=\Res_{z=0} (-1)^{\chi(\alpha, \beta)}z^{\chi_\sym(\alpha, \beta)}c_{z^{-1}}(\Theta)\Psi_1^\ast \Sigma^\ast(-)\,,\]
where we are regarding $\Psi_1^\ast$ as a map 
\[H^\ast(\FM\times \FM)\to H^\ast(B\BG_m\times \FM\times \FM)\simeq H^\ast(\FM\times \FM)\llbracket z\rrbracket\,.\]
The map $b_K$ is the anti-symmetrization of the ``coproduct'' $m_K$, i.e. $b_K=m_K-\sigma^\ast\circ m_K$ where $\sigma$ switches the two copies of $\FM^\rig$ and 
\[m_K=[u^0]\Gamma_-(u)\otimes \Psi_1^\ast \Sigma^\ast(-)\,.\]

For a derived stack $\FX$, we introduce the completion 
\[H^\ast(\FX)[(1-u)^{\pm 1}]^{\widehat{\hspace{0.2cm}}} \subseteq H^\ast(\FX)\lpp (1-u)^{-1}\rpp \]
of $H^\ast(\FX)[(1-u)^{\pm 1}]$ as the set of Laurent series in $(1-u)^{-1}$ with coefficients in $H^\ast(\FX)$ with the property that, for any $N\in \BZ_{+}$, their image modulo $H^{\geq N}(\FX)$ is a Laurent polynomial in $1-u$; in other words, the coefficient of $(1-u)^{-i}$ has large cohomological degree for large $i$. Since homology is the direct sum over degrees, we have a pairing
\[\langle -, -\rangle\colon H_\ast(\FX)\otimes H^\ast(\FX)[(1-u)^{\pm 1}]^{\widehat{\hspace{0.2cm}}}\to \BQ[(1-u)^{\pm 1}]\,.\]
There is a residue map $\Res_{u=1}\colon H^\ast(\FX)[(1-u)^{\pm 1}]^{\widehat{\hspace{0.2cm}}}\to H^\ast(\FX)$ and it satisfies
\[\langle -, \Res_{u=1}-\rangle=\Res_{u=1}\langle - , -\rangle\,.\]

An element of $H^\ast(\FX)[(1-u)^{\pm 1}]^{\widehat{\hspace{0.2cm}}}$ admits an expansions in $u$ or in $u^{-1}$. By the proof of Lemma \ref{lem: symmetryexpansion}, for any $V\in K^\ast(\FX)$ we can define an element of $H^\ast(\FX)[(1-u)^{\pm 1}]^{\widehat{\hspace{0.2cm}}}$
with the property that its $u$-expansion is $\ch(\Lambda_{-u}(V))$ and its $u^{-1}$-expansion is $(-u)^{\rk(V)}\ch\big(\Lambda_{-u^{-1}}(V^\vee)\otimes \det(V)\big)$. Note that the algebraic splitting principle is not necessary for this, since the statement only depends on the image of $V$ in topological $K$-theory, i.e. on the homotopy class of the map $\FX\to \FM_{\Perf}\simeq  \BZ\times BU$ induced by a perfect complex representing $V$; therefore the splitting principle applies even if algebraically $V$ does not admit a resolution by vector bundles.

We recall the complex
\[\Gamma(u)=\Lambda_{-u}(\Ext_{12}^\vee)\otimes \Lambda_{-u^{-1}}(\Ext_{21}^\vee)\,.\]
We will denote by $\ch(\Gamma(u))$ the element of $H^\ast(\FX)[(1-u)^{\pm 1}]^{\widehat{\hspace{0.2cm}}}$ obtained from $\Gamma(u)$ as explained above. Its $u^{-1}$-expansion is equal to $\ch(\Gamma_-(u))$ and its $u$-expansion is equal to $\ch(\Gamma_+(u))$ where
\[\Gamma_+(u)=(-u)^{-\rk_{21}}\Lambda_{-u}(\Ext_{12}^\vee)\otimes \Lambda_{-u}(\Ext_{21})\otimes \det(\Ext_{21}^\vee)=\sigma^\ast \Gamma_-(u^{-1})\,.\]
Using the above relation, we find that
\[\sigma^\ast \circ m_K=[u^0]\Gamma_+(u)\otimes \Psi_1^\ast \Sigma^\ast(-)\,;\]
combining this with Remark \ref{rmk: polesresidue} we obtain
\[\tau(b_K(V))=\Res_{u=1} u^{-1}  \ch\big(\Gamma(u)\big)\ch\big(\Psi_1^\ast \Sigma^\ast V\big)\td(T_{\alpha})\td(T_\beta) \in H^\ast(\FM^\rig_\alpha\times \FM^\rig_\beta)\]
where $T_\alpha$ is shorthand for $\BT_{\FM_\alpha}$. We will use the following elementary lemma to compare $\Gamma(u)$ and $c_{z^{-1}}(\Theta)$:
\begin{lemma}\label{lem: KtoH}
Let $V\in K^\ast(\FX)$. 
We have equalities
\begin{align*}
\ch\big(\Lambda_{-u^{-1}}(V^\vee) \big)\cdot \td\big(V\otimes u\big)&=z^{\rk(V)}c_{z^{-1}}\big(V\big)\\
\ch\big(\Lambda_{-u}(V^\vee) \big)\cdot \td\big(V\otimes u^{-1}\big)&=(-z)^{\rk(V)}c_{z^{-1}}\big(V^\vee\big)
\end{align*}
in $H^\ast(\FX)[(1-u)^{\pm 1}]^{\widehat{\hspace{0.2cm}}}$ after the change of variable $u=e^z$.
\end{lemma}
\begin{proof}
All the 4 quantities are multiplicative in $V$, in the sense that they all satisfy $f(V+V')=f(V)f(V')$. By the splitting principle (as before, we do not need it algebraically, but only in topological $K$-theory), it is enough to verify the claim when $V$ is a line bundle. In this case, the equalities turn into
\begin{align*}
(1-e^{-z-\alpha})\frac{z+\alpha}{1-e^{-z-\alpha}}&=z(1+\alpha/z)\\
(1-e^{z-\alpha})\frac{-z+\alpha}{1-e^{z-\alpha}}&=-z(1-\alpha/z)
\end{align*}
where $\alpha=c_1(V)$.\qedhere
\end{proof}

Using the previous lemma for $V=\Ext_{12}, \Ext_{21}$, we obtain the identity
\[\ch(\Gamma(u))=(-1)^{\chi(\alpha, \beta)}z^{\chi_\sym(\alpha, \beta)}c_{z^{-1}}(\Theta)\td(-\Ext_{12}\otimes u^{-1})\td(-\Ext_{21}\otimes u)\,.\]
Note also that
\[\Psi_1^\ast\Sigma^\ast T_{\alpha+\beta}=T_\alpha+T_\beta-\Ext_{12}\otimes u^{-1}-\Ext_{21}\otimes u\,.\]
Using the last two equations together, one obtains the equality 
\begin{align*}\ch\big(\Gamma(u)\big)&\ch\big(\Psi_1^\ast \Sigma^\ast V\big)\td(T_{\alpha})\td(T_\beta)\\
&=(-1)^{\chi(\alpha, \beta)}z^{\chi_\sym(\alpha, \beta)}c_{z^{-1}}(\Theta)\Psi_1^\ast \Sigma^\ast \big(\ch(V)\td(T_{\alpha+\beta})\big)
\end{align*}
in the completion $H^\ast(\FX)[(1-u)^{\pm 1}]^{\widehat{\hspace{0.2cm}}}$, after the change of variables $u=e^z$. Applying $\Res_{u=1}u^{-1}=\Res_{z=0}$ (see Remark \ref{rmk: residuechangevariables}) to both sides gives the equality
\[\tau(b_K(V))=b_H(\tau(V))\]
as desired.

\subsection{Homological lifts of $\varepsilon$ classes}\label{subsec: homliftsepsilon}
In the (co)homological setting, Joyce defines elements that we denote by
\[\varepsilon_\alpha^{\mu, H}\in H_{2-2\chi(\alpha, \alpha)}(\FM^\rig_\alpha)\,\]
assuming the existence of framing functors. Note that $1-\chi(\alpha, \alpha)$ is the dimension of the derived stack $\FM^\rig_\alpha$. When all semistable objects are stable, these are simply the pushforward of the virtual fundamental class $[M_\alpha^\mu]^\vir$ to $\FM^\rig_\alpha$. The general construction utilizes the moduli of Joyce--Song pairs $P_{(1,\alpha)}^{\mu_+}$, where stable=semistable holds, and defines $\varepsilon_\alpha^{\mu, H}$ recursively through a formula that resembles \eqref{eq: framingdefinitionepsilon}. More precisely,
    \begin{align}\label{eq: framingdefinitionepsilonhomology}
        \Pi_\ast\big(c_{\rk}(\BT_\Pi)&\cap \varepsilon_{(\alpha, 1)}^{\mu_+, H}\big)\\&=\sum_{\substack{\alpha_{1}+\ldots+\alpha_n=\alpha\\
        \mu(\alpha_i)=\mu(\alpha)}}\frac{(-1)^{n-1}}{n!}\lambda(\alpha_1)\big[\big[\ldots\big[\varepsilon_{\alpha_1}^{\mu, H},\varepsilon_{\alpha_2}^{\mu, H}\big],\ldots, \varepsilon_{\alpha_n}^{\mu, H}\big]\,. \nonumber
    \end{align}
    where, on the left hand side, $\varepsilon_{(\alpha, 1)}^{\mu_+, H}$ is defined as the homological virtual fundamental class $[P_{(1,\alpha)}^{\mu_+}]^\vir$. The equality holds in the Lie algebra $H_\ast(\FM^\rig)$.
\begin{theorem}\label{thm: homliftsjoyce}
    Assume that there is a framing functor as in Section \ref{sec: framing}. Then Joyce's classes $\varepsilon^{\mu, H}_\alpha$ are homological lifts of our classes $\varepsilon_\alpha^\mu$. 
\end{theorem}
\begin{proof}
    We argue by induction on $\alpha$, as in the proof of Theorem \ref{thm: comparisonJL}. By induction and Theorem \ref{thm: homliftsLiealgebras}, $\big[\big[\ldots\big[\varepsilon_{\alpha_1}^{\mu, H},\varepsilon_{\alpha_2}^{\mu, H}\big],\ldots, \varepsilon_{\alpha_n}^{\mu, H}\big]$ is a homological lift of $\big[\big[\ldots\big[\varepsilon_{\alpha_1}^{\mu},\varepsilon_{\alpha_2}^{\mu}\big],\ldots, \varepsilon_{\alpha_n}^{\mu}\big]$ for all non-trivial ($n>1$) partitions of $\alpha$. Therefore, to prove that the remaining term $\varepsilon_{\alpha}^{\mu, H}$ on the right hand side of \eqref{eq: framingdefinitionepsilonhomology} is a homological lift of the remaining term $\varepsilon_{\alpha}^{\mu}$ on the right hand side of \eqref{eq: framingdefinitionepsilon}, it is enough to show that the left hand side of \eqref{eq: framingdefinitionepsilonhomology} is a homological lift of the left hand side of \eqref{eq: framingdefinitionepsilon}.

    By Lemma \ref{lem: homlift}(4), the class $\varepsilon_{(\alpha, 1)}^{\mu+, H}$ is a homological lift of $\varepsilon_{(\alpha, 1)}^{\mu+}$. Note that we have the identity
    \[\ch(\Lambda_{-1}(\BT_\Pi^\vee))\td(\BT_\Pi)=c_{\rk}(\BT_\Pi)\]
    by setting $z=0$ (i.e. $u=1$) in Lemma \ref{lem: KtoH}. By (2) and (3) in Lemma \ref{lem: homlift} 
    it follows that $\Pi_\ast\big(c_\rk(\BT_\Pi)\cap \varepsilon_{(\alpha, 1)}^{\mu+, H}\big)$ is a homological lift of $\Pi_\ast\big(\Lambda_{-1}(\BT_\Pi^\vee)\cap \varepsilon_{(\alpha, 1)}^{\mu+}\big)$, as we wanted.\qedhere
\end{proof}

We expect that (canonical) homological lifts exist, even when there is no framing functor. We state this as a conjecture:

\begin{conjecture}[Homological lift and homogeneity]\label{conj: homlift}
If $\mu$ is a stability condition as in Assumption \ref{ass: stability} and $\alpha\in C(\CA)_{\pe}$, then the classes $\varepsilon_\alpha^\mu\in K_\ast(\FM^\rig)$ admit a homological lift with the expected homological degree:\[\varepsilon_\alpha^{\mu, H}\in H_{2-2\chi(\alpha, \alpha)}(\FM^\rig)\,.\]
\end{conjecture}

As explained in Remark \ref{rmk: homliftfiniteness}, this conjecture is closely related to the no-pole Conjecture~\ref{conj: nopole}. The homogeneity property with respect to the homological degree is an even stronger property.

\begin{remark}\label{rmk: homliftswallcrossing}
The homological lift conjecture is compatible under wall-crossing by Theorem \ref{thm: homliftsLiealgebras}. This is a consequence of the homological Lie bracket being well behaved with respect to the degree, in the sense that it maps
\[[-,-]\colon H_{2-2\chi(\alpha, \alpha)}(\FM^\rig_\alpha)\otimes H_{2-2\chi(\beta, \beta)}(\FM^\rig_\beta)\to H_{2-2\chi(\alpha+\beta, \alpha+\beta)}(\FM^\rig_{\alpha+\beta})\,.\]
\end{remark}

\subsection{Cohomological descendents from $K$-theory descendents}
\label{subsec: homliftdescendents}

Suppose that we have homological lifts $\varepsilon_\alpha^{\mu, H}$ of $\varepsilon_\alpha^\mu$. The $K$-theoretic classes $\varepsilon_\alpha^\mu$ satisfy wall-crossing formulas as shown in Section \ref{sec: NAL}. The homological lifts are expected to satisfy the same wall-crossing formulas, and indeed Joyce proves such formulas for his invariants, under stronger assumptions. We will now explain that (co)homological wall-crossing formulas follow formally from $K$-theoretic wall-crossing formulas, at least when we restrict ourselves to tautological integrals on moduli of sheaves/complexes on a variety $X$. By \eqref{eq: epsilonwc} and Theorem \ref{thm: homliftsLiealgebras}, both sides of
\[\varepsilon_{\alpha}^{\mu', H}
   \overset{?}{=}\sum_{\alpha_1+\ldots+\alpha_l=\alpha}\tilde U(\alpha_1, \ldots, \alpha_n; \mu, \mu')\cdot [[\ldots [\varepsilon_{\alpha_1}^{\mu, H}, \varepsilon_{\alpha_2}^{\mu, H}], \ldots, ],\varepsilon_{\alpha_n}^{\mu, H}]\]
are homological lifts of the same element $\varepsilon_\alpha^{\mu'}$, and in particular their images under $\tau^\vee$ are the same. Hence, deducing cohomological wall-crossing from $K$-theoretic wall-crossing is related to injectivity properties of $\tau^\vee$, or equivalently surjectivity of $\tau$. Of course, $\tau$ is far from surjective since for example there are no odd classes in the image.

Given a smooth projective variety $X$, suppose that $\CA$ is either $\Coh(X)$ or, more generally, the heart of a $t$-structure on $D^b(X)$. Then, there is a universal sheaf/complex $\CF$ on $\FM\times X$. 

\begin{definition}[Tautological classes]\label{def: tautclasses}
Given $\gamma\in H^\ast(X)$ and $k\geq 1$, we let
\[\ch_k(\gamma)=p_\ast\big(\ch_k(\CF)q^\ast\gamma\big)\in H^\ast(\FM)\]
where $p, q$ are the projections of $\FM\times X$ onto $\FM$ and $X$, respectively. 

We say that a monomial 
\begin{equation}
    \label{eq: monomialchk}
D=\prod_{i=1}^n \ch_{k_i}(\gamma_i)\end{equation}
is algebraic if $\gamma_1\otimes \ldots \otimes \gamma_n$ is in the image of the cycle class map $\textup{CH}^\ast(X^n)\to H^\ast(X^n)$. 

We denote by 
\[H^\ast_{\textup{taut,alg}}(\FM)\subseteq H^\ast_{\textup{taut}}(\FM) \subseteq H^\ast(\FM)\]
the subalgebra spanned by algebraic monomials and the subalgebra of tautological classes, respectively.
\end{definition}

\begin{remark}\label{rmk: hodgeconj}
    If the Hodge conjecture holds for powers of $X$, then $H^\ast_{\textup{taut,alg}}(\FM)$ is spanned by Hodge balanced monomials of the form \eqref{eq: monomialchk} with $\gamma_i\in H^{p_i, q_i}(X)$ and $\sum_{i=1}^n p_i=\sum_{i=1}^n q_i$. Suppose $M$ is a moduli space with a virtual fundamental class. Since virtual fundamental classes $[M]^\vir$ are algebraic, unbalanced monomials have trivial integral since\footnote{This argument was explained to the second author by Y. Bae.}
    \begin{align*}
        \nonumber\int_{[M]^\vir} &\prod_{i=1}^n \ch_{k_i}(\gamma_i)&\\
       & =\int_{X^n} (\gamma_1\otimes \ldots \otimes \gamma_n)\cap q_\ast\left((\ch_{k_1}(\CF_1)\ldots \ch_{k_n}(\CF_n))\cap p^\ast[M]^\vir\right)=0\,,
        \end{align*}
    where $\CF_i$ is the pullbacks of $\CF$ along the $i$-th projection $M\times X^n\to M\times X$ and $q\colon M\times X^n\to X^n$, $p\colon M\times X^n\to M$ are the two obvious projections.  
\end{remark}

The following proposition provides a tool to lift equalities from $K$-homology to homology.

\begin{proposition}\label{prop: injectivityhomlift}
    Let $A, B\in H_\ast(\FM)$ and suppose that $\tau^\vee(A)=\tau^\vee(B)$. Then $A$ and $B$ agree on algebraic tautological classes, i.e. 
    \[\int_A D=\int_B D\]
    for any $D\in H^\ast_{\textup{taut,alg}}(\FM)$.

    If $A, B$ are algebraic, in the sense that they are in the image of $\textup{CH}_\ast(Z)\to H_\ast(Z)\xrightarrow{f_\ast} H_\ast(\FM)$ for some morphism $f\colon Z\to \FM$ from $Z$ a proper algebraic space, and the Hodge conjecture holds for powers of $X$ then the statement holds for any $D\in H^\ast_{\textup{taut}}(\FM)$.
\end{proposition}
\begin{proof}
    Since $\td(\BT_\FM)$ can itself be expressed in terms of algebraic tautological classes \cite[Proposition 3.3]{shencobordism}, it is enough to prove that for any monomial \eqref{eq: monomialchk} there is some $V\in K^\ast(\FM)$ such that
    \[\int_A D=\int_A \ch(V)=\int_B \ch(V)=\int_B D\,.\]
    Note that $D$ is given by
    \[p_\ast\left( \prod_{i}\ch_{k_i}(\CF_i)q^\ast (\gamma_1\otimes \ldots \otimes \gamma_n)\right)\]
    where $\CF_i, p, q$ are as in Remark \ref{rmk: hodgeconj}. We will construct the class $V$ using Adams operations to isolate $\ch_{k_i}(\CF_i)$ from $\ch(\CF_i)$.

    Recall that the Adams operation $\psi^j$ on $K$-theory has the property that $\ch_k(\psi^j \CF)=j^k \ch_k(\CF)$. Let $N=N'+n\cdot \dim(X)$ where $N'$ is such that  
    \[A, B\in H_{\leq N'}(\FM)\,.\]

    Given $k\leq N$, we can pick rational numbers $a_j^k$ such that 
    \[\sum_{i=0}^N a_j^k j^m\]
    is $1$ when $m=k$ and $0$ when $m\in [0,N]\setminus \{k\}$; indeed, $a_{j}^k$ are the entries in the inverse of a Vandermonde matrix. Now let 
    \[\widetilde \CF_i=\sum_{j=0}^N a_j^{k_i}\psi^j(\CF_i)\in K^\ast(\FM\times X^n)\,.\]
    We have
    \[\ch(\widetilde \CF_i)=\ch_{k_i}(\CF_i) \mod H^{>N}(\FM\times X^n)\,.\]

Since the Chern character induces an isomorphism between $K$-theory and Chow for $X^n$ (with $\BQ$-coefficients), we can pick a class $\delta\in K^\ast(X^n)_\BQ$ so that
\[\ch(\delta)=(\gamma_1\otimes\ldots \otimes \gamma_n)\td(T_{X^n})^{-1}\,.\]

Finally, we claim that
\[V=p_\ast\big(\widetilde \CF_1\otimes \ldots\otimes \widetilde \CF_n \otimes q^\ast \delta\big)\]
does the job. Indeed, by Grothendieck--Riemann--Roch and the definition of $\widetilde \CF_i$, we have
\[\ch(V)=p_\ast(\ch_{k_1}(\CF_1)\ldots \ch_{k_n}(\CF_n)q^\ast (\gamma_1\otimes\ldots \otimes \gamma_n)) \mod H^{>N'}(\FM)\,.\]

 The second part of the statement follows from Remark \ref{rmk: hodgeconj}. \qedhere
\end{proof}

We will see in Appendix \ref{appendix} that a better version of the previous proposition can be achieved if we replace algebraic $K$-theory by topological $K$-theory.
%\begin{remark}\label{rmk: blanccohomology}
%If one replaces algebraic $K$-theory by Blanc's topological $K$-theory (cf. Remark \ref{rmk: blanc}) then the same argument works for any $D\in H^\ast_{\textup{taut}}(\FM)$ since the Chern character defines an isomorphism between $K^\ast_{\textup{top}}(X^n)$ and $H^\ast(X^n)$. We expect that equality in Blanc's $K$-theory actually implies $A=B$.
%\end{remark}

\appendix

\section{$\sE$-Hall algebra for additive invariant $\sE$}
\label{appendix}

In this section we will explain how the material in the present paper generalizes if we replace algebraic $K$-theory by other additive invariants, such as Hochschild homology or Blanc's topological $K$-theory.

\subsection{Additive invariants}

Let $R$ be a ring and let $\mathsf{Mod}_R$ be the category of $R$-modules. An additive invariant with values in $\mathsf{Mod}_R$, in the sense of Tabuada \cite{tabuadaadditifs}, is a functor $\sE\colon \mathsf{dgcat}\to \mathsf{Mod}_R$ from the category of ($\BC$-linear) dg categories to the category of $R$-modules which satisfies 2 properties: the image by $\sE$ of a Morita equivalence of dg categories is an isomorphism, and if $\mathcal C=\langle \mathcal C_1, \mathcal C_2\rangle$ is a semi-orthogonal decomposition then $\sE(\CC_1)\oplus \sE(\CC_2)\xrightarrow{\sim} \sE(\CC)$. The most basic example of an additive invariant is $K_0$ (with $R=\BZ$).

By \cite[Théorème 6.1.]{tabuadaadditifs}, $\sE$ being additive is equivalent to factoring through
\[\mathsf{dgcat}\to \mathsf{Hmo}_0\to \mathsf{Mod}_R\]
where $\mathsf{Hmo}_0$ is a category whose objects are the same as $\mathsf{dgcat}$ and 
\[\Hom_{\mathsf{Hmo}_0}(\CC, \CD)=K_0\big(\mathrm{rep}(\CC, \CD)\big)\,.\]
This means that if $F\colon \CC\to \CD$ is a dg functor then $\sE(F)$ is the image of $[F]\in K_0\big(\mathrm{Fun}^{\textup{ex}}(\CC, \CD)\big)$ through the composition
\[K_0\big(\mathrm{Fun}^{\textup{ex}}(\CC, \CD)\big)\to K_0\big(\mathrm{rep}(\CC, \CD)\big)\to \mathrm{Hom}_{R}(\sE(\CC), \sE(\CD))\,,\]
and in particular if $0\to F_1\to F_2\to F_3\to 0$ is an exact sequence of dg functors $\CC\to \CD$ then $\sE(F_2)=\sE(F_1)+\sE(F_3)$.

Any additive invariant is a module over $K_0$, in the sense that for any two dg categories $\CC$ and $\CD$ we have a canonical morphism of $R$-modules
    \begin{equation}
        \label{eq: moduleoverk0}
    K_0(\mathcal C)\otimes_\BZ \sE(\mathcal D)\to \sE(\mathcal C\otimes \mathcal D)\,.\end{equation}
Indeed, this morphism is obtained by adjunction from
    \begin{align*}
        K_0(\CC)=K_0(\textup{Fun}^{\textup{ex}}({\bf 1} ,\CC))\xrightarrow{\otimes \id_{\CD}} K_0(\textup{Fun}^{\textup{ex}}(\CD, \CC\otimes \CD))\to \Hom_{R}(\sE(\CD), \sE(\CC\otimes \CD))\,. 
    \end{align*}

We say that $\sE$ is lax monoidal if it comes together with $R$-module homomorphisms
\begin{align*}R&\to \sE(\Perf(\pt))\\
\sE(\CC)\otimes_R \sE(\CD)&\to E(\CC\otimes \CD)
\end{align*}
satisfying some natural compatibilities, and we say that $\sE$ is (strongly) monoidal if these homomorphisms are isomorphisms.

\begin{example}
Hochschild homology $HH(\mathcal \CC)=\bigoplus_{i\in \BZ}HH_i(\mathcal C)$ is a strongly monoidal additive invariant with values in $\BC$ vector spaces. Other variations such as cyclic homology or periodic cyclic homology are also monoidal additive invariants. On the other hand, $HH_0(-)$ is lax monoidal but not strongly monoidal.
\end{example}

\subsection{$\sE$-Hall algebra and $\varepsilon$ invariants} Let $\sE$ be an additive invariant which we will further assume that commutes with infinite direct sums (this is required for the analog of Proposition \ref{prop: Ktheoryrig} to hold for $\sE$). 

We define for a derived stack $\FX$
\[\sE^\ast(\FX)=\sE(\Perf(\FX))\,.\]
Just as $K^\ast$, $\sE^\ast$ admits arbitrary pullbacks and pushforwards along maps of stacks whose pushforward preserves perfect complexes. There is a tensor product map
\begin{equation}\label{eq: E*moduleoverK*}
K^\ast(\FX)\otimes_\BZ \sE^\ast(\FX)\xrightarrow{\eqref{eq: moduleoverk0}} \sE(\Perf(\FX)\otimes \Perf(\FX))\xrightarrow{\otimes} \sE(\Perf(\FX))=\sE^\ast(\FX)\,.\end{equation}
When $\sE$ is lax monoidal there is also a tensor product map $\sE^\ast(\FX)\otimes_R \sE^\ast(\FX)\to \sE^\ast(\FX)$. 

We may define the $\sE$-homology group $\sE_\ast(\FX)$ in complete analogy with Definition \ref{def: K-hom}, with the caveat that there are two possible slightly different variations: if $\sE$ is a general additive invariant, an element of $\sE_\ast(\FX)$ is defined to be a collection of functionals $$\{\phi_S\colon \sE^\ast(\FX\times S)\to \sE^\ast(S)\}$$ which are $K^\ast(S)\otimes_\BZ R$-linear, with the same compatibility requirement. If $\sE$ is lax monoidal it seems more natural to require the functionals to be $\sE^\ast(S)$-linear, or otherwise $\sE_\ast(\FX)$ becomes potentially very large. Either option provides a reasonable theory, so in what follows we will leave ambiguous which option we take. If $\sE$ is strongly monoidal and we require $\sE^\ast(S)$-linearity then 
\[\sE_\ast(\FX)\simeq \Hom_R(\sE^\ast(\FX), R)\,\]
is just a space of functionals.

We have analogs of all the functoriality properties of Section \ref{subsec: K-hom}. For example, for any additive invariant we have a cap product 
\[\cap \colon K^\ast(\FX)\otimes \sE_\ast(\FX)\to \sE_\ast(\FX)\]
obtained by dualizing \eqref{eq: E*moduleoverK*}. The analog of \eqref{eq: GtoKmorphism} is still a map $G(\FX)\to \sE_\ast(\FX)$. Indeed, \eqref{eq: GtoKmorphism} is obtained from Lemma \ref{lem: gtokhom}, which gives a functor $D^b\Coh(\FX)\otimes \Perf(\FX\times S)\to \Perf(S)$,
and we can use it together with \eqref{eq: moduleoverk0} to obtain
\begin{align*}G(\FX)\otimes \sE^\ast(\FX\times S)=K_0(D^b\Coh(\FX))\otimes \sE(\Perf(\FX\times S)) 
&\to \sE(D^b\Coh(\FX)\otimes \Perf(\FX\times S))\\
&\to \sE(\Perf(S))=\sE^\ast(S)\,.
\end{align*}
In particular, we get fundamental classes $[\FX]_{\sE}\in \sE_\ast(\FX)$ for stacks as in Theorem \ref{thm: deltainvariants}.

With this in place, all the material from Section \ref{sec: Khallinvariants} is adapted to additive invariants $\sE$ in a straightforward way, and we obtain a $\sE$-Hall algebra
\[\mathbb E(\CA)\coloneqq \sE_\ast(\FM_{\CA}^\rig)\]
and invariants
\[\delta_{\alpha}^{\mu, \sE}, \varepsilon_{\alpha}^{\mu, \sE}\in \mathbb E(\CA)\,.\]
For $\varepsilon$ invariants we need to either assume that $R$ contains $\BQ$ or extend to a ring containing $\BQ$ as in Definition~\ref{def: KHall}.

\subsection{Non-abelian localization}
The last ingredient in our setup is the non-abelian localization formula for additive invariants, which will soon appear in a revision of \cite{HLNAL}. Once we have it, it follows by the same reasoning that the invariants $\delta^{\mu, \sE}_\alpha, \varepsilon^{\mu, \sE}_\alpha$ satisfy the exact same wall-crossing formulas \eqref{eq: deltawc} and \eqref{eq: epsilonwc}. We explain the statement here.

To imitate the notation used in $K$-theory we denote by
\[\chi_\sE(\FX, -)\colon \sE^\ast(\FX)\to \sE^\ast(\pt)\,\]
the map induced by the pushforward morphism $\Perf(\FX)\to \Perf(\pt)=D^b\Coh(\pt)$, where 
$\FX$ is a stack satisfying the hypothesis of Theorem~\ref{thm: NAloc}.

\begin{theorem}[{\cite{HLNAL}}]Let $\FX$ be as in Theorem \ref{thm: NAloc} and let $\sE$ be an additive invariant. Then, for any $V \in \sE^\ast(\FX)$, the following formula holds:
\begin{equation}\label{eq: NAlocE}\chi_\sE(\FX, V) = \chi_\sE(\mathcal \FX^{\textup{ss}}, V|_{\mathcal \FX^{ss}}) + \sum_{c\in \Gamma\setminus \{0\}} \chi_\sE(\FZ_{c}, V|_{\FZ_{c}} \otimes E_{c})\in \sE^\ast(\pt)\,.\end{equation}
\end{theorem}
%\begin{proof}
 %   We follow closely \cite{HLNAL}, and in particular use the category $\mathfrak C_\ast(\FX)^{<\infty}$ defined in Definition 2.7 of loc. cit., which we remind the reader is a module over $\Perf(\FX)$ (and could be thought of as a categorification of $K_\ast(\FX)$). Let $\FZ$ be the union of all centers $\bigcup_{c\in \Gamma}\FZ_c$, let $\Sigma\colon \FZ\to \FX$ and let $E$ be the element of $\mathfrak C_\ast(\FZ)^{<\infty}$ whose restriction to $\FZ_c$ is $E_c$; note that $E$ is defined as an element of $K$-theory, but we consider an arbitrary lift to an actual element of $\mathfrak C_\ast(\FZ)^{<\infty}$. The left and right hand sides of \eqref{eq: NAlocE} are obtained by applying $\sE(-)$ to the two functors
%    \begin{center}
 %       \begin{tikzcd}
 %           \Perf(\FX)\arrow[r, bend right, swap, "(-)\otimes \CO_\FX"] \arrow[r, bend left, "(-)\otimes \Sigma_\ast E"] &
  %          \mathfrak{C}_\ast(\FX)^{<\infty}\arrow[r,"\Gamma(\FX{,}-)"]& [0.5cm]\Perf(\pt)\,.
%        \end{tikzcd}
%    \end{center}
%By \cite[Theorem 3.15, Corollary 3.17]{HLNAL}, under our hypothesis $\CO_\FX$ and $\Sigma_\ast E$ are both elements of $\mathfrak{C}_\ast(\FX)^{<\infty}$ and they are equal in $K_0(\mathfrak{C}_\ast(\FX)^{<\infty})$. Thus, the two functors considered above are the same in $K_0(\textup{Fun}^{\textup{ex}}(\Perf(\FX), \Perf(\pt))$, and thus induce the same map of $R$-modules after applying $\sE$, by the definition of additive invariant.
%\end{proof}

\subsection{Topological $K$-theory versus cohomology}

In  \cite{B16}, Blanc defines topological $K$-theory of a dg category as a spectrum which, by taking homotopy groups, produces abelian groups $K^{\mathrm{top}}_i(\CC)$ satisfying Bott periodicity $K_{i+2}^{\mathrm{top}}(\CC)\simeq K_i^{\mathrm{top}}(\CC)$. Then
\[K^{\mathrm{top}}(\CC)=K^{\mathrm{top}}_0(\CC)\oplus K^{\mathrm{top}}_1(\CC)\]
is another example of a lax monoidal additive invariant (with $R=\BZ$). Let us denote by $K^\ast_{\textup{top}}(\FX), K_\ast^{\textup{top}}(\FX)$ the topological $K$-theory of $\Perf(\FX)$ and topological $K$-homology of $\FX$. 

Topological $K$-theory admits a canonical map from algebraic $K$-theory $K_0(-)\to K_0^{\textup{top}}(-)$ and the Chern character map factors through
\[\ch\colon K^\ast_{\textup{top}}(\FX)\to H^\ast(\FX)\,.\]
If $X$ is a separated finite type scheme over $\BC$ then the Chern character map is an isomorphism after tensoring with $\BQ$ \cite[Theorem 1.1.b]{B16}. Results in a similar direction for (global quotient) stacks can be found in \cite{HLPomerleano}. 

The entire content of Section \ref{sec: Ktocoh} can be adapted to topological $K$-theory. In particular, we can talk about an element of $H_\ast(\FM^\rig_\CA)$ being a homological lift of the $\varepsilon$ class
\[\varepsilon_{\alpha}^{\mu, K_{\textup{top}}}\in K_\ast^{\textup{top}}(\FM_\CA^\rig)\eqqcolon \BK_{\textup{top}}(\CA)\,.\]
Since topological $K$-theory can ``see'' non-algebraic cohomology classes we get the following improvement of Proposition \ref{prop: injectivityhomlift}:

\begin{proposition}\label{prop: injectivityhomlifttop}
    Let $\FM$ be a moduli stack of complexes on a smooth projective variety $X$. Suppose that $A, B\in H_\ast(\FM)$ are homological lifts of the same class in $K_\ast^{\textup{top}}(\FM)$. Then $A$ and $B$ agree on tautological classes, i.e. 
    \[\int_A D=\int_B D\]
    for any $D\in H^\ast_{\textup{taut}}(\FM)$.
\end{proposition}
\begin{proof}
   The proof goes exactly as in Proposition \ref{prop: injectivityhomlift}, but we further use the fact that $\ch\colon K^\ast_{\textup{top}}(X^n)_\BQ\to H^\ast(X^n)$ is an isomorphism.
\end{proof}

We expect that having the same homological lift actually implies that the two classes in homology agree. Indeed, this would be a consequence of the following natural looking conjecture:

\begin{conjecture}
For any (reasonable) stack $\FX$ and any $N\geq 0$ the truncated Chern character map
\[\ch_{\leq N}\colon K^\ast_{\textup{top}}(\FX)_\BQ\to H^\ast(\FX)\twoheadrightarrow H^{\leq N}(\FX)\]
is surjective.
\end{conjecture}

\end{document}